\documentclass[11pt]{amsart}

\usepackage{mathptmx}
\usepackage{helvet}
\usepackage{courier}
\usepackage{graphicx}
\usepackage{multicol}

\usepackage[T1]{fontenc}
\usepackage{euscript}
\usepackage{amsmath}
\usepackage{amsthm}
\usepackage{amssymb}
\usepackage{amscd}
\usepackage{epic}
\usepackage{enumerate}
\usepackage{array}
\usepackage{tabularx}
\usepackage[T1]{fontenc}
\usepackage{calrsfs}
\DeclareMathAlphabet{\pazocal}{OMS}{zplm}{m}{n}

\usepackage{amsfonts,amsthm,amsmath}
\usepackage{amssymb}

\numberwithin{equation}{section}
\numberwithin{equation}{subsection}

\theoremstyle{plain}

\newtheorem{theorem}[equation]{Theorem}
\newtheorem{lemma}[equation]{Lemma}
\newtheorem{proposition}[equation]{Proposition}

\newtheorem{corollary}[equation]{Corollary}

\newtheorem{question}[equation]{Question}

\theoremstyle{definition}

\newtheorem{example}[equation]{Example}
\newtheorem{remark}[equation]{Remark}

\newtheorem{definition}[equation]{Definition}

\newtheorem{bekezdes}[equation]{}

\newcommand{\et}{\EuScript{T}}

\newcommand{\bC}{{\mathbb C}}\newcommand{\bc}{{\mathbb C}}

\newcommand{\bZ}{{\mathbb Z}}

\newcommand{\calC}{{\mathcal C}}

\newcommand{\cV}{{\mathcal V}}

\newcommand{\cS}{{\mathcal S}}\newcommand{\calS}{{\mathcal S}}

\newcommand{\calO}{{\mathcal O}}
\newcommand{\cF}{{\mathcal F}}
\newcommand{\cH}{{\mathcal H}}

\newcommand{\calT}{\mathcal{T}}
\newcommand{\calL}{{\mathcal L}}
\newcommand{\frR}{\mathfrak{R}}

\newcommand{\frc}{\mathfrak{c}}
\newcommand{\bPE}{{\bf PE}}

\newcommand{\Lb}{\pazocal{L}}

\newcommand{\calQ}{{\mathcal Q}}

\newcommand{\frm}{\mathfrak{m}}

\newcommand{\C}{{\calc}}

\newcommand{\rank}{{\rm rank}\, }

\newcommand{\frakv}{\mathfrak{v}}

\newcommand{\bt}{{\bf t}}

\newcommand{\bH}{{\mathbb H}}

\newcommand{\calF}{{\mathcal F}}

\newcommand{\setP}{\mathbb{P}}

\newcommand{\hh}{\mathfrak{h}}
\newcommand{\pp}{\mathfrak{p}}

\newcommand{\calv}{\mathcal{V}}

\newcommand{\Z}{\mathbb{Z}}

\newcommand{\R}{\mathbb{R}}

\newcommand{\bS}{{\mathbb S}}

 \newcommand{\frX}{\mathfrak{X}}

\newcommand{\calo}{{\mathcal O}}
\newcommand{\calt}{{\mathcal T}}

\def\C{\mathbb C}

\def\R{\mathbb R}
\def\bH{\mathbb H}

\def\Z{\mathbb Z}

\newcommand{\Cc}{A^{-}}
\DeclareMathOperator{\HFL}{HFL}
\DeclareMathOperator{\gr}{gr}
\newcommand{\BF}{\mathbb{F}}

\author{Andr\'as N\'emethi}
\thanks{The author is partially supported by  ``\'Elvonal (Frontier)'' Grant KKP 144148}
\address{Alfr\'ed R\'enyi Institute of Math.,
Re\'altanoda utca 13-15, H-1053, Budapest, Hungary \newline
 \hspace*{3mm} ELTE - Univ. of Budapest, Dept. of Geo.,
 P\'azm\'any P\'eter s\'et\'any 1/A, 1117, Budapest, Hungary \newline \hspace*{3mm}
  BBU - Babe\c{s}-Bolyai Univ., Str, M. Kog\u{a}lniceanu 1, 400084 Cluj-Napoca, Romania
 \newline \hspace*{3mm}
BCAM - Basque Center for Applied Math.,
Mazarredo, 14 E48009 Bilbao, Basque Country, Spain}
\email{nemethi.andras@renyi.hu }

\title{Filtered lattice homology of curve singularities} 

\begin{document}


\begin{abstract}
Let $(C,o)$ be a complex analytic isolated curve singularity of arbitrary large embedded dimension.
Its lattice cohomology  $\bH^*=\oplus_{q\geq 0}\bH^q$ was introduced in \cite{AgostonNemethi}, each
$\bH^q$ is a graded $\Z[U]$--module. Here we study its homological version $\bH_*(C,o)=\oplus_{q\geq 0}\bH_q$.
The construction uses the multivariable Hilbert function
associated with the valuations provided by the normalization of the curve.
A key intermediate product is  a tower of spaces
$\{S_n\}_{n\in \Z}$ such that $\bH_q=\oplus_n H_q(S_n,\Z)$.

In this article for every $n$ we consider a natural filtration of the space $S_n$, which provides
a homological spectral sequence converging to the homogeneous summand $H_q(S_n,\Z)$ of the lattice homology.
All the entries of all the pages of the spectral sequences are new invariants of $(C,o)$.
We show how the collection of the  first pages is equivalent with
the motivic Poincar\'e series of $(C,o)$.

We provide several concrete computations of the corresponding multivariable Poincar\'e series associated with the entries of the spectral sequences.

In the case of plane curve singularities, the first page can also be identified with the Heegaard Floer Link homology of the
link of the singularity. In this way, the new invariants provide   for an arbitrary (non necessarily plane) singularity a
homological theory which
is the analogue of the  Heegaard Floer Link theory for links of plane curve singularities.
\end{abstract}

\maketitle

\linespread{1.2}


\pagestyle{myheadings} \markboth{{\normalsize   A. N\'emethi}} {{\normalsize Filtered lattice homology of curve singularities}}

\section{Introduction}

\subsection{}
Let $(C,o)$ be a complex  analytic isolated singular germ of dimension one with $r$ irreducible components. In \cite{AgostonNemethi}
the authors associated with such a curve singularity the {\it analytic lattice cohomology} $\bH^*(C,o)$.
In this note we will consider the homological version, the analytic lattice homology $\bH_*(C,o)$,
with a very similar construction. The definition is based on the construction of a tower of spaces
$\emptyset=S_{m_w-1}\subset S_{m_w}\subset S_{m_w+1}\subset \cdots \subset S_n \subset \cdots$. Then the lattice homology has the form
$\bH_*(C,o)=\oplus_{q\geq 0}\bH_q(C,o)$, where each  $\bH_q(C,o):=\oplus_n H_q(S_n,\Z)$ is a $\Z$--graded $\Z$--module.
Its homogeneous elements $(\bH_q)_{-2n}$ consists of the elements of the summand $H_q(S_n,\Z)$. Moreover, $\bH_q(C,o)$ admits a homogeneous $U$--action
of degree $-2$, $H_q(S_{n-1},\Z)\to H_q(S_n,\Z)$, the homological morphism induced by the inclusion $S_{n-1}\hookrightarrow S_n$.
In this way, each $\bH_q(C,o)$ becomes a graded $\Z[U]$--module. The Euler characteristic of $\bH_*(C,o)$ is the {\it delta invariant} $\delta(C,o)$ of
$(C,o)$.

The construction uses the cubical decomposition of $\frX=(\R_{\geq 0})^r$ determined by  the lattice points $(\Z_{\geq 0})^r\subset
(\R_{\geq 0})^r$ and the canonical bases $\{E_i\}_{i=1}^r$ of $\Z^r$. Moreover, one also requires a weight function $w:(\Z_{\geq 0})^r\to\Z$. In the present case
 it is given by $w(l)=2\hh(l)-|l|$, where $l\mapsto \hh(l)$ is the Hilbert function of $(C,o)$ associated with
valuations given by the normalization map, and $|l|=|(l_1,\ldots , l_r)|=\sum_il_i$ (cf. section 3).

In this note, for any $n$  we consider an increasing filtration $\{S_n \cap \frX_{-d}\}_{d\geq 0}$ of the  space $S_n$.
It is canonically associated with the normalization map of $(C,o)$, therefore, any output of the filtration is a
well--defined  invariant of the singularity $(C,o)$.
Note that in the definition of the lattice homology it is enough to know the homotopy type of the tower of spaces $\{S_n\}_n$.
However, in the  filtration we use a more subtle information:   the embedding of the finite cubical complex $S_n$
into $(\R_{\geq 0})^r$. The filtration $\{S_n\cap \frX_{-d}\}_d$ is induced by the filtration
$\{\frX_{-d}\}_d$ of $(\R_{\geq 0})^r$, where $\frX_{-d}$ is the union of those cubes whose left-lower vertex  $l$ satisfies
$|l|\geq d$ (see section \ref{s:levfiltr}).
 By this construction we wish to emphasize once more the importance of the tower $\{S_n\}_n$
and to point out the structural  subtleties and riches of these (embedded) cubical complexes.

The filtration induces a homological spectral sequence $(E^k_{-d,q})_n\Rightarrow (E^\infty_{-d,q})_n$
 for every $n$, which stabilizes after finitely many steps.
 Its terms are the following:
  \begin{equation*}\begin{split}
  (E^1_{-d,q})_n=& H_{-d+q}(S_n\cap \frX_{-d}, S_n\cap \frX_{-d-1},\Z),\\
   (E^\infty_{-d,q})_n=& \frac{(F_{-d}\, \bH_{-d+q}(\frX))_{-2n}}
   { (F_{-d-1}\, \bH_{-d+q}(\frX))_{-2n}}=({\rm Gr}^F_{-d}\, \bH_{-d+q}(\frX)\,)_{-2n}.
   \end{split}\end{equation*}

In particular, all the information coded by the  pages of these spectral sequences are   invariants of $(C,o)$.
In this way, for any $k\geq 1$ we associate with $(C,o)$
the  triple graded $\Z$--module $(E^k_{-d,q})_n$.
We refer to the filtration given by $n$ as the `weight filtration',
given by $d$ as the `level filtration', and given by $b=-d+q$ as the `homological filtration'.

For an additional $\Z[Y_1,\ldots, Y_r]$--module
structure see paragraph \ref{bek:Y} and section \ref{s:Uoperators}.

Moreover, as further new invariants,
 we can consider those minimal values $k$ for which  $(E^k_{*,*})_n=(E^\infty_{*,*})_n$, see \ref{ss:ss},  (this is the analogue
of the $\tau$--invariant of Ozsv\'ath and Szab\'o \cite{OSztau}, or of the $s$--invariant of Rasmussen \cite{ras_s} in the context of
Heegaard Floer Link and Khovanov theories).

We show via examples that this minimal  $k$ for which $(E^k_{*,*})_n=(E^\infty_{*,*})_n$ for every $n$
can be large. In particular, all the pages of the spectral sequences might contain deep information.

The spectral sequence converges to the lattice homology $\bH_*$:
the $\infty$--pages (collected for all $S_n$)
  provide  ${\rm Gr}^F_*\bH_*$, the `graded lattice homology'.
Already the lattice homology has an interesting rich structure, but its triple graded version  ${\rm Gr}^F_*(\bH_*)_*$
contains considerable  additional information.

\subsection{}
From the spectral sequences we can also extract the  multivariable Poincar\'e series coding the corresponding
ranks of the entries:
 $$PE_k(T,Q,h):=\sum_{d,q,n} \ \rank (E_{-d,q}^k)_{n}\cdot T^dQ^nh^{-d+q}\in\Z[[T,Q]][Q^{-1},h].$$
 Regarding the $PE_\infty$ series we have the following structure theorem (see Proposition  \ref{prop:infty}).
\begin{theorem}\label{prop:infty_intro} \

\noindent (a) $PE_\infty(1, Q,h)=\sum_{n\geq m_w}\, \big(\, \sum_b \, {\rm rank}\, H_b(S_n,\Z)\,h^b\, \big)\cdot Q^n$ is the Poincar\'e series of \ $\bH_*(C,o)$.

\noindent (b) $PE_\infty(1, Q,-1)=\sum_{n\geq m_w}\, \chi_{top}(S_n)\cdot Q^n$ \ (where $\chi_{top}$ denoted the topological Euler characteristic)

\noindent (c) Let $R$ be any  rectangle  $\{x\, :\, 0\leq x\leq c'\}$ with $c'\geq c$, where $c$ is the conductor of $(C,o)$. Then
$$PE_\infty(1, Q,-1)=\frac{1}{1-Q}\cdot \sum_{\square_q\subset R}\, (-1)^q \, Q^{w(\square_q)}.$$
(This shows that this  Euler characteristic type invariant can be computed in two different ways: either homologically
--- \`a la Betti --- or by cell-decomposition --- \`a la Euler.)\\
(d) $$\lim_{Q\to 1}\Big( PE_\infty (1,Q, -1)-\frac{1}{1-Q}\Big)=eu(\bH_*(C,o))=\delta(C,o).$$
(e) For two series $P$ and $P'$ we write
 $P\geq P'$ if  $P-P'$ has all of its coefficients nonnegative. Then
 $$PE_\infty(1, Q,h)\geq PE_\infty(1, Q,h=0)\geq \frac{1}{1-Q}$$  and \
 $PE_\infty(1, Q,h=0)-1/(1-Q)$ is finitely supported.
\end{theorem}
\subsection{} The series $PE_1(T,Q,h)$ can be
 enhanced by a richer multivariable  series. Indeed, we consider for any lattice point $l\in (\Z_{\geq 0})^r$
the shifted first quadrant $\frX_{-l}:= \{x\in (\R_{\geq 0})^s\,:\, x\geq l\}$ and
the set of cubes in it (i.e. those cubes with left-lower vertex $\geq l$).
Then we have the natural direct sum decomposition
\begin{equation*}
 (E^1_{-d,q})_{n}= \bigoplus_{l\in\Z^r_{\geq 0},\, |l|=d}\
 (E^1_{-l,q})_{n}, \ \mbox{where} \ (E^1_{-l,q})_{n}:=
 H_{-|l|+q}(S_n\cap \frX_{-l}, S_n\cap \frX_{-l}\cap \frX_{-|l|-1},\Z)\end{equation*}
 and the corresponding Poincar\'e series
 $${\bPE}_1({\bf T}, Q,h)={\bPE}_1(T_1, \ldots , T_r, Q,h):=\sum_{l\in\Z^r_{\geq 0},\, n,q}\ \rank\big((E^1_{-l,q})_{n}\big)\cdot
 T_1^{l_1}\cdots T_r^{l_r}\, Q^n\, h^{-|l|+q}. $$
 It satisfies   ${\bPE}_1(T_1=T,\ldots, T_r=T, Q,h)=PE_1(T,Q,h)$.

Rather surprisingly, the series $\bPE_1({\bf T}, Q,h)$ is related with another analytic invariant, defined in a very different
way and context in \cite{cdg3}, the {\it motivic Poincar\'e series}
 $P^m({\bf t}, q)$ of the singularity. In Theorem 5.1.3 we show that
$\bPE_1({\bf T}, Q,h)$ and  $P^m({\bf t}, q)$ determine each other by a very explicit procedure
(see Theorem \ref{th:PP} for a more precise statement and connections with other invariants).
\begin{theorem}\label{th:PP_intro}
(a)
$$\bPE_1({\bf T}, Q,h)|_{T_i\to t_i\sqrt{q},\ Q\to \sqrt{q},\ h\to -\sqrt{q}}= P^m({\bf t}, q).$$

(b) Write  $P^m(\bt,q)$ as $\sum_l\pp^m_l(q)\cdot\bt^l$.  Then each $\pp^m_l(q)$
can be written in a unique way in  the form
$$\pp^m_l(q)=\sum_{k\in\Z_{\geq 0} }\pp^m _{l,k}q^{k+\hh(l)}, \ \ (\pp^m_{l,k}\in\Z).$$

(c) Write $P^m(\bt,q)=\sum_l\ \sum_{k\in\Z_{\geq 0} }\pp^m _{l,k}q^{k+\hh(l)}\cdot \bt^l$. Then
$$\bPE_1({\bf T}, Q,h)= \sum_l\ \sum_{k\in\Z_{\geq 0} }\pp^m _{l,k}\ {\bf T}^l Q^{w(l)}\cdot (-Qh)^k.$$

(d)\ $\bPE_1$ is a rational function of type
$$\overline{\bPE_1}({\bf T}, Q, h)/\prod_i(1-T_iQ), \ \ \mbox{where} \ \ \overline{\bPE_1}({\bf T}, Q, h)\in\Z[{\bf T}, Q, Q^{-1}, h].$$


(e) In the Gorenstein case, after substitution $h=-Q$ one has a  symmetry of type
 $$\overline{\bPE_1}({\bf T}, Q, h=-Q)|_{T_i\to T_i^{-1}}=\prod_i T_i^{-c_i}\cdot \overline{\bPE_1}({\bf T}, Q, h=-Q).$$
\end{theorem}

The identification of $\bPE_1({\bf T}, Q,h)$ with $P^m({\bf t}, q)$ is
based on the following (non)vanishing result:
$$\mbox{if $H_{b}(S_n\cap \frX_{-l}, S_n\cap \frX_{-l}\cap\frX_{-|l|-1})\not=0$ then necessarily
$n=w(l)+b$.}$$
This (non)vanishing follows from another identification of the modules
$H_{b}(S_n\cap \frX_{-l}, S_n\cap \frX_{-l}\cap\frX_{-|l|-1})$: in \cite{GorNem2015}
they appear under the name `local lattice cohomology', and they were studied via
 certain hyperplane arrangements and their Orlik--Solomon algebras (see here subsection 3.3).
In fact, the present work can be considered as a completion of \cite{GorNem2015}
by the  construction in the algebraic case of the object
whose `localization' is the (co)homology studied in \cite{GorNem2015}.

It is worth to mention that
 the article \cite{GorNem2015} introduced in the topological context the analogue
 of the Hilbert function of curve singularities, called the $H$--function,
 which generated a rather intense activity in topology, see e.g.
 \cite{GH,GLM,GLLM,BLZ} and the references therein.
We definitely expect that all the results of the present note can be extended to the purely topological context, if we replace the Hilbert function of algebraic singularities with the $H$--function of links in $S^3$ (or, in $L$--spaces), hence the (filtered)
tower of spaces
$\{S_n\}_n$ considered in this note with the corresponding tower associated with the topological $H$--function and the weight function $w(l)=2H(l)-|l|$.

\subsection{}\label{bek:Y} {\bf The $\Z[Y_1,\ldots, Y_r]$--module structure.} \
Along the spectral sequence the $U$--action (induced by the inclusions $S_{n-1}\hookrightarrow S_n$)
is trivial. However, we introduce new homological operators $Y_1,\ldots, Y_r$
induced by the lattice shifts $x\mapsto x+E_i$. Each $Y_i$
 is a spectral sequence morphism connecting  the spectral sequences of $S_n$ and  $S_{n+1}$:
  $$Y_i: (E^k_{-d, q})_{n} \to (E^k_{-d-1, q+1})_{n+1}.$$
 This new action usually is non--trivial;
 in fact, even on the $\infty$--pages (i.e. on  ${\rm Gr}^F_*\bH_*$)
the  information carried by them
  is more subtle than the information coded by $U$--action acting on  $\bH_*$.

\subsection{}
In the body of the article we discuss several key examples and  several families of germs. E.g.,
we analyse with several details the case of irreducible germs and decomposable singularities (as two extremal cases).
We also stress the additional properties valid for  plane curve singularities.
\subsection{}
The reader familiar with the {\bf Heegaard Floer Link theory} will realize strong similarities with the structure of  that theory and the
results of the present note. However, we wish to emphasise that the two setups are different. In the HFL theory one
constructs topological invariants associated with the embedded topological type of a link  $L$ of the three sphere  $S^3$,
while in our case we construct analytic invariants using the rigid analytic structure of a {\it not-necessarily plane curve singularity}.

{\bf However,  in the case of plane curve singularities the two theories meet!}
Indeed, for such germs, it turns out that the Hilbert function (which is used to define the
analytic lattice cohomology) can be recovered from the embedded topological type of the link.
In particular, in these cases, all the invariants what we construct (lattice homology, the spaces $S_n$,
their filtrations, the spectral sequences) are invariants of the embedded topological type.
E.g., Theorem \ref{th:PP_intro} gets an additional  new colour: $P^m(\bt,q)=P(\bt)$ can be identified with the multivariable Alexander
polynomial of the link, cf. Theorem \ref{Poincare vs Alexander}.

More surprisingly, the first page of the collection of spectral sequences
can be identified with the Heegaard Floer Link homologies HFL$_*^-(L)=\sum_{l\in (\Z^r_{\geq 0})}
{\rm HFL}_*^-(L,l)$ of the corresponding link $L\subset S^3$.
Nevertheless, by the present theory we create an extra `partition' of  HFL$_*^-(L)$ given by $n$.
The connecting bridge is the following.
\begin{equation}\label{eq:lb2_intro}
\HFL^-_{-n-|l|}(L,l)\simeq H_b(S_n\cap\frX_{-l}, S_n\cap\frX_{-l}\cap \frX_{-|l|-1}),
\end{equation}
where $b=n-w(l)$. For any fixed $d$, summation over $\{l\,:\, |l|=d\}$ gives for any $n,\, b$ and $d$:
\begin{equation}\label{eq:lb3_intro}
\bigoplus_{|l|=d,\ w(l)=n-b}
\HFL^-_{-n-|l|}(L,l)\simeq \bigoplus_{|l|=d}\,
H_b(S_n\cap\frX_{-l}, S_n\cap\frX_{-l}\cap \frX_{-|l|-1})= (E^1_{-d, b+d})_n.
\end{equation}
In particular, for any fixed $n$, the spectral sequence associated with $S_n$
uses and capture only the specially chosen summand $\oplus_{l}{\rm HFL}^-_{-n-|l|}(L,l)$ of
$\oplus _l{\rm HFL^-}_*(L,l)$. This is a partition of $\oplus _l{\rm HFL^-}_*(L,l)$ indexed by $n$.
Each $\oplus_{l}{\rm HFL}^-_{-n-|l|}(L,l)$, interpreted as an $E^1$--term, converges to
$({\rm Gr}^F_*\bH_*)_{-2n}$. (This is not the spectral sequence of the  Heegaard Floer Link theory,
which converges to $HF^-(S^3)$, though the {\it entries}   --- but not
the differentials --- of the first pages can be identified).

In this  identification we use again  \cite{GorNem2015}.

In this way, our theory in fact provides for any analytic  (not necessarily plane) curve singularity
the analogue of HFL$^-_*$ (where this later one
 is constructed  only for plane germs, or for links in $S^3$). In the case of  higher embedded dimension,
the information coded in the embedded topological type needed  for the definition of HFL$^-_*$ is replaced by
information read from the analytic rigidity of the analytic germ.
(However, nota that the two constructions are very-very different.)


\subsection{} {\bf Enhances grading supported by the graded root.} \
All the constructions have an additional subtlety. Recall that the core of all the constructions is the tower of spaces $\{S_n\}_{n\geq m_w}$. Now, each $S_n$ can be considered as the disjoint union of its connected components $\sqcup_v S_n^v$. They are indexed by the vertices $\cV(\mathfrak{R})$ of the graded root
$\mathfrak{R}(C,o)$, see \ref{ss:grroot}. On the other hand, for each $n$, the filtration
$\{S_n\cap \frX_{-d}\}_d$ also decomposes into a disjoint union
$\sqcup_{v\in\cV(\mathfrak{R}),\, w_0(v)=n}\{S_n^v\cap \frX_{-d}\}_d$. In particular, all the spectral sequence (and all its outputs) split according to this decomposition indexed by $\cV(\mathfrak{R})$. That is, all the invariants which originally were graded by $\{n\in\Z\,:\, n\geq m_w\}$ will have a refined grading given by the
vertices of the graded root.

In particular, the graded root $\mathfrak{R}$ supports 
the following invariants partitioned as decorations of the  vertices $\cV(\mathfrak{R})$:
$\bH_*(C,o)$, $PE_k(T,Q,h)$, ${\bf PE}_1({\bf T},Q,h)$, the actions $Y_i$ (guided by the edges of
  $\mathfrak{R}$), and $\oplus_l{\rm HFL}^-_*(L, l)$.
  For details see subsections \ref{ss:grroot}, \ref{ss:decroot}, or  equations (\ref{eq:yi2}) and   (\ref{eq:v}).

For another case in the literature when some important invariants (series)
 appear as decorations of a graded root  see \cite{AJK}.

\subsection{} {\bf The structure of the article is the following.}
Section 2 reviews the general definition of the lattice homology. Section 3 provides the definition of the lattice homology associated with
isolated curve singularities. In section 4 we also provide several invariants of the singular germ (Hilbert function, semigroup,
motivic Poincar\'e series, local lattice homology of \cite{GorNem2015}) and we discuss several relationships connecting them.
The case of plane curves is also highlighted. In section 5 we introduce and discuss the (level) filtration and the
corresponding  spectral sequences associated with any fixed space  $S_n$. Section 6 contains an improvement of this level filtration:
the page $E^1_{*,*}$ carries a lattice filtration which makes the invariants more subtle, and the
corresponding multivariable  Poincar\'e series. In section 7 we introduce the operators $Y_1,\ldots, Y_r$.
Section 8 treats plane curve singularities and we
compare our theory with the Heegaard Floer Link theory.

In section 9 we suggest some other possibilities as level filtrations, which work equally well, and should be developed in the future.

\section{The definition of the lattice homology}\label{ss:latweight}

 \subsection{The lattice homology associated with  a system of weights} \cite{Nlattice}

\bekezdes
 We consider a free $\Z$--module, with a fixed basis
$\{E_i\}_{i\in\calv}$, denoted by $\Z^s$. It is also convenient to fix
a total ordering of the index set $\calv$, which in the
sequel will be denoted by $\{1,\ldots,s\}$.

The lattice homology construction associates
 a graded $\Z[U]$--module with the
pair $(\Z^s, \{E_i\}_i)$ and a set of weights.
The construction follows closely the construction of the lattice cohomology developed in \cite{Nlattice} (for more see also \cite{NOSz,NGr,Nkonyv}).
In particular, we will not prove all the statements, the corresponding modifications are rather natural.

\bekezdes\label{9zu1} {\bf $\Z[U]$--modules.}  We will modify the usual grading of the polynomial
ring $\Z[U]$ in such a way that the new degree of $U$ is $-2$.
Besides $\calt^-_0:=\Z[U]$, considered as a graded $\Z[U]$--module,  we will consider the modules
$\calt_0(n):=\Z[U]/(U^n)$ too with the induced grading.
 Hence, $\calt_0(n)$, as a $\Z$--module, is freely
generated by $1,U^1,\ldots,U^{n-1}$, and has finite
$\Z$--rank $n$.

More generally, for any graded $\Z[U]$--module $P$ with
$d$--homogeneous elements $P_d$, and  for any  $k\in\Z$,   we
denote by $P[k]$ the same module graded in such a way
that $P[k]_{d+k}=P_{d}$. Then set $\calt^-_k:=\calt^-_0[k]$ and
$\calt_k(n):=\calt_0(n)[k]$. Hence, for $m\in \Z$,
$\calt_{-2m}^-=\Z\langle U^{m}, U^{m+1},\ldots\rangle$ as a $\Z$-module.

\bekezdes\label{9complex} {\bf The chain complex.}
$\Z^s\otimes \R$ has a natural cellular decomposition into cubes. The
set of zero-dimensional cubes is provided  by the lattice points
$\Z^s$. Any $l\in \Z^s$ and subset $I\subset \calv$ of
cardinality $q$  defines a $q$-dimensional cube  $(l, I)$, which has its
vertices in the lattice points $(l+\sum_{i\in I'}E_i)_{I'}$, where
$I'$ runs over all subsets of $I$. On each such cube we fix an
orientation. This can be determined, e.g.,  by the order
$(E_{i_1},\ldots, E_{i_q})$, where $i_1<\cdots < i_q$, of the
involved base elements $\{E_i\}_{i\in I}$. The set of oriented
$q$-dimensional cubes defined in this way is denoted by $\calQ_q$
($0\leq q\leq s$).

Let $\calC_q$ be the free $\Z$-module generated by oriented cubes
$\square_q\in\calQ_q$. Clearly, for each $\square_q\in \calQ_q$,
the oriented boundary $\partial \square_q$ (of `classical'  cubical homology) has the form
$\sum_k\varepsilon_k \, \square_{q-1}^k$ for some
$\varepsilon_k\in \{-1,+1\}$.
These $(q-1)$-cubes $\square_{q-1}^k$ are the {\em faces} of $\square_q$.

Clearly,  $H_*(\calC_*, \partial_*)=H_*(\R^s,\Z)$.
In order to define a `non-trivial'  homology theory of these cubes,
 we consider a set of compatible {\em weight
functions} $\{w_q\}_q$.

\begin{definition}\label{9weight}  A set of functions
$w_q:\calQ_q\to \Z$  ($0\leq q\leq s$) is called a {\em set of
compatible weight functions}  if the following hold:

(a) For any integer $k\in\Z$, the set $w_0^{-1}(\,(-\infty,k]\,)$
is finite;

(b) for any $\square_q\in \calQ_q$ and for any of its faces
$\square_{q-1}\in \calQ_{q-1}$ one has $w_q(\square_q)\geq
w_{q-1}(\square_{q-1})$.
\end{definition}

\bekezdes\label{bek:grading}
In the presence of any  fixed set of compatible weight functions
$\{w_q\}_q$  we set  ${\Lb}_q:=\calC_q\otimes_{\Z}\Z[U]$.
Note that $\Lb_q$ is a $\Z[U]$-module by
$U*(U^m\square_q):= U^{m+1}\square$.
Moreover, $\Lb_q$ has a $\Z$--grading: by definition the degree
of $U^m\square$ is ${\rm deg}_{\Lb}(U^m\square):=-2m-2w(\square)$.
 In fact, the grading is $2\Z$--valued; we prefer  this convention in order to keep
the compatibility  with (Link)  Heegaard  Floer theory.

We define $\partial_{w,q}:\Lb_{q}\to \Lb_{q-1}$ as follows.  First write
$\partial\square_{q}=\sum_k\varepsilon_k \square ^k_{q-1}$, then
set
$$\partial_{w,q}(U^m\square_{q}):=U^m\sum_k\,\varepsilon_k\,
U^{w(\square_{q})-w(\square^k_{q-1})}\, \square^k_{q-1}.$$
Then one shows that
$\partial_w\circ\partial_w=0$, i.e.
$(\Lb_*,\partial_{w,*})$ is a chain complex.

\bekezdes Next we define an augmentation of $(\Lb_*,\delta_{w,*})$. Set $m_w:=\min_{l\in \Z^s}w_0(l)$ and define
the $\Z[U]$-linear map
$\epsilon_w:\Lb_0\to  \calt^-_{-2m_w}$ by $\epsilon_w(U^ml)=U^{m+w_0(l)}$,
 for any $l\in \Z^s$ and  $m\geq 0$. Then
 $\epsilon_w$ is surjective,
$\epsilon_w\circ \partial_w=0$, and
 $\epsilon_w$ and $\partial_w$
are  homogeneous morphisms of $\Z[U]$--modules  of degree
zero.

\begin{definition}\label{9def12} The homology of the chain complex
$(\Lb_*,\partial_{w,*})$ is called the {\em lattice homology} of the
pair $(\R^s,w)$, and it is denoted by $\bH_*(\R^s,w)$. The
homology of the augmented chain complex
$$\ldots\stackrel{\partial_w}{\longrightarrow} \Lb_1
\stackrel{\partial_w}{\longrightarrow}
\Lb_0\stackrel{\epsilon_w}{\longrightarrow}\calt^-_{-2m_w}\longrightarrow 0
$$ is called the {\em
reduced lattice homology} of the pair $(\R^s,w)$, and it is
denoted by $\bH_{*,red}(\R^s,w)$.
\end{definition}

If the pair $(\R^s,w)$ is clear
from the context, we omit it from the notation.

For any $q\geq 0$ fixed,
 the $\Z$--grading of $\Lb_q$ induces a $\Z$--grading
on $\bH_q$ and $\bH_{q,red}$; the homogeneous part of degree $d$
is denoted by $(\bH_{q})_{d}$, or $(\bH_{q,red})_{d}$.
Moreover,
both $\bH_q$ and $\bH_{red,q}$ admit an induced graded
$\Z[U]$--module structure and $\bH_q=\bH_{q,red}$ for $q>0$.

\begin{lemma}\label{9lemma3}  One has a graded
$\Z[U]$--module isomorphism
$\bH_0=\calt^-_{-2m_w}\oplus\bH_{0,red}$.\end{lemma}

\bekezdes\label{9rem} Next, we present another realization of the
modules $\bH_*$.
For each $n\in \Z$ we
define $S_n=S_n(w)\subset \R^s$ as the union of all
the cubes $\square_q$ (of any dimension) with $w(\square_q)\leq
n$. Clearly, $S_n=\emptyset$, whenever $n<m_w$. For any  $q\geq 0$, set
$$\bS_q(\R^s,w):=\oplus_{n\geq m_w}\, H_q(S_n,\Z).$$
Then $\bS_q$ is $\Z$ (in fact, $2\Z$)--graded: the
$(-2n)$--homogeneous elements $(\bS_q)_{-2n}$ consist of  $H_q(S_n,\Z)$.
Also, $\bS_q$ is a $\Z[U]$--module; the $U$--action is the homological morphism
$H_q(S_{n},\Z)\to H_q(S_{n+1},\Z)$ induced by the inclusion $S_n\hookrightarrow S_{n+1}$.
Moreover, for
$q=0$, a fixed base-point $l_w\in S_{m_w}$ provides an augmentation
(splitting)
 $H_0(S_n,\Z)=
\Z\oplus \widetilde{H}_0(S_n,\Z)$, hence a splitting of the graded
$\Z[U]$-module
$$\bS_0=\calt^-_{-2m_w}\oplus \bS_{0,red}=(\oplus_{n\geq m_w}\Z)\oplus (
\oplus_{n\geq m_w}\widetilde{H}_0(S_n,\Z)).$$

\begin{theorem}\label{9STR1} \ (For the cohomological version see
\cite{Nlattice},\cite[Theorem 11.1.12]{Nkonyv})
There exists a graded $\Z[U]$--module isomorphism, compatible with the
augmentations:
$\bH_*(\R^s,w)=\bS_*(\R^s,w)$.
%
%
\end{theorem}
\noindent
Let us sketch the proof of  this isomorphism, since several versions of it will be used later.
 Let $(\Lb_*)_{-2n}$ denote the $\Z$-module of homogeneous elements of degree $(-2n)$ of $\Lb_*$, cf. \ref{bek:grading}.
 Since  $\partial_{w,*}$ preserves the homogeneity, the
chain complex $(\Lb_*,\partial_{w,*})$ has the following  direct sum decomposition $\oplus_n   ((\Lb_*)_{-2n},\partial_{w,*})$.
 We claim that $((\Lb_*)_{-2n},\partial_{w,*})$
can be identified with the chain complex $(\calC_*(S_n),\partial_*(S_n))$ of the space $S_n$.
Indeed, if $U^m\square_q\in (\Lb_q)_{-2n}$ is a homogeneous element, then
$-2m-2w_q(\square_q)=-2n$. Since $m\geq 0$,
necessarily $w(\square_q)\leq n$, hence $\square_q\in S_n$.
Then the correspondence $U^m\square_q\mapsto \square_q$, $(\Lb_*)_{-2n}\to \calC_*(S_n)$, is a morphism of chain complexes
which  realizes the wished  $\Z$--module isomorphism.

Choose  $\square\in S_n$. It corresponds to $U^m\square$  in $(\Lb_*)_{-2n}$, where $m =n-w(\square)$.
Let $u\square $ be $\square$  considered in $S_{n+1}$. Then, as a cube of $S_{n+1}$ it corresponds to
$U^{m+1}\square \in (\Lb_*)_{-2n-2}$ (since $m+1=n+1-w(\square)$). Hence the morphism
$H_q(S_{n},\Z)\to H_q(S_{n+1},\Z)$ corresponds to multiplication by $U$ in $\Lb_*$.

\subsection{Restrictions and the Euler characteristic}\label{ss:REu}

\bekezdes\label{9SSP} {\bf Restrictions.}  Assume that $\frR\subset \R^s$ is a subspace
of $\R^s$ consisting of a union of some cubes (from $\calQ_*$). Let
$\calC_q(\frR)$ be the free $\Z$-module generated by $q$-cubes of
$\frR$ and  $\Lb_q(\frR)=\calC_*(\frR)\otimes_{\Z}\Z[U]$.  Then
$(\Lb_*(\frR),\partial_{w,*})$ is a  chain complex, whose homology will be
denoted by $\bH_*(\frR,w)$. It has an augmentation and a  natural graded $\Z[U]$--module structure.

In some cases it can happen that the weight functions are defined only for cubes belonging to $\frR$.

Some of the possibilities (besides $\frR=\R^s$) used in  the present note are  the following:

(1) $\frR=(\R_{\geq 0})^s$ is the first quadrant and weight functions are defined for its  cubes;

(2) $\frR=R(0,c)$ is the rectangle $\{x\in\R^s \,:\, 0\leq x\leq c\}$,  where $c\geq 0$ is a lattice point;

(3) $\frR=\{x\in\R^s \,:\,  x\geq l\}$ for some fixed $l\in (\Z_{\geq0})^s$.

\bekezdes \label{bek:eu}{\bf The Euler characteristic of $\bH_*$  \cite{JEMS}.}
Though
$\bH_{*,red}(\R^s,w)$ has a finite $\Z$-rank in any fixed
homogeneous degree, in general, it is not
finitely generated over $\Z$, in fact, not even over $\Z[U]$.

Let $\frR$ be  as in \ref{9SSP} and assume that each $\bH_{q,red}(\frR,w)$ has finite $\Z$--rank.
(This happens automatically when $\frR$ is a finite rectangle.)
We define the Euler  characteristic of $\bH_*(\frR,w)$ as
$$eu(\bH_*(\frR,w)):=-\min\{w(l)\,:\, l\in \frR\cap \Z^s\} +
\sum_q(-1)^q\rank_\Z(\bH_{q,red}(\frR,w)).$$
If $\frR=R(0, c)$ (for a lattice point $c\geq 0$), then by  \cite{JEMS},
\begin{equation}\label{eq:eu}
\sum_{\square_q\subset \frR} (-1)^{q+1}w(\square_q)=eu(\bH_*(\frR,w)).\end{equation}

\section{The analytic lattice homology of curves}\label{ss:curveslattice}

\subsection{The definition and first properties of the lattice homology}

\bekezdes \label{bek:ANcurves} {\bf Some classical invariants of curves.}
Let $(C,o)$ be  an isolated curve singularity with local algebra $\calO=\calO_{C,o}$. Let
$\cup_{i=1}^r(C_i,o)$ be the irreducible decomposition of $(C,o)$, denote the local algebra of
$(C_i,o)$ by $\calO_i$. We denote the integral closure of $\calO_i$ by $\overline{\calO_i}= \C\{t_i\} $,
and we consider $\calO_i$ as a subring of $\overline{\calO_i}$. Similarly, we denote
 the integral closure of $\calO$ by $\overline{\calO}= \oplus_i \C\{t_i\}$.
 Let $\delta=\delta(C,o)$ be the delta invariant $\dim _\C\, \overline{\calO}/\calO$ of $(C,o)$.

We denote by $\frakv_i: \overline{\calO_i}\to \overline{\Z_{\geq 0}}=\Z_{\geq 0}\cup\{\infty\}$ the discrete valuation of
$\overline{\calO_i}$, where $\frakv_i(0)=\infty$. This restricted to $\calO_i$ reads as $\frakv_i(f)=
{\rm ord}_{t_i}(f)$  for $f\in \calO_i$.

  Let $\calS_i=\frakv_i(\calO_i)\cap \Z_{\geq 0}\subset \Z_{\geq 0}$ and
$\calS=(\frakv_1,\ldots, \frakv_r)(\calO)\cap (\Z_{\geq 0})^r$, or
$$\calS=\{\frakv(f):=(\frakv_1(f),\ldots, \frakv_r(f))\,:\, f \ \mbox{is a nonzero  divisor}\}\subset (\Z_{\geq 0})^r.$$
It is called the {\it semigroup of $(C,o)$}.
Let ${\mathfrak c}=(\calO:\overline{\calO})$  be the conductor ideal of $\overline{\calO}$, it is the largest ideal
of $\calO$, which is an ideal of $\overline{\calO}$ too. It has the form
$(t_1^{c_1}, \ldots, t_r^{c_r})\overline{\calO} $. The element $c=(c_1,\ldots , c_r)$ is called the
conductor of $\calS$. From definitions $c+\Z^r_{\geq 0}\subset \calS$ and $c$ is the smallest lattice point
with this property (whenever $(C,o)$ is not smooth). If $r=1$ then $\delta(C,o)=|\{\Z_{\geq0}\setminus \calS\}|$
(otherwise $|\{\Z_{\geq0}\setminus \calS\}|=\infty$).

Assume that $(C,o)$ is the union of two (not necessarily irreducible)
germs $(C',o)$ and $(C'',o)$ without common components,
and fix some embedding
$(C,o)\subset (\C^n,0)$. One defines the {\it Hironaka's intersection multiplicity }
of $(C',o)$ and $(C'',o)$ by $(C',C'')_{Hir}:=\dim ( \calO_{\C^n,o}/ I'+I'')$, where
$I'$ and $I''$ are the ideals which define $(C',o)$ and $(C'',o)$. Then one has the following formula
\cite{Hironaka,BG80,Stevens85}:
\begin{equation}\label{eq:Hir}
\delta(C,o)=\delta(C',o)+\delta(C'',o)+ (C',C'')_{Hir}.\end{equation}

From this it follows inductively that $\delta(C,o)\geq r-1$. In fact, $\delta(C,o)= r-1$
if and only if $(C,o)$ is analytically equivalent with the union of the coordinate axes of
$\C^r,0)$, called {\it ordinary $r$-tuple} \cite{BG80}.

 For {\bf plane curve germs}  $(C',C'')_{Hir}$ agrees with the usual
intersection multiplicity at $(\C^2,0)$.
In this case, the conductor entry is
$c_i=2\delta(C_i,o)+ (C_i,\cup_{j\not=i} C_j)$.
For a formula of $c_i$ in this general case see \cite{D'Anna}.
(For some additional inequalities see also the end of \ref{bek:AnnFiltr}.)

\bekezdes \label{bek:AnnFiltr} {\bf The valuative filtrations.}
We will focus on two vector spaces; the first is the infinite dimensional  local algebra $\calO$, the second one is
the finite dimensional $\overline{\calO}/\calO$.

Consider the lattice $\Z^r$ with its natural basis $\{E_i\}_i$ and partial ordering. If $l=(l_1, \ldots, l_r)\in \Z^r$
we set $|l|:=\sum_il_i$.
Then
$\overline{\calO}$ has a natural filtration indexed by $l\in \Z^r_{\geq 0}$ given by
$\overline{\calF}(l):=\{g\,:\, \frakv(g)\geq l\}$. This induces an ideal filtration of
$\calO$ by $\calF(l):=\overline{\calF}(l)\cap \calO\subset \calO$, and also a filtration
$$\calF^\circ (l)=\frac{\overline{\calF}(l)+\calO}{\calO}=\frac{\overline{\calF}(l)}{\overline{\calF}(l)\cap\calO}\subset
\overline {\calO}/\calO.$$
The first filtration of $\calO$ is `classical', it
 was considered e.g. in \cite{cdg,cdg2,Moyano,NPo,AgostonNemethi}, see also \cite{GLS,GLS2b}.
 The second filtration provides a `multivariable sum-decomposition' of $\delta$.

Set $\hh(l)=\dim \calO/\calF(l)=\dim (\overline{\calF}(l)+\calO)/\overline{\calF}(l)$.
Then  $\hh$ is increasing and $\hh(0)=0$. Moreover, we also set
$\overline{\hh} (l):= {\rm codim}( \calF^\circ (l)\subset \overline{\calO}/\calO)=\dim(\overline{\calO}/(\overline{\calF}(l)+\calO))$.
$\overline{\hh}$ is also increasing, $\overline{\hh}(0)=0$,
and $\overline{\hh}(l)=\delta$ for any $l\geq c$ (since $\overline {\calF}(l)\subset
{\mathfrak c}\subset \calO$ for such $l$).  Finally set $\hh^\circ (l):=\delta-\overline {\hh}(l)$.
One has  \begin{equation}\label{eq:DUAL1}
\hh(l)+\overline{\hh}(l)=\dim\, \overline{\calO}/\overline{\calF}(l)=|l|.
\end{equation}
Assume that $(C,o)$ is not smooth.
Since $\dim(\overline{\calO}/\frc)=|c|$ and $\dim(\calO/\frc)=\hh(c)$ we have $\delta =|c|-\hh(c)$. Since
$\hh(c)\geq  1$, $\delta \leq |c|-1$, with equality if and only if $\frc$ is the maximal ideal of $\calO$.
On the other hand, $\delta\geq |c|/2$ too, with equality if and only if $(C,o)$ is Gorenstein (cf. \cite[page 72]{Serre},
see also \ref{bek:GORdualoty}).

\bekezdes \label{bek:ANweightsCurves} {\bf The weights and the analytic lattice homology } (For the cohomology version see \cite{AgostonNemethi}.)
We consider the  lattice $\Z^r$ with its fixed basis $\{E_i\}_{i=1}^r$, and the functions $h$ and $h^\circ$ defined in \ref{bek:AnnFiltr}.
We also set $\cV:=\{1, \ldots, r\}$.
For the construction of the analytic lattice (co)homology of $(C,o)$ we consider only the first quadrant, namely
the lattice points $(\Z_{\geq 0})^r$ in $\frX:=(\R_{\geq 0})^r$ and the cubes from $\frX$.
The weight function on the lattice points $(\Z_{\geq 0})^r$ is defined by
$$w_0(l)=\hh(l)+\hh^\circ (l)-\hh^\circ(0)=\hh(l)-\overline{\hh}(l)= 2\cdot \hh(l)-|l|, $$
and $w_q(\square):=\max\{w_0(l)\,:\, \mbox{$l$ is a vertex of $\square$}\}$.
Then the cubical decomposition of $\frX$ and $l\mapsto w_0(l)=2\hh(l)-|l|$
define a lattice homology $\bH_*(\frX,w)$.
It is denoted by $\bH_*(C,o)$.

From the very construction we get
\begin{equation}\label{eq:VANISH}
\bH_q(C,o)=0 \ \mbox{for any $q\geq r$.}
\end{equation}
\begin{theorem}\label{cor:EUcurves} \cite{AgostonNemethi}
(a) For  any $c'\geq c$ the inclusion $S_n\cap R(0,c')\hookrightarrow S_n$ is a homotopy equivalence.
In particular, $S_n$ is contractible for $n\gg 0$.

(b)  One has a graded $\Z[U]$--module isomorphism  $\bH_*(C,o)=\bH_*(R(0,c'),w)$ for any $c'\geq c$
induced by the natural inclusion map. Therefore, $\bH_*(C,o)$ is determined by the weighted  cubes
of the rectangle $R(0,c)$ and $\bH_{*,red}(C,o)$ has finite $\Z$--rank.

(c) $eu(\bH_*(C,o))=\delta(C,o)$, that is,
$\bH_*(C,o)$ is a `categorification' of \,$\delta(C,o)$.
\end{theorem}
The ranks can be coded in
 the following  Poincar\'e series
$$PH(Q,h):= \sum_{n,b}\, {\rm rank} \ H_b(S_n,\Z)\,Q^nh^b=\sum_{n,b}\, {\rm rank} \ (\bH_b)_{-2n}\,Q^nh^b.$$

\subsection{The homological graded root}\label{ss:grroot}

There is an important enhancement of $\bH_0(C,o)$, the homological graded root ${\mathfrak R}(C,o)$ associated with $(C,o)$.
It is the natural homological version of the {\it cohomological graded root} already present in the literature.
 For the general construction  and
 its relation with $\bH^0$ of the cohomological root see e.g. \cite{NGr,Nlattice,Nkonyv}.

The homological root  is a connected tree whose vertices are graded by $\Z$. The vertices graded by $n\in\Z$ correspond to the connected components  $\{S_n^v\}_v$ of $S_n$. If a component $S_n^v$ of $S_n$ and a component
$S_{n+1}^u$ of $S_{n+1}$ satisfy $S_n^v\subset S_{n+1}^u$ then we introduce an edge $[v,u]$ in
${\mathfrak{R}}(C,o)$ connecting $v$ and $u$. There is a natural way to read from ${\mathfrak R}(C,o)$ the lattice homology $\bH_0(C,o)$ together with the $U$--action (see e.g. Example \ref{ex:34}).
The vertices of ${\mathfrak {R}}(C,o)$ will be denoted by $\cV({\mathfrak R})$.

 The lattice homology can be regarded as a `decoration' of ${\mathfrak {R}}(C,o)$. Indeed,
 $\bH_*(C,o)$ has a direct sum decomposition $\oplus_n \bH_*(C,o)_{-2n}$ according to the
 grading of $\bH_*$ (weights of the cubes). The homogeneous part $(\bH_*)_{-2n}$ equals $H_*(S_n,\Z)$.
 The point is that for each $S_n$ we can consider its connected components $\{S_n^v\}_v$; they are indexed  by the set of vertices of $\mathfrak{R}$ with $w_0(v)=n$. In particular, we can decorate each vertex $v\in\cV(\mathfrak{R})$  with $H_*(S^v_{w_0(v)},\Z)$.
 This shows that
 $$\bH_*(C,o)=\oplus_{v\in \cV(\mathfrak{R})} \, H_*(S^v_{w_0(v)},\Z).$$
 This introduces an additional (enhanced) grading of $\bH_*$. Namely,  the $\Z$ grading of
 $\bH_*$ (supported on $\{m_w,m_w+1,\ldots\}$)
 is replaced by the index set $\cV( {\mathfrak{R}})$, i.e. $\bH_*=
 \oplus_{v\in \cV(\mathfrak{R})} \, (\bH_*)^v_{-2w_0(v)}$.

 The $U$--action of $\bH_*$ also extends to a (graded) $U$--action of
 $ \oplus_{v\in \cV(\mathfrak{R})} \, (\bH_*)^v_{-2w_0(v)}$. Indeed, an  inclusion
 $S_n^v\subset S_{n+1}^u$  induces $U:  (\bH_*)^v_{-2w_0(v)}\to  (\bH_*)^u_{-2w_0(u)}$.

 This `homological  decoration' can be simplified if we replace $(\bH_*)^v_{-2w_0(v)}$ by its Poincar\'e
 polynomial $PH^v_{-2w_0(v)}(h)=\sum_b \, {\rm rank}\, (\bH_b)^v_{-2w_0(v)}\, h^b$.
 Then for each $n\in\Z$ the sum $\sum_{v\,:\, w_0(v)=n}PH^v_{-2n}(h)$ is the coefficient of $Q^n$ in
 $PH(Q,h)$.

\subsection{Rosenlicht's forms}\label{rem:Rosenlicht}
We have the following reinterpretation of  $\hh^\circ$ in terms of forms.

Write  $n:\overline{(C,o)}\to (C,o)$ for the normalization.
Let $\Omega^1(*)$ be the germs of meromorphic differential forms on the normalization
$\overline{(C,o)}$  with a pole (of any order) at most in $\overline {o}=n^{-1}(o)$.
Let  $\Omega^1_{\overline{(C,o)}}$ be the germs of regular differential forms on $\overline{(C,o)}$.
 The {\it Rosenlicht's regular differential forms}
  are defined as
 $$\omega^R_{C,o}:=\{ \alpha\in \Omega^1(*)\,:\, \sum _{p\in \overline{o}} {\rm res}_p(f\alpha)=0 \ \
 \mbox{for all $f\in \calO$}\}.$$
 (In fact one shows that it is canonically isomorphic with the dualizing module of Grothendieck associated with $(C,o)$.)
 Then, by \cite{Serre,BG80},  one has a perfect duality between $\omega^R_{C,o}/ n_*\Omega^1_{\overline{(C,o)}}$ and
 $\overline{\calO}/\calO=n_*\calO_{\overline{(C,o)}}/ \calO_{C,o}$:
 \begin{equation}\label{eq:ROS}
 n_*\calO_{\overline{(C,o)}}/ \calO_{C,o}\ \times \ \omega^R_{C,o}/ n_*\Omega^1_{\overline{(C,o)}}\to \C,\ \ \
 [f]\times [\alpha]\mapsto \sum_{p\in\overline{o}} \, {\rm res}_p(f\alpha).\end{equation}
Moreover,  one can define a $\Z^r$--filtration in  $\omega^R_{C,o}/ n_*\Omega^1_{\overline{(C,o)}}$
 such that the duality is compatible with the  filtrations in
 $n_*\calO_{\overline{(C,o)}}/ \calO_{C,o}= \overline{\calO}/\calo$ and
 $\omega^R_{C,o}/ n_*\Omega^1_{\overline{(C,o)}}$.

\bekezdes \label{bek:GORdualoty} {\bf The Gorenstein case.}
By Serre \cite{Serre} or Bass  \cite{Bass} (see also \cite{Huneke}) $(C,o)$ is Gorenstein if and only if
$\dim (\overline{\calO}/\calO)=\dim(\calO/\mathfrak{c})$. On the other hand, Delgado in \cite{delaMata}
proved that the condition $\dim (\overline{\calO}/\calO)=\dim(\calO/\mathfrak{c})$ is equivalent with the symmetry of the semigroup of values
$\calS$. If  $r=1$ then the symmetry can be formulated easily:
$l\in\calS\ \Leftrightarrow \ c-1-l\not\in\calS$. If $r\geq 2$ then the definition is the following \cite{delaMata}.
For any $l \in \Z^r$ and $i\in\{1,\cdots, r\}$ set
$$\Delta_i(l)=\{s\in\calS\, :\,
s_i=l_i \ \mbox{and} \ s_j>l_j \ \mbox{for all} \ j\not= i\}, \ \ \mbox{and} \ \  \Delta(l):=\cup_i\Delta_i(l).$$
Then $\calS$ is called symmetric if $l\in\calS\ \Leftrightarrow \ \Delta(c-{\bf 1}-l)=\emptyset$. (Here ${\bf 1}=(1,\ldots, 1)$.)

If $(C,o)$ is Gorenstein then in \cite{cdk} (see also \cite{Moyano}) is proved that
$$\hh(l)-\hh(c-l)=|l|-\delta.$$
This combined with (\ref{eq:DUAL1}) gives
$\hh(c-l)=\hh^\circ(l)$. In particular, $\hh^\circ$ is recovered from $\hh$ as  its symmetrization
with respect to $c$, namely from $h^{sym}(l):=\hh(c-l)$ for any $l\in R(0, c)$.

In particular, the weight function and $\bH_*(C,o)$ admits a $\Z_2$--symmetry induced by $l\mapsto c-l$.

\subsection{The `hat'--version $\hat{\bH}$.}\label{ss:hat}
Once the tower of  spaces $\{S_n\}_n$ is  defined, it is natural to consider the relative homologies as well:
$$\hat{\bH}_b(C,o)=\hat{\bH}_b(\frX,w):= \oplus_{n}\, H_b(S_n,S_{n-1},\Z).$$
It is a graded $\Z$--module, with $(-2n)$--homogeneous summand $H_b(S_n,S_{n-1},\Z)$. The $U$--action
induced by the inclusion of pairs $(S_{n}, S_{n-1})\hookrightarrow (S_{n+1},S_{n})$  is trivial.
By Theorem \ref{cor:EUcurves},  $ H_b(S_n,S_{n-1},\Z)\not=0$ only for finitely many pairs $(n,b)$.
By a  homological argument (based on the contractibility of $S_n$ for $n\gg 0$, cf. Theorem \ref{cor:EUcurves}, see also
Remark \ref{rem:hat})
\begin{equation}\label{eq:hateu}
\sum_{n,b}\ (-1)^b \,{\rm rank}\, H_b(S_n, S_{n-1},\Z)=1.
\end{equation}
One also has the exact sequence
$\cdots \to \bH_b\stackrel{U}{\longrightarrow} \bH_b\longrightarrow \hat{\bH}_b\longrightarrow \bH_{b-1}
\stackrel{U}{\longrightarrow} \cdots$.

\subsection{Deformations and functors}\label{ss:deffunc}
It is natural to ask what are the functors of the homology theory $\bH_*$. We expect that they are
induced by flat deformation of singularities. In \cite{AgostonNemethi}  we proved the following partial result.

 \begin{theorem}\label{th:DEF}
    Consider  a flat deformation of isolated  curve singularities ${(C_t,o)}_{t\in(\bC,0)}$.

   Assume that  either (a) $(C_{t=0},o)$ irreducible, or (b)  ${(C_t,o)}_{t\in(\bC,0)}$ is a delta-constant and $r$-constant
    deformation of plane curve germs.

        Then the linear map connecting the corresponding lattices
    (a)  $\Z^r\to \Z$, $\ell\mapsto \phi(\ell)=|\ell|=\sum_i\ell_i$, respectively
     (b)  $\Z^r\to \Z^r$, $\phi(\ell)=\ell$,
 induces a degree zero graded $\Z[U]$-module morphism $ \bH_*(C_{t\not=0},o)\to \bH_*(C_{t=0},o)$, and similarly a graded (graph) map of degree zero at the level of graded roots ${\mathfrak R}(C_{t\not=0},o)\to{\mathfrak R}(C_{t=0},o)$.

 \vspace{1mm}

 In the general case, when we do not impose any restriction regarding the deformation, then still
 one constructs a natural map which induces a degree zero graded $\Z[U]$-module morphism $ \bH_0(C_{t\not=0},o)\to \bH_0(C_{t=0},o)$, and similarly a graded (graph) map of degree zero at the level of graded roots ${\mathfrak R}(C_{t\not=0},o)\to{\mathfrak R}(C_{t=0},o)$.
\end{theorem}

\section{More invariants read from $\hh$}\label{s:more}

\subsection{The Hilbert series}

\begin{definition}\label{def:Hil}
 The {\it Hilbert series of the multi-index filtration} is
\begin{equation}
H(t_1,\ldots,t_r)=\sum_{l\in \Z^r}\ \hh(l)\cdot t_1^{l_1}\cdots\, t_r^{l_r}=\sum_{l\in \Z^r} \hh(l)\cdot \bt^l \in
{\mathbb Z}[[t_1,t_1^{-1},\ldots, t_r,t_r^{-1}]].
\end{equation}
\end{definition}
\noindent  Here $l=\sum_{i=1}^r l_iE_i$.  Note that
\begin{equation}\label{eq:MAXh}
\hh(l)=\hh(\max\{l,0\}).
\end{equation}
 Hence $H$ determined completely by
$H(\bt)|_{l\geq 0}:=\sum_{l\geq 0} \hh(l)\cdot \bt^l$.

%
%
%
%

\subsection{The semigroup of $C$.} The semigroup $\calS$ and the Hilbert function $H$ determine each other:
\begin{lemma} (See e.g. \cite{GorNem2015})
\label{eq:semi}
The semigroup can be deduced from the Hilbert function as follows:
$$
\calS=\{l\in \Z_{\geq 0}^r\ |\ \hh(l+E_i)>\hh(l) \ \ \mbox{for every \ $i=1,\ldots, r$}\}.
$$
On the other hand,
 $\hh(l+E_i)-\hh(l)\in\{0,1\}$ for any $l\geq 0$ and $i\in \cV$,
Moreover,   $\hh(l+E_i)=\hh(l)+1$ if there
is an element $s\in \calS$  such that
$s_i=l_i$ and  $s_j\ge l_j$ for $j\neq i$. Otherwise $\hh(l+E_i)=\hh(l).$
\end{lemma}

\subsection{The Poincar\'e series.} If $r=1$, then the Poincar\'e series
of the graded ring $\oplus_{l} \cF(l)/\cF(l+E_1)$ is   $P(t)=-H(t)(1-t^{-1})$. For general $r$, one
 defines the Poincar\'e series  by
\begin{equation}\label{eq:HPC}
P(t_1,\ldots,t_r):=-H(t_1,\ldots,t_r)\cdot \prod_i(1-t_i^{-1}).
\end{equation}
This means that the  coefficient $\pp_l$ of $
P(\bt)=\sum_{l}\pp_l\cdot t_1^{l_1}\ldots t_r^{l_r}$ satisfies
\begin{equation}
\label{htopi}
\pp_{l}=\sum_{I\subset \cV}(-1)^{|I|-1}\hh(l+E_{I}), \ \ \ \ (E_I=\sum_{i\in I}E_i).
\end{equation}

The space ${\mathbb Z}[[t_1,t_1^{-1},\ldots, t_r,t_r^{-1}]]$
is a module over the ring of Laurent power series, hence
the multiplication by $\prod_i(1-t_i^{-1})$ in \eqref{eq:HPC} is  well-defined. 
 One can check (using e.g. \eqref{eq:MAXh}) that the right hand side of \eqref{eq:HPC}
is a power series involving only nonnegative powers of $t_i$. In fact, cf. Lemma \ref{prop:motProp},
the support of $P(\bt)$ is included in  $\calS$.

If $r=1$, then Lemma \ref{eq:semi} implies that $P(t)=\sum_{s\in \calS}t^s=-\sum _{s\not\in \calS}t^s+ 1/(1-t)$, where
$P^+(t):=-\sum _{s\not\in \calS}t^s$ is a polynomial. Furthermore, by \cite{cdg}, $P(\bt)$ is a polynomial for $r>1$.

Multiplication by $\prod_i(1-t_i^{-1})$ of series with $\hh(0)=0$ is injective if $r=1$, however it is not if $r>1$. In particular, in such cases
 it can happen that
for two different series $H(\bt)$ we obtain the very same $P(\bt)$. For a concrete pair see e.g.
\cite{cdg3}. Nevertheless, even in such cases $r>1$ one can recover $H(\bt)$ as follows.

For any subset $J=\{i_1,\ldots,i_{|J|}\}\subset \cV:=\{1,\ldots, r\}$, $J\not=\emptyset$,
 consider the curve germ  $(C_{J},o)=\cup_{i\in J}(C_{i},o)$. As above, this germ defines
 the Hilbert series $H_{C_J}$ of $(C_J,o)$ in
 variables $\{t_i\}_{i\in J}$:
 $$H_{C_J}(t_{i_1},\ldots,t_{i_{|J|}})=\sum_{l}
  \hh_{J}(l)\cdot
 t_{i_1}^{l_{i_1}}\ldots t_{i_{|J|}}^{l_{i_{|J|}}}.$$
 By the very definition,
 \begin{equation}\label{eq:redHilb}
H_{C_J}(t_{i_1},\ldots,t_{i_{|J|}})=H(t_1,\ldots,t_r)|_{t_i=0\ i\not\in J}.\end{equation}
Analogously, we also consider  the  Poincar\'e series of $(C_J,o)$:
$$P_{C_J}(t_{i_1},\ldots,t_{i_{|J|}})=\sum_{l}
 \pp_{J,l}\cdot
t_{i_1}^{l_{i_1}}\ldots t_{i_{|J|}}^{l_{i_{|J|}}}$$
computed from $H_{C_J}$ by a similar identity as (\ref{eq:HPC}).
By definition, for $J=\emptyset$ we take $P_{\emptyset}\equiv 0.$

The next theorem inverts  (\ref{htopi}) in the sense that  we  recover $H$
from the collection $\{P_{C_J}\}_J$.

\begin{theorem}\label{reconst} (\cite[Theorem 3.4.3]{GorNem2015}, see also  \cite[Corollary 4.3]{julioproj})
With the above notations
\begin{equation}
\label{hilbert}
H(t_1,\ldots,t_r)|_{l\geq 0}
=\frac{1}{ \prod_{i=1}^{r}(1-t_i)}\sum_{J\subset \cV}(-1)^{|J|-1}
\Big(\prod_{i\in J}t_i\Big)\cdot P_{C_J}(t_{i_1},\ldots,t_{i_{|J|}}).
\end{equation}
In particular,
the restricted Hilbert series  $H(t)|_{l\geq 0}$
of a multi-component curve is a rational function with denominator $\prod_{i=1}^{r}(1-t_i)^2$.
\end{theorem}

Assume that $(C,o)$ is Gorenstein. If $r=1$ and one writes $P(t)=\Delta(t)/(1-t)$ (for a certain  polynomial $\Delta(t)$),
then from the Gorenstein symmetry
of the semigroup one gets the symmetry of $\Delta(t)$, namely $\Delta(t^{-1})=t^{-\mu(C,o)}\Delta(t)$.
More generally, for any $r>1$, the polynomial $P(\bt)$ satisfies the Gorenstein symmetry
$P(t_1^{-1},\ldots, t_r^{-1})=(-1)^r\prod_i t_i^{1-c_i}\cdot P(t_1,\ldots , t_r)$.

\subsection{The local hyperplane arrangements. }\label{ss:ARR}  For any fixed $l$ let us consider the
set
\begin{equation*}
\cH(l):=\{f\in \calO\, :\, \frakv(f)=l\}=\calF(l)\setminus \bigcup_i\, \calF(l+E_i).
\end{equation*}
Since $\calF(l+E_i)$ is either $\calF(l)$ or one of its hyperplanes (cf. \ref{eq:semi}), $\cH(l)$ is
either empty or it is a hyperplane
arrangement in  $\calF(l)$. This can be reduced to a finite dimensional
central hyperplane arrangement
$$\cH'(l):=\frac{\calF(l)}{\calF(l+E_{\cV})}\setminus \bigcup_i \frac{\calF(l+E_i)}{\calF(l+E_{\cV})},$$
since $\cH(l)\simeq \calF(l+E_{\cV})\times \cH'(l)$ as vector spaces. Note that both $\cH'(l)$ and $\cH(l)$
admit a free $\C^*$--action (multiplication by nonzero scalar), hence one automatically has the
two projective arrangements $\setP\cH'(l)=\cH'(l)/\C^*$ and $\setP\cH(l)=\cH(l)/\C^*$.
In fact $\cH'(l)=\setP\cH'(l)\times \C^*$.
The first part of the following proposition follows from \cite[Theorem 5.2]{or1},
the second part
can be deduced from \eqref{htopi} and inclusion-exclusion formula (see e.g. \cite{cdg2,cdg} and Lemma \ref{lem:L} below).

\begin{proposition}
$H_*(\cH'(l),\Z)$ and $H_*(\setP\cH'(l),\Z)$  have trivial  $\Z$--torsion.
The Euler characteristic of\ \  $\setP\cH(l)$ (and of\ \ $\setP\cH'(l)$) equals $\pp_{l}$, the coefficient of the Poincar\'e series $P(\bt)$ at \ $\bt^{l}$.
\end{proposition}

\subsection{Motivic Poincar\'e series.}\label{ss:MP} The series $P^m(t_1,\ldots, t_r;q)\in
\Z[[t_1,\ldots, t_r]][q]$ is defined in \cite{cdg3} as a refinement of $P(\bt)$.
By definition, the coefficient  of
$t_1^{l_1}\ldots t_r^{l_r}$  is the (normalized) class of
$\setP\cH'(l)$ in the Grothendieck ring of algebraic varieties. It turns out that the class
of a central hyperplane arrangement can always be expressed in terms of the class ${\mathbb L}$
of the affine line. Indeed, one has:
\begin{lemma}\label{lem:L}
$V$ be a vector space and let $\cH=\{\cH_i\}_{i\in\cV}$ be a collection of linear hyperplanes in $V$.
For a subset $J\subset \cV$ we define the rank function by
$\rho(J)={\rm codim} \{\cup_{i\in J}\cH_i\subset V\}$.
Then in the Grothendieck ring of varieties
(by the inclusion-exclusion formula) one has
 $$[V\setminus \cup_{i\in\cV}\cH_i]=\sum_{J\subset \cV}(-1)^{|J|}\ [\cap_{i\in J}\cH_i]
 =\sum_{J\subset \cV}(-1)^{|J|}\ {\mathbb L}^{\dim V-\rho(J)}.$$
 Since $[\C^*]={\mathbb L}-1$, one also has $[(V\setminus \cup_{i\in\cV}\cH_i)/\C^*]=[V\setminus \cup_{i\in\cV}\cH_i]
 /({\mathbb L}-1)$.
\end{lemma}
\begin{corollary}
The class of the (finite) local hyperplane arrangement $\cH'(l)$ equals
$$[\cH'(l)]=({\mathbb L}-1)[\setP\cH'(l)]=\sum_{J\subset \cV}(-1)^{|J|}\ {\mathbb L}^{\hh(l+E_{\cV})-\hh(l+E_J)}.$$
\end{corollary}
Replacing ${\mathbb L}^{-1}$ by a new variable $q$, one can define (following \cite{cdg3}) the {\it motivic Poincar\'e series}
$P^{m}(\bt;q):=\sum_{l}\pp^m_{l}(q)\cdot \bt^{l}$ by
\begin{equation}\label{eq:pmot}
 \pp^m_{l}(q):={\mathbb L}^{1-\hh(l+E_{\cV})}[\setP\cH'(l)]\Big|_{{\mathbb L}^{-1}=q}
=\sum_{J\subset \cV}(-1)^{|J|}\frac{q^{\hh(l+E_J)}}{1-q}=
\sum_{J\subset \cV}(-1)^{|J|}\cdot\frac{q^{\hh(l+E_J)}-q^{\hh(l)}}{1-q}.
\end{equation}
For another definition in terms of motivic integrals, see \cite{cdg3}.
\begin{proposition}\label{prop:motProp} \cite{cdg3,MoyanoTh,Moyano,Gorsky,GorNem2015}
(a) $\lim_{q\to 1}P^m(\bt;q)=P(\bt)$;

(b)  $P^m(\bt;q)$ is a rational function of type
$\overline{{P}^m}(\bt;q)/(\prod_{i\in\cV}(1-t_{i}q))$, where
$\overline{{P}^m}(\bt;q)\in \Z[\bt,q]$.

(c) The support of $P^m(\bt;q)$ is exactly $\calS$. That is, $\pp^m_l(q)\not=0$ if and only if  $l\in \calS$.

(d)  In the
Gorenstein case, $$\overline{P^m}((qt_1)^{-1},\ldots , (qt_r)^{-1};q) =q^{-\delta(C,o)}
\prod_{i\in \cV} t_i^{-c_i}\cdot \overline{P^m}(t_1,\ldots, t_r;q).$$
\end{proposition}
For a combinatorial formula, valid for plane curve singularities, in terms
of the embedded resolution graph, see \cite{cdg3}. The formula from  \cite{cdg3} was simplified in \cite{Gorsky}.

\subsection{The case of plane curves singularities; Poincar\'e series versus Alexander polynomial.}\label{ss:PA}\ \\
In the case of plane curve singularities the above invariants ($H$, $P$, $\calS$) can be compared with the embedded topological
type of the link of $(C,o)$ embedded into $S^3$. One has the following statement:
\begin{theorem} [\cite{cdg2,cdg}]
\label{Poincare vs Alexander}
Let $\Delta(t_1,\ldots,t_r)$ be the multi-variable Alexander polynomial of the link of $(C,o)$.
If \,$r=1$ then $P(t)(1-t)=\Delta(t)$,
while  $P(\bt)=\Delta(\bt))$  if \,$r>1$.
\end{theorem}
Since plane curve germs are Gorenstein, the above symmetry of $P$ is compatible with the well--known symmetry
property  of the Alexander polynomials.

By \cite{Yamamoto} the multi-variable Alexander polynomial
(and hence by Theorem \ref{Poincare vs Alexander},
the Poincar\'e series $P(\bt)$) determines the embedded topological type of
$(C,o)$, in particular it determines all the series  $\{P_{C_J}\}_{J\subset \cV}$.
In particular, it determines via (\ref{hilbert}) the series $H(\bt)$ as well.

However,  the
reduction  procedure from $P$ to $P_{C_J}$ is more complicated than
the analog of (\ref{eq:redHilb}) valid for the Hilbert series (the `naive substitution').
Indeed, these reductions   are of type (see \cite{torres}):
 \begin{equation}\label{eq:redP}
P_{C_{\cV\setminus \{1\}}}(t_2,\ldots,t_r)=
P(t_1,\ldots,t_r)|_{t_1=1}\cdot \frac{1}{(1-t_2^{(C_1,C_2)}\cdots t_r^{(C_1,C_r)})}.
\end{equation}

Let us prove the identity $\hh(l)=|l|-\delta$, valid for $l\geq c$,  using
(\ref{hilbert}).
For $J=\{i\}$ we have  $\sum_{0\leq u_i\leq l_i-1}\pp^J_{u_i}=l_i-\delta(C_i,o)$.
For $J=\{i,j\}$ (since $P_{C_J}$ is a polynomial)
$\sum_l \pp^{\{i,j\}}_l=P_{C_J}(1,1)$. This equals $(C_i,C_j)$ by \eqref{eq:redP}.
By similar argument, for $|J|>2$ the contribution is zero.
Hence
$\hh(l)=\sum_i (l_i-\delta(C_i,o))-\sum_{i\not= j}(C_i, C_j)=|l|-\delta(C,o).$

Note also that $\Delta(\bt)$ can also be deduced from the splice diagram, or embedded resolution graph
of the link (or of the embedded topological type of the pair $(C,o)\subset (\C^2,0)$) \cite{EN}.
(In fact, this is the most convenient way to determine $\Delta(t)$, at least for high $r$.)

The above discussions show that in the case of plane curve singularities the invariants
$H$, $P$, $\calS$, $P^m$ are all equivalent and are complete topological invariants of the embedded link into $S^3$.

Note also that the embedded link is fibred, the first Betti number of the   (Milnor) fiber $F$ is $\mu(F)$, the Milnor number of
$(C,o)$.   It satisfies $\mu(C,o)=2\delta(C,o)-r+1$, cf. \cite{MBook}, i.e., $\delta(C,o)$ is the
genus of the link. Furthermore, $\mu(C,o)$ equals ${\rm deg}\,\Delta(t)$ if $r=1$ and
$1+{\rm deg}\,\Delta(t,\cdots, t)$ if $r>1$.

\subsection{The local lattice cohomology \cite{GorNem2015}.} \label{bek:2.22}
Consider again the space (CW complex) $\frX=(\R_{\geq 0})^r$  with its cubical decomposition. Recall that
the $q$-cubes have the form $\square=(l, I)$, with $l\in (\Z_{\geq 0})^r$,
and $I\subset \cV$, $q=|I|=\dim(\square)$, where the vertices of $(l,I)$ are $\{l+\sum_{i\in I'}E_i\}_{I'\subset I}$.

Let us consider the weight function $l\mapsto \hh(l)$, {\it provided by the Hilbert function}, and the  lattice homology associated with it.
(This should not be confused with the object defined in \ref{bek:ANweightsCurves}:
note that the weight functions are different!)
More precisely, in this case the weights of the cubes are defines as
$\hh((l,I))=\max\{\hh(l'), \ \mbox{$l'$ vertex of $(l,I)$}\}=\hh(l+E_I)$. The homological complex $\mathcal{L}^-_*$ is $\calC_*\otimes_\Z\Z[U]$,
the boundary operator is $\partial_U$ given by
$\partial _U(U^m\square)=U^m\sum_k \epsilon _k U^{\hh(\square)-\hh(\square^k)}\square^k$; it satisfies $\partial _U^2=0$.
Motivated by the Link Heegaard Floer Theory, in \cite{GorNem2015} it was introduced the {\it homological degree} of
the generators $U^m\square$ as well:
\begin{equation}\label{eq:szamdeg}\deg(U^m\square):= -2m -2\hh(\square)+\dim(\square).\end{equation}
Then $\partial _U$ decreases the homological degree by one. The homology of $(\calL^-_*,\partial _{U,*})$ is not very rich
(it is $\Z[U]$, cf. \cite{GorNem2015}), however  deep information is coded in the `{\it local lattice homology groups}'
 read from the graded version of $(\calL^-_*,\partial_{U,*})$, which is defined as follows.

First note that $\calL^-$, as a $\Z$--module has a sum decomposition $\oplus _l\calL^-(l)$, where
$\calL^-(l)$ is generated by cubes of the form $(l,I)$.
Then  one sets the multi-graded direct sum complex ${\rm gr}\,\, \calL^-=\oplus_{l\in (\Z_{\geq 0})^r}{\rm gr}_l\,\calL^-$, where
 $${\rm gr}_l\,\calL^-=\calL^-(l)/\sum_i \calL^-(l+E_i)=\Z\langle (l,I); \ I\subset \cV\rangle\otimes _{\Z}\Z[U]$$
 together with the boundary operator  ${\rm gr}_l\,\partial _U$ defined by
 $${\rm gr}_l\,\partial _U((l, I))=\sum_k (-1)^kU^{\hh(l+E_I)-\hh(l+E_{I\setminus \{k\}})}\ (l, I\setminus \{k\}).$$
 Then $H_*({\rm gr}_l\calL^-, {\rm gr}_l \partial _U)$, graded by the induced homological degree $\deg$, is called the {\it
 local lattice homology}  associated with the weight function $\hh$ and the lattice point $l$. It is denoted by ${\rm HL}^-(l)$, $l\in (\Z_{\geq 0})^r$.
 Its Poincar\'e series is denoted by
 $$P_l^{\calL^-}(s):=\sum_p \ {\rm rank}\, H_p({\rm gr}_l \, \calL^-, {\rm gr}_l\, \partial _U)\cdot s^p.$$
 In fact, cf. \cite{GorNem2015},  the complex is bigraded by
 \begin{equation}\label{eq:bideg}
 {\rm bdeg} (U^m \square):= (-2m -2\hh(\square), \, \dim(\square)\,)\in\Z^2.
 \end{equation}
 The total degree of bdeg is the homological degree $-2m-2\hh(\square)+\dim(\square)$, cf \ref{eq:szamdeg}. The operator ${\rm gr}_l\partial_U$
 has bidegree $(0,-1)$. In particular, the local lattice homology is also bigraded. Let ${\rm HL}^-_{a,b}(l)$ denote the
 $(a,b)$-homogeneous components of ${\rm HL}^-(l)$.

 In \cite{GorNem2015} the following facts are proved.

 \begin{theorem}\cite[Theorems 4.2.1 and 5.3.1]{GorNem2015}
\label{zlat}

(1)  Consider the motivic Poincar\'e series of $(C,o)$,
$P^m(\bt;q)=\sum_{l}\pp^m_{l}(q)\,\bt^{l}$.
 Then the Poincar\'e polynomial of\, ${\rm HL}^-(l)$,
satisfies
\begin{equation}
\label{hfminus}
P^{\calL^-}_{l}(-q^{-1})=q^{\hh(l)}\cdot \pp^m_{l}(q).
\end{equation}
In particular, $(-1)^{\hh(l)}\cdot \pp^m_l(-q)$ is a polynomial in $q$ with nonnegative coefficients, and
the Euler characteristic $P^{\calL^-}_l(-1)=\sum_p(-1)^p \,{\rm rank} \, H_p({\rm gr}_l\calL^-, {\rm gr}_l\partial_U)$
equals $\pp_l^m(1)=\pp_l$, the $l$-coefficient of $P(\bt)$.

(2)  $H_{-2\hh(l)-p}({\rm gr}_l\,\calL^-, {\rm gr}_l\, \partial _U)\simeq H_p({\mathbb P}\cH'(l),\Z)$. Hence $H_{*}({\rm gr}_l\,\calL^-, {\rm gr}_l\, \partial _U)$
has no $\Z$-torsion.

(3) If $H_{a,b}({\rm gr}_l\,\calL^-, {\rm gr}_l\, \partial _U)\not=0$ then necessarily $(a,b)$ has the form
$(-2\hh(l)-2|I|, |I|)$. (In other words, $a+2b=-2\hh(l)$ and $r\geq b\geq 0$.)

(4) The $U$--action on  $H_*({\rm gr}_l\,\calL^-, {\rm gr}_l\, \partial _U)$ is trivial.

\end{theorem}

\begin{remark}\label{rem:fontos}
 Theorem \ref{zlat}  was stated  in \cite{GorNem2015} for plane curves, but proofs are valid for any
algebraic curve based on the general properties of the Hilbert function.

The key statement {\it (3)} is based on properties of the Orlik--Solomon algebras associated with
the local hyperplane  arrangements.
\end{remark}

%

\section{The level filtration of the analytic lattice homology}\label{s:levfiltr}

\subsection{The submodules ${\rm F}_{-d}\bH_*(\frX,w)$.}\

Let us fix an isolated curve singularity $(C,o)$.
We consider again the space $\frX=(\R_{\geq 0})^r$, its cubical decomposition and the lattice homology
$\bH_*(C,o)=\bH_*(\frX,w)=H_*(\Lb_*,\partial_{w,*})$ associated with the weight function
$l\mapsto w_0(l):= 2\hh(l)-|l|$. For different notations see the previous  sections. 

In the present article we define several filtrations at the level of complexes and also at the level
of the lattice homology  $\bH_*(C,o)$, and we analyse the corresponding spectral sequences.

We start with the most natural $\Z$--filtration. It is induced by the grading of  the lattice points
$(\Z_{\geq 0})^r$ indexed by
$\Z_{\geq 0}$,   $l\mapsto |l|$.
 Since in the spectral sequences associated with a filtration of subspaces the literature
prefers {\it increasing} filtrations, we will follow this setup. For any $d\in \Z_{\geq 0}$ we define
the subspaces $\frX_{-d}:= \{\cup(l, I) \,:\, |l|\geq d\}$ of $\frX$. Then we have the infinite sequence of subspaces
$$\frX=\frX_0\supset \frX_{-1}\supset \frX_{-2}\supset \ \cdots. $$
Note also that for any $\square_q =(l,I)\in \frX_{-d}$ its faces belong to $\frX_{-d}$ too.
Accordingly we have the chain complexes
$\calC_*(\frX_{-d})=\Z\langle (l,I)\,:\, |l|\geq d\rangle$ endowed with the natural boundary operators
$\partial \square_q=\sum_k\varepsilon_k \, \square_{q-1}^k$, together with the chain inclusions
$\calC_*=\calC_*(\frX)\supset \calC_*(\frX_{-1})\supset  \calC_*(\frX_{-2})\supset\ \cdots.$

Then we define the graded $\Z[U]$--module chain complexes
$\Lb_*(\frX_{-d})=\calC_*(\frX_{-d})\otimes _{\Z}\Z[U]$
endowed with its natural  boundary operator $\partial _{*,w}$.
In this way  we obtain the sequence of chain complexes
$\Lb_*=\Lb_*(\frX) \supset \Lb_*(\frX_{-1})\supset \Lb_*(\frX_{-2})\supset \cdots$
and a  sequence of graded $\Z[U]$--module morphisms
$\bH_*(\frX,w)\leftarrow \bH_*(\frX_{-1},w)\leftarrow \bH_*(\frX_{-2},w)\leftarrow \cdots$

For each $b$, the map $\bH_b(\frX,w)\leftarrow \bH_b(\frX_{-d}, w)$ induced at lattice homology level is
homogeneous of degree zero. 
These morphisms provide the following filtration of  $\Z[U]$--modules in $\bH_*(\frX,w)$
$${\rm F}_{-d}\bH_*(\frX,w):={\rm im}\big( \bH_*(\frX,w)\leftarrow \bH_*(\frX_{-d},w)\, \big).$$

\begin{example}\label{ex:34}
Consider the irreducible plane curve singularity $\{x^3+y^4=0\}$ with semigroup $\calS=\langle 3,4\rangle$.
The conductor is $c=\mu=6$ and the Hilbert function and $w_0$ are shown in the next picture.

\begin{picture}(320,60)(0,0)

\put(20,50){\makebox(0,0){\small{$\calS$}}}
\put(20,35){\makebox(0,0){\small{$\hh$}}}
\put(20,20){\makebox(0,0){\small{$w_0$}}}
 \put(40,50){\line(1,0){180}}

\put(40,50){\circle*{4}} \put(100,50){\circle*{4}} \put(120,50){\circle*{4}}
\put(160,50){\circle*{4}} \put(180,50){\circle*{4}} \put(200,50){\circle*{4}}
\put(240,50){\makebox(0,0){$\ldots$}}
\put(240,35){\makebox(0,0){$\ldots$}}
\put(240,20){\makebox(0,0){$\ldots$}}

\put(60,50){\makebox(0,0){\small{$\circ$}}}
\put(80,50){\makebox(0,0){\small{$\circ$}}}
\put(140,50){\makebox(0,0){\small{$\circ$}}}


\put(40,35){\makebox(0,0){\small{$0$}}}
\put(60,35){\makebox(0,0){\small{$1$}}}
\put(80,35){\makebox(0,0){\small{$1$}}}
\put(100,35){\makebox(0,0){\small{$1$}}}
\put(120,35){\makebox(0,0){\small{$2$}}}
\put(140,35){\makebox(0,0){\small{$3$}}}
\put(160,35){\makebox(0,0){\small{$3$}}}
\put(180,35){\makebox(0,0){\small{$4$}}}
\put(200,35){\makebox(0,0){\small{$5$}}}

\put(40,20){\makebox(0,0){\small{$0$}}}
\put(60,20){\makebox(0,0){\small{$1$}}}
\put(80,20){\makebox(0,0){\small{$0$}}}
\put(100,20){\makebox(0,0){\small{$-1$}}}
\put(120,20){\makebox(0,0){\small{$0$}}}
\put(140,20){\makebox(0,0){\small{$1$}}}
\put(160,20){\makebox(0,0){\small{$0$}}}
\put(180,20){\makebox(0,0){\small{$1$}}}
\put(200,20){\makebox(0,0){\small{$2$}}}

\end{picture}

Then $\bH_{>0}(\frX)=0$ and $\bH_0(\frX)=\calt^-_{2}\oplus \calt_0(1)^2$. One has  $eu(\bH_*)=3=\delta(C,o)$.

$\bH_0=\oplus_n H_0(S_n,\Z)$ can be illustrated by  its `homological graded root' ${\mathfrak R}(C,o)$:

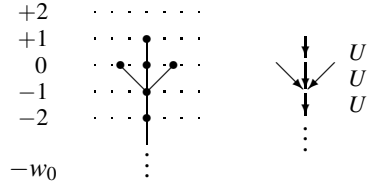
\begin{picture}(300,82)(80,330)

\put(180,380){\makebox(0,0){\footnotesize{$0$}}} \put(177,370){\makebox(0,0){\footnotesize{$-1$}}}
\put(177,360){\makebox(0,0){\footnotesize{$-2$}}}\put(177,390){\makebox(0,0){\footnotesize{$+1$}}}
\put(177,340){\makebox(0,0){\small{$-w_0$}}}\put(177,400){\makebox(0,0){\footnotesize{$+2$}}}
\dashline{1}(200,370)(240,370)
\dashline{1}(200,380)(240,380) \dashline{1}(200,400)(240,400)
\dashline{1}(200,360)(240,360)
 \dashline{1}(200,390)(240,390)
\put(220,345){\makebox(0,0){$\vdots$}} \put(220,360){\circle*{3}}
\put(220,390){\circle*{3}} \put(210,380){\circle*{3}}
\put(230,380){\circle*{3}} \put(220,380){\circle*{3}}
\put(220,370){\circle*{3}} \put(220,390){\line(0,-1){40}}
\put(220,370){\line(1,1){10}} \put(210,380){\line(1,-1){10}}

\put(280,381){\vector(0,-1){10}}\put(269,381){\vector(1,-1){10}}\put(291,381){\vector(-1,-1){10}}
\put(300,375){\makebox(0,0){\footnotesize{$U$}}}
\put(280,391){\vector(0,-1){8}}
\put(280,369){\vector(0,-1){8}}
\put(300,385){\makebox(0,0){\footnotesize{$U$}}}
\put(300,365){\makebox(0,0){\footnotesize{$U$}}}
\put(280,355){\makebox(0,0){$\vdots$}}

\end{picture}

This means that each vertex $v$ weighted by $-w_0(v)=-n$ in the root denotes a free summand $\Z=\Z\langle 1_v\rangle\in (\bH_0)_{-2w_0(v)}$ (i.e. a connected component of $S_{n}$), and if $[v,u]$ is an  edge
connecting the vertices $v$ and  $u$ with $w_0(v)=w_0(u)+1$, then $U(1_u)=1_v\in (\bH_0)_{-2w_0(v)}$
(i.e. the edges codify the corresponding inclusions of the connected components of $S_{n-1}$ into the connected components of $S_n$).
The weighs  $-w_0$ are marked on the left of the graph. (See also subsection \ref{ss:grroot}.)

The graded $\Z[U]$--modules  ${\rm F}_{-d}\bH_0(\frX)$ for $d>0$ are illustrated below.

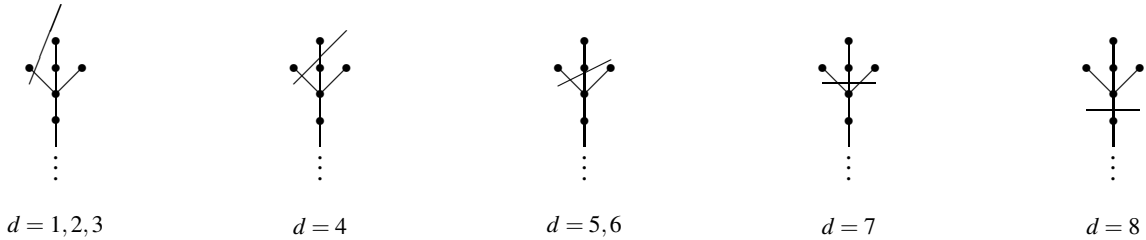
\begin{picture}(500,100)(20,310)

\put(20,345){\makebox(0,0){$\vdots$}} \put(20,360){\circle*{3}}
\put(20,390){\circle*{3}} \put(10,380){\circle*{3}}
\put(30,380){\circle*{3}} \put(20,380){\circle*{3}}
\put(20,370){\circle*{3}} \put(20,390){\line(0,-1){40}}
\put(20,370){\line(1,1){10}} \put(10,380){\line(1,-1){10}}

\put(120,345){\makebox(0,0){$\vdots$}} \put(120,360){\circle*{3}}
\put(120,390){\circle*{3}} \put(110,380){\circle*{3}}
\put(130,380){\circle*{3}} \put(120,380){\circle*{3}}
\put(120,370){\circle*{3}} \put(120,390){\line(0,-1){40}}
\put(120,370){\line(1,1){10}} \put(110,380){\line(1,-1){10}}

\put(220,345){\makebox(0,0){$\vdots$}} \put(220,360){\circle*{3}}
\put(220,390){\circle*{3}} \put(210,380){\circle*{3}}
\put(230,380){\circle*{3}} \put(220,380){\circle*{3}}
\put(220,370){\circle*{3}} \put(220,390){\line(0,-1){40}}
\put(220,370){\line(1,1){10}} \put(210,380){\line(1,-1){10}}

\put(320,345){\makebox(0,0){$\vdots$}} \put(320,360){\circle*{3}}
\put(320,390){\circle*{3}} \put(310,380){\circle*{3}}
\put(330,380){\circle*{3}} \put(320,380){\circle*{3}}
\put(320,370){\circle*{3}} \put(320,390){\line(0,-1){40}}
\put(320,370){\line(1,1){10}} \put(310,380){\line(1,-1){10}}

\put(420,345){\makebox(0,0){$\vdots$}} \put(420,360){\circle*{3}}
\put(420,390){\circle*{3}} \put(410,380){\circle*{3}}
\put(430,380){\circle*{3}} \put(420,380){\circle*{3}}
\put(420,370){\circle*{3}} \put(420,390){\line(0,-1){40}}
\put(420,370){\line(1,1){10}} \put(410,380){\line(1,-1){10}}

\put(20,320){\makebox(0,0){\footnotesize{$d=1,2,3$}}}
\put(120,320){\makebox(0,0){\footnotesize{$d=4$}}}
\put(220,320){\makebox(0,0){\footnotesize{$d=5,6$}}}
\put(320,320){\makebox(0,0){\footnotesize{$d=7$}}}
\put(420,320){\makebox(0,0){\footnotesize{$d=8$}}}

\put(10,374){\line(2,5){12}}
\put(110,374){\line(1,1){20}}
\put(210,373){\line(2,1){20}}
\put(310,374){\line(1,0){20}}
\put(410,364){\line(1,0){20}}
\end{picture}
 From the root of $\bH_0$ one has to delete those edges which intersect
the `cutting line'. The $\Z[U]$--module  ${\rm F}_{-d}\bH_0(\frX)$ sits below the cutting line,
 where the $U$--action is determined from the remaining edges by the principle described above.
For $d\geq 8$ the cutting line moves down  one by one.

\end{example}

\subsection{The homological spectral sequence associated with the subspaces $\{S_n\cap \frX_{-d}\}_d$.}\label{ss:ss} \

For any fixed $n$, by an argument as in the proof of
Theorem \ref{9STR1},
the morphism  $(\bH_*(\frX_{-d}, w))_{-2n}\to (\bH_*(\frX,w))_{-2n}$ is identical with
the morphisms $  H_*(S_n\cap \frX_{-d}, \Z)\to  H_*(S_n,\Z)$ induced by the inclusion
$S_n\cap \frX_{-d}\hookrightarrow S_n$. In particular, for any fixed $n$ one can analyse the spectral sequence
associated with the filtration $\{S_n\cap \frX_{-d}\}_{d\geq 0}$ of subspaces of $S_n$.
Since $S_n$ is compact, this filtration is finite.

The spectral sequence will be denoted by
 $(E^k_{-d,q})_n\Rightarrow (E^\infty_{-d,q})_n$. Its terms are the following:
  \begin{equation}\begin{split}
  (E^1_{-d,q})_n=& H_{-d+q}(S_n\cap \frX_{-d}, S_n\cap \frX_{-d-1},\Z),\\
   (E^\infty_{-d,q})_n=& \frac{(F_{-d}\, \bH_{-d+q}(\frX))_{-2n}}
   { (F_{-d-1}\, \bH_{-d+q}(\frX))_{-2n}}=({\rm Gr}^F_{-d}\, \bH_{-d+q}(\frX)\,)_{-2n}.
   \end{split}\end{equation}

 We wish to emphasize that each $\Z$--module $(E^k_{-d,q})_{n}$ is well-defined (in the sense that in its definition
 does not depend on any choice or additional construction which might produce a certain ambiguity);
  it is an invariant of
 the curve $(C,o)$.  In fact, this is true for all the spaces $\{S_n\}_n$ as well.

In this way, for any $k\geq 1$ we associate with $(C,o)$
the  triple graded $\Z$--module $(E^k_{-d,q})_n$.
 Moreover,  for every $1\leq k\leq \infty$,
 we have the  well-defined Poincar\'e series associated with $(C,o)$:
 $$PE_k(T,Q,h):=\sum_{d,q,n} \ \rank (E_{-d,q}^k)_{n}\cdot T^dQ^nh^{-d+q}\in\Z[[T,Q]][Q^{-1},h].$$
 Note also  that all the coefficients of $PE_k(T,Q,h)$ are nonnegative.

 The differential $d_{-d,q}^k$ acts as $(E_{-d,q}^k)_{n}\to (E_{-d-k,q+k-1}^k)_{n}$,
 hence $PE_{k+1}$ is obtained from $PE_k$ by deleting terms of type
 $Q^n(T^dh^{-d+q}+T^{d+k}h^{-d+q-1})=Q^nT^dh^{-d+q-1}(h+T^k)$ ($k\geq 1$).

 We say that two series $P$ and $P'$ satisfies
 $P\geq P'$ if and only if $P-P'$ has all of its coefficients nonnegative. Then the above discussion shows that
 (cf. \cite[page 15]{McCleary})
 $$PE_1(T,Q,h)\geq PE_2(T,Q,h)\geq \cdots \geq PE_\infty(T,Q,h).$$
 If $E^2_{*,*}=E^{\infty}_{*,*}$ then $PE_1-PE_2=(h+T)R^+$, where all the coefficients of $R^+$ are nonnegative.
 In general,  $(PE_1-PE_{\infty})|_{T=1}=(h+1)\bar{R}^+$, where all the coefficients of $\bar{R}^+$ are nonnegative.
 Thus,
 \begin{equation}\label{eq:spseq}
  PE_1(T,Q,h)|_{T=1,h=-1}= PE_2(T,Q,h)|_{T=1,h=-1}= \cdots = PE_\infty(T,Q,h)|_{T=1,h=-1}.\end{equation}
Since the whole  spectral sequence is an invariant of $(C,o)$,
$$k_{(C,o)}(n):=\min\{k\,:\, (E_{*,*}^k)_n=(E_{*,*}^\infty)_n\,\} \ \mbox{and } \ \ \
k_{(C,o)}=\max_n \,k_{(C,o)}(n)$$ are  invariants of $(C,o)$ as well.

\subsection{The term $E^\infty_{*,*}$.}

First note that for any fixed $b$ and $n$ one has
\begin{equation}\label{eq:pqr}
\sum_{-d+q=b}{\rm rank}\big( E_{-d, q}^\infty\big)_{n}={\rm rank }\, H_b(S_n, \Z)={\rm rank} \,(\bH_b(C,o))_{-2n}.
\end{equation}
\begin{proposition}\label{prop:infty} \

\noindent (a) $PE_\infty(1, Q,h)=PH(Q,h)=\sum_{n\geq m_w}\, \big(\, \sum_b \, {\rm rank}\, H_b(S_n,\Z)\,h^b\, \big)\cdot Q^n.$

\noindent (b) $PE_\infty(1, Q,-1)=\sum_{n\geq m_w}\, \chi_{top}(S_n)\cdot Q^n$ \ (where $\chi_{top}$ denoted the topological Euler characteristic)

\noindent (c) Let $R$ be any  rectangle of type $R(0,c')$ with $c'\geq c$ (where $c$ is the conductor). Then
$$PE_\infty(1, Q,-1)=\frac{1}{1-Q}\cdot \sum_{\square_q\subset R}\, (-1)^q \, Q^{w(\square_q)}.$$
(d) $$\lim_{Q\to 1}\Big( PE_\infty (1,Q, -1)-\frac{1}{1-Q}\Big)=eu(\bH_*(C,o))=\delta(C,o).$$
(e) $$PE_\infty(1, Q,h)\geq PE_\infty(1, Q,h=0)\geq \frac{1}{1-Q}$$  and \
 $PE_\infty(1, Q,h=0)-1/(1-Q)$ is finitely supported.
\end{proposition}
\begin{proof}
{\it (a)} follows from (\ref{eq:pqr}), while {\it (b)} from {\it (a)}. Next we prove {\it (c)}.

Set  $Eu(Q):=\sum_{\square_q\subset R}\, (-1)^q \, Q^{w(\square_q)}\in \Z[Q,Q^{-1}]$ and write
$Eu(Q)/(1-Q)$ as $\sum _{n\geq m_w} a_nQ^n$.
Then
$$a_n=\sum_{\square_q\subset R,\, w(\square_q)\leq n}\,(-1)^q=\chi_{top}(S_n\cap R).$$
But $S_n\cap R\hookrightarrow S_n$ is a homotopy equivalence, cf. \ref{cor:EUcurves}. Then use part {\it (b)}.

For {\it (d)} use {\it (c)}, $\sum_{\square_q\subset R}(-1)^{q}=1$
and $eu(\bH_*(C,o))=\sum_{\square_q\subset R}(-1)^{q+1}w(\square_q)$, cf. (\ref{eq:eu}).
For {\it (e)} use Th. \ref{cor:EUcurves}.
\end{proof}
\begin{remark}\label{rem:hat}
From part {\it (b)-- (c)} of the above proposition
$$\sum_{\square_q\subset R}\, (-1)^q \, Q^{w(\square_q)}=\sum_n\chi_{top}(S_n)(Q^n-Q^{n+1})=\sum_n \chi_{top}(S_n,S_{n-1})
Q^n=\sum_n\sum_b (-1)^b {\rm rank}(\hat{\bH}_b)_{-2n}Q^n.
$$
This, in fact  says, that the bigraded $\hat{\bH}$ is the categorification of
$\sum_{\square_q\subset R}\, (-1)^q \, Q^{w(\square_q)}$.

The above identity  for $Q=1$ gives (\ref{eq:hateu}).
\end{remark}

\subsection{The case of irreducible curves.}\label{ss:irredu}

If $r=1$ then $\frX=[0,\infty)$, and any non-empty $S_n $ (for $n\geq m_w$) is a union of intervals
$\cup_{\lambda  \in\Lambda } [a_\lambda, b_\lambda]$. Since $w(b_\lambda+1)>n$ (and, in general, $w(l+1)-w(l)\in\{1,-1\}$),
we obtain that
$w(b_\lambda)=n$  and $b_\lambda\in\calS$. In particular,
$(E^1_{-d,q})_{n}=H_{-d+q}(S_n\cap \frX_{-d}, S_n\cap \frX_{-d-1},\Z)$ is nonzero if and only if
$\Lambda\not=\emptyset$, $d=q$, and  $d=b_\lambda$ for some $\lambda$. Moreover, the differential of the spectral sequence
$d_{*,*}^k=0$ for $k\geq 1$.
Hence,
$$PE_1(T,Q,h)=PE_{\infty}(T,Q,h)=
\sum_{s\in \calS} T^s\,Q^{w(s)}\,h^0= \sum_{s\in \calS} T^s\,Q^{w(s)}.$$
This can be compared with the motivic Poincar\'e series associated with $(C,o)$ (cf. \cite[2.3]{Gorsky}):
$$P^m(t,q)=\sum_{s\in\calS}\, t^s\, q^{\hh(s)}.$$
Since $w(s)=2\hh(s)-s$,  we obtain for any $1\leq k\leq \infty$:
\begin{equation}\label{eq:pme}
PE_k(T,Q)|_{T=t\sqrt{q}, \ Q=\sqrt{q}}= P^m(t,q).
\end{equation}
Since $P^m(t, q=1)=P(t)$, we also obtain
$PE_k(T=t,Q=1,h)= P(t)$.

Since $P^m(t,q)$ is a rational function of type $\overline{P^m}(t,q)/(1-tq)$, cf.
\ref{prop:motProp}, a similar property should hold for $PE_k$ too. Indeed, since for $s\geq c$ one has $s\in\calS$ and
$w(s)=w(c)+s-c=(c-2\delta)+s-c=s-2\delta$, we obtain
$$PE_k(T,Q,h)=\sum_{s\in\calS, \, s<c}T^sQ^{w(s)}+\sum_{s\geq c}T^sQ^{s-2\delta}=\sum_{s\in\calS, \, s<c}T^sQ^{w(s)}
+\frac{T^cQ^{c-2\delta}}{1-TQ}.$$
Let us write $PE_k(T,Q,h)$ as $\overline{PE_k}(T,Q,h)/(1-TQ)$ with $\overline{PE_k}\in\Z[T,Q,Q^{-1}]$.

If $(C,o)$ is Gorenstein then $c=2\delta$ above, and the symmetry \ref{prop:motProp}{\it (c)} reads as
$$\overline{PE_k}(T^{-1}, Q)=T^{-c}\cdot \overline{PE_k}(T,Q).$$
That is, $\overline{PE_k}$ is a polynomial of degree $c$ in $T$.
Here are some examples for some torus knots:

$\calS=\langle 2,3\rangle$: \ \ \ \ \ \ \ \ \ \ \  $\overline{PE_k}(T,Q,h)=1-TQ+T^2$;

$\calS=\langle 3,4\rangle$: \ \ \ \ \ \ \ \ \ \ \  $\overline{PE_k}(T,Q,h)=1-TQ+T^3Q^{-1}-T^5Q+T^6$;

$\calS=\langle 2,2m+1\rangle$: \ \  $\overline{PE_k}(T,Q,h)=(1+t^2+\cdots +T^{2m-2})(1-TQ)+T^{2m}$;

$\calS=\langle 3,3m+1\rangle$: \ \ $\overline{PE_k}(T,Q,h)=\Big[\frac{1-T^{3m}Q^{-m}}{1-T^3Q^{-1}}+T^{3m}Q^{-m}
(1+TQ)\frac{ 1-T^{3m}Q^{m}}{1-T^3Q}\Big](1-TQ)+T^{6m}.$

\subsection{The case of plane curve singularity $x^2+y^2=0$}\label{ss:22}

In this case $r=2$, $\delta=1$, $c=(1,1)$, $\calS=\{(0,0)\}\cup \{(1,1)+(\Z_{\geq 0})^2\}$.
The Hilbert finction $\hh$, the weight function $w$ in the rectangle $R((0,0), (4,4))$,
and the homological graded root are the following

\begin{picture}(320,80)(-50,-20)
\put(-40,30){\makebox(0,0){$\hh:$}}
\put(140,30){\makebox(0,0){$w:$}}

\put(-15,0){\line(1,0){90}}
\put(-5,-10){\line(0,1){60}}

\put(5,-5){\makebox(0,0){\small{$0$}}}
\put(20,-5){\makebox(0,0){\small{$1$}}}
\put(35,-5){\makebox(0,0){\small{$2$}}}
\put(50,-5){\makebox(0,0){\small{$3$}}}
\put(65,-5){\makebox(0,0){\small{$4$}}}

\put(-10,5){\makebox(0,0){\small{$0$}}}
\put(-10,15){\makebox(0,0){\small{$1$}}}
\put(-10,25){\makebox(0,0){\small{$2$}}}
\put(-10,35){\makebox(0,0){\small{$3$}}}
\put(-10,45){\makebox(0,0){\small{$4$}}}

\put(5,5){\makebox(0,0){\small{$0$}}}
\put(5,15){\makebox(0,0){\small{$1$}}}
\put(5,25){\makebox(0,0){\small{$2$}}}
\put(5,35){\makebox(0,0){\small{$3$}}}
\put(5,45){\makebox(0,0){\small{$4$}}}

\put(20,5){\makebox(0,0){\small{$1$}}}
\put(20,15){\makebox(0,0){\small{$1$}}}
\put(20,25){\makebox(0,0){\small{$2$}}}
\put(20,35){\makebox(0,0){\small{$3$}}}
\put(20,45){\makebox(0,0){\small{$4$}}}

\put(35,5){\makebox(0,0){\small{$2$}}}
\put(35,15){\makebox(0,0){\small{$2$}}}
\put(35,25){\makebox(0,0){\small{$3$}}}
\put(35,35){\makebox(0,0){\small{$4$}}}
\put(35,45){\makebox(0,0){\small{$5$}}}

\put(50,5){\makebox(0,0){\small{$3$}}}
\put(50,15){\makebox(0,0){\small{$3$}}}
\put(50,25){\makebox(0,0){\small{$4$}}}
\put(50,35){\makebox(0,0){\small{$5$}}}
\put(50,45){\makebox(0,0){\small{$6$}}}

\put(65,5){\makebox(0,0){\small{$4$}}}
\put(65,15){\makebox(0,0){\small{$4$}}}
\put(65,25){\makebox(0,0){\small{$5$}}}
\put(65,35){\makebox(0,0){\small{$6$}}}
\put(65,45){\makebox(0,0){\small{$7$}}}

\put(160,0){\line(1,0){90}}
\put(170,-10){\line(0,1){60}}


\put(180,5){\makebox(0,0){\small{$0$}}}
\put(180,15){\makebox(0,0){\small{$1$}}}
\put(180,25){\makebox(0,0){\small{$2$}}}
\put(180,35){\makebox(0,0){\small{$3$}}}
\put(180,45){\makebox(0,0){\small{$4$}}}

\put(195,5){\makebox(0,0){\small{$1$}}}
\put(195,15){\makebox(0,0){\small{$0$}}}
\put(195,25){\makebox(0,0){\small{$1$}}}
\put(195,35){\makebox(0,0){\small{$2$}}}
\put(195,45){\makebox(0,0){\small{$3$}}}

\put(210,5){\makebox(0,0){\small{$2$}}}
\put(210,15){\makebox(0,0){\small{$1$}}}
\put(210,25){\makebox(0,0){\small{$2$}}}
\put(210,35){\makebox(0,0){\small{$3$}}}
\put(210,45){\makebox(0,0){\small{$4$}}}

\put(225,5){\makebox(0,0){\small{$3$}}}
\put(225,15){\makebox(0,0){\small{$2$}}}
\put(225,25){\makebox(0,0){\small{$3$}}}
\put(225,35){\makebox(0,0){\small{$4$}}}
\put(225,45){\makebox(0,0){\small{$5$}}}

\put(240,5){\makebox(0,0){\small{$4$}}}
\put(240,15){\makebox(0,0){\small{$3$}}}
\put(240,25){\makebox(0,0){\small{$4$}}}
\put(240,35){\makebox(0,0){\small{$5$}}}
\put(240,45){\makebox(0,0){\small{$6$}}}

\put(320,40){\makebox(0,0){\small{$0$}}}
\dashline[60]{1}(270,40)(310,40)
\put(290,5){\makebox(0,0){$\vdots$}}
\put(290,20){\circle*{3}}
\put(290,30){\circle*{3}}
\put(300,40){\circle*{3}}
\put(280,40){\circle*{3}}
\put(290,30){\line(0,-1){20}}
\put(290,30){\line(1,1){10}}
\put(290,30){\line(-1,1){10}}

\end{picture}

Therefore,  $\bH_{\geq 1}=0$ and  $\bH_0=\calt^-_0\oplus \calt_0(1)$.

The spaces $S_n$ for $n=0,1,2,3$ are the following (the sequence of spaces can be continued easily):

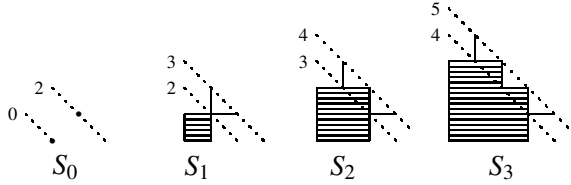
\begin{picture}(300,80)(-50,-10)

\put(10,10){\circle*{2}}\put(20,20){\circle*{2}}
\put(15,0){\makebox(0,0){$S_0$}}
\put(5,30){\makebox(0,0){\tiny{$2$}}}
\dashline[60]{1}(10,30)(30,10)
\put(-5,20){\makebox(0,0){\tiny{$0$}}}
\dashline[60]{1}(0,20)(10,10)

\put(60,10){\line(1,0){10}}\put(60,10){\line(0,1){10}}
\put(70,10){\line(0,1){20}}\put(60,20){\line(1,0){20}}
\put(60,12){\line(1,0){10}}\put(60,14){\line(1,0){10}}\put(60,16){\line(1,0){10}}\put(60,18){\line(1,0){10}}
\put(65,0){\makebox(0,0){$S_1$}}

\put(55,30){\makebox(0,0){\tiny{$2$}}}\put(55,40){\makebox(0,0){\tiny{$3$}}}
\dashline[60]{1}(60,30)(80,10)\dashline[60]{1}(60,40)(90,10)

\put(110,10){\line(1,0){20}}\put(110,10){\line(0,1){20}}
\put(130,10){\line(0,1){20}}\put(110,30){\line(1,0){20}}
\put(120,30){\line(0,1){10}}\put(130,20){\line(1,0){10}}
\put(110,12){\line(1,0){20}}\put(110,14){\line(1,0){20}}\put(110,16){\line(1,0){20}}
\put(110,18){\line(1,0){20}}\put(110,20){\line(1,0){20}}\put(110,22){\line(1,0){20}}\put(110,24){\line(1,0){20}}
\put(110,26){\line(1,0){20}}\put(110,28){\line(1,0){20}}
\put(120,0){\makebox(0,0){$S_2$}}
\put(105,40){\makebox(0,0){\tiny{$3$}}}\put(105,50){\makebox(0,0){\tiny{$4$}}}
\dashline[60]{1}(110,40)(140,10)\dashline[60]{1}(110,50)(145,15)

\put(160,10){\line(1,0){30}}\put(160,10){\line(0,1){30}}
\put(190,10){\line(0,1){20}}\put(160,40){\line(1,0){20}}
\put(180,30){\line(0,1){10}}\put(180,30){\line(1,0){10}}
\put(190,20){\line(1,0){10}}\put(170,40){\line(0,1){10}}
\put(160,12){\line(1,0){30}}\put(160,14){\line(1,0){30}}\put(160,16){\line(1,0){30}}
\put(160,18){\line(1,0){30}}\put(160,20){\line(1,0){30}}\put(160,22){\line(1,0){30}}\put(160,24){\line(1,0){30}}
\put(160,26){\line(1,0){30}}\put(160,28){\line(1,0){30}}
\put(160,30){\line(1,0){20}}\put(160,32){\line(1,0){20}}\put(160,34){\line(1,0){20}}
\put(160,36){\line(1,0){20}}\put(160,38){\line(1,0){20}}
\put(180,0){\makebox(0,0){$S_3$}}
\put(155,50){\makebox(0,0){\tiny{$4$}}}\put(155,60){\makebox(0,0){\tiny{$5$}}}
\dashline[60]{1}(160,50)(200,10)\dashline[60]{1}(160,60)(205,15)
\end{picture}

In the case of $S_0$ the two points are the lattice points $(0,0)$ and $(1,1)$, hence in their case the values $|l|$
are $d=0$ and $d=2$, hence $(E^1_{*,*})_0$ contributes in $PE_1$ with $Q^0(1+T^2)$.
In the case of $S_1$, for $d=2$ we have a relative 1-cycle, and for $d=3$ two relative 0-cycles.
Hence $(E^1_{*,*})_1$ (i.e. $S_1$)  contributes with $Q(T^2h^1+2T^3h^0)$.
Similarly, the contribution of $S_2$ is $Q^2(2T^3h^1+3T^4h^0)$,
of $S_3$ is $Q^3(3T^4h^1+4T^5h^0)$, and so on.  Therefore,
$$PE_1(T,Q,h)=1+T^2(1+2TQ+3T^2Q^2+\cdots )+T^2Qh^1(1+2TQ+3T^2Q^2+\cdots )=1+\frac{T^2(1+Qh)}{(1-TQ)^2}.$$
  The entries of the page $(E^1_{*,*})_{n}$ for $n=0,1,2$ are the following:

  
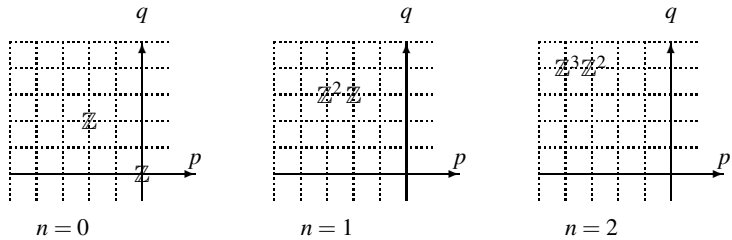
\begin{picture}(300,100)(-40,-20)

  \put(10,10){\vector(1,0){70}}\put(60,0){\vector(0,1){60}}
  \dashline[200]{1}(50,0)(50,60)\dashline[200]{1}(40,0)(40,60)\dashline[200]{1}(30,0)(30,60)
  \dashline[200]{1}(20,0)(20,60)\dashline[200]{1}(10,0)(10,60)
   \dashline[200]{1}(10,20)(70,20)\dashline[200]{1}(10,30)(70,30)\dashline[200]{1}(10,40)(70,40)
  \dashline[200]{1}(10,50)(70,50)\dashline[200]{1}(10,60)(70,60)
  \put(60,70){\makebox(0,0){\footnotesize{$q$}}} \put(80,15){\makebox(0,0){\footnotesize$p$}}
   \put(30,-10){\makebox(0,0){\footnotesize{$n=0$}}}
    \put(60,10){\makebox(0,0){\footnotesize{$\Z$}}}
    \put(40,30){\makebox(0,0){\footnotesize{$\Z$}}}

     \put(110,10){\vector(1,0){70}}\put(160,0){\vector(0,1){60}}
  \dashline[200]{1}(150,0)(150,60)\dashline[200]{1}(140,0)(140,60)\dashline[200]{1}(130,0)(130,60)
  \dashline[200]{1}(120,0)(120,60)\dashline[200]{1}(110,0)(110,60)
   \dashline[200]{1}(110,20)(170,20)\dashline[200]{1}(110,30)(170,30)\dashline[200]{1}(110,40)(170,40)
  \dashline[200]{1}(110,50)(170,50)\dashline[200]{1}(110,60)(170,60)
  \put(160,70){\makebox(0,0){\footnotesize{$q$}}} \put(180,15){\makebox(0,0){\footnotesize$p$}}
   \put(130,-10){\makebox(0,0){\footnotesize{$n=1$}}}
    \put(140,40){\makebox(0,0){\footnotesize{$\Z$}}}
    \put(131,41){\makebox(0,0){\footnotesize{$\Z^2$}}}

   \put(210,10){\vector(1,0){70}}\put(260,0){\vector(0,1){60}}
  \dashline[200]{1}(250,0)(250,60)\dashline[200]{1}(240,0)(240,60)\dashline[200]{1}(230,0)(230,60)
  \dashline[200]{1}(220,0)(220,60)\dashline[200]{1}(210,0)(210,60)
   \dashline[200]{1}(210,20)(270,20)\dashline[200]{1}(210,30)(270,30)\dashline[200]{1}(210,40)(270,40)
  \dashline[200]{1}(210,50)(270,50)\dashline[200]{1}(210,60)(270,60)
  \put(260,70){\makebox(0,0){\footnotesize{$q$}}} \put(280,15){\makebox(0,0){\footnotesize$p$}}
   \put(230,-10){\makebox(0,0){\footnotesize{$n=2$}}}
    \put(231,51){\makebox(0,0){\footnotesize{$\Z^2$}}}
    \put(221,51){\makebox(0,0){\footnotesize{$\Z^3$}}}

  \end{picture}

Note that for different $n$'s the corresponding spectral sequences do not interfere, they run independently.
 Since $\bH_{\geq 1}=0$, $E_{-d,q}^\infty=0$ for $-d+q>0$, hence $(E^1_{-1-n, 2+n})_{n}=\Z^n$ should be injected  by the
 differential $(d^1_{-1-n, 2+n})_{n}$.
 Then $(E_{*,*}^2)_n=(E_{*,*}^\infty)_n$, and
 $$PE_\infty(T,Q,h)=1+T^2 +T^3Q+T^4Q^2+\cdots = 1+\frac{T^2}{1-TQ},$$
 compatibly  with the filtration $\{{\rm Gr}_{-d}^F \,\bH_0(C,o)\}_d$ of $\bH_0(C,o)$ (which can be determined
  similarly as the filtration in Example \ref{ex:34}).

This example  shows that in general it can happen that $(E_{*,*}^1)_n\not = (E_{*,*}^\infty)_n$, i.e. $k_{(C,o)}(n)\not=1$.
Also, even if $\bH_{\geq m}(C,o)=0$ for some $m$, the first page might have nonzero terms with $-d+q\geq m$.

Let us return back to $PE_1(T,Q,h)$ and consider
$$\overline{PE_1}(T,Q,h)=(1-TQ)^2+T^2(1+Qh).$$
This can be compared with the motivic multivariable Poinvar\'e polynomial of $(C,o)$,
$\overline{P^m}(t_1,t_2,q)=1-qt_1-qt_2+qt_1t_2 $, cf. \cite[5.1]{Gorsky}.
Indeed
$$\overline{PE_1}(T,Q,h)|_{T\to t\sqrt{q}, \ Q\to \sqrt{q}, \ h\to -\sqrt{q}}=\overline{P^m}(t,t,q).$$
This will be generalized and proved for any $(C,o)$ in Theorem \ref{th:PP}.

 Above all  the differentials have `maximal ranks', but this is not the case in general, see Example \ref{ex:decsing2}.

\begin{remark}\label{rem:3.5.1}
The graded  lattice homology, or $PE_\infty(T,Q,h)$, contains several additional information compared with
the lattice homology
$\bH_*(C,o)$. E.g., in the irreducible case $PE_\infty(T,Q,h)=\sum_{s\in\cS}T^sQ^{w(s)}$, hence
from $PE_\infty$ one can recover the semigroup $\cS$. However, this is not possible from $\bH_*$.

Take for example the semigroups $\langle 4,5,7\rangle$ and $\langle 3,7,8\rangle$.
They have different $PE_k$ (for all $1\leq k\leq \infty$) but
they have identical $\Z[U]$--module   lattice homology $\bH_0=\et^-_{4}\oplus \et_{2}(1)\oplus \et_{0}(1)$.
In particular, for the two cases the series $PE_k(T=1,Q)$ are the same as well.
Merely in  $\bH_*$ (or in $PE_1(T=1,Q)$)
the values of the level filtration are not coded. (For more see also Example \ref{ex:irredY1}.)
\end{remark}

\subsection{The spectral sequence as decoration of $\mathfrak{R}(C,o)$.}\label{ss:decroot}

Above, for any fixed $n\geq m_w$, we considered the space $S_n$, its filtration
$\{S_n\cap\frX_{-d}\}_d$ and the homological spectral sequence $(E^k_{*,*})_n$.
We can define the generating function of the corresponding ranks
$$PE_k(T,h)_n =\sum_{d,q} \,{\rm rank}(E^k_{-d,q})_nT^dh^{-d+q}\in \Z[[T]][h],$$
which satisfies $\sum_n  PE_k(T,h)_n Q^n=PE_k(T,Q,h)$.

Now, we can proceed as in subsection \ref{ss:grroot}: we can replace the $\Z$--grading given by $n$
(or, the generating functions indexed by $n$) by an index set given by the vertices of the graded root.

Indeed, consider the connected components $\{S^v_n\}_v$ of $S_n$, where $v$ runs over the vertices of
${\mathfrak{R}}$ with $w_0(v)=n$. Then one can consider its filtration
$\{S_n^v\cap\frX_{-d}\}_d$ and the corresponding homological spectral sequence $(E^k_{*,*})^v_n$.
In this way the spectral sequence $\{(E^k_{*,*})^v_n\}_{k\geq 1}$, indexed by $v\in\cV(\mathfrak{R})$,
 appears as the decoration of the (vertices of the) graded root $\mathfrak{R}$.

 Clearly, $\oplus_{v:w_0(v)=n}\, (E^k_{*,*})^v_n=(E^k_{*,*})_n$. At Poincar\'e series level
 we have generating functions
 $$PE_k(T,h)^v_{w_0(v)}= \sum _{d,q}{\rm rank}(E^k_{-d,q})^v_{w_0(v)}T^dh^{-d+q}\in \Z[[T]][h],$$
 with $\sum_{v\,:\, w_0(v)=n}PE_k(T,h)^v_{w_0(v)}=PE_k(T,h)_n$.

 The decoration of $\cV(\mathfrak{R})$ by the $E_1$ pages
  can be improved even more. In section \ref{s:LFilt} we will consider a lattice filtration
 and the corresponding series ${\bf PE}_1({\bf T},Q,h)=\sum_n {\bf PE}_1({\bf T},h)_nQ^n$
  and  ${\bf PE}_1(T_1=T, \ldots, T_r=T,h)_n=PE_1(T,h)_n$.
   Similarly as above, each ${\bf PE}_1({\bf T},h)_n$ has a decomposition
 ${\bf PE}_1({\bf T},h)_n=\sum _{v\,:\, w_0(v)=n}{\bf PE}_1({\bf T},h)^v_{w_0(v)}$, where
$ {\bf PE}_1({\bf T},h)^v_{w_0(v)}$ is indexed by the vertices of $\mathfrak {R}$.

 \section{The lattice filtration and the multigraded  $E_{*,*}^1$.}\label{s:LFilt}

 \subsection{The improved first page of the spectral sequence.} \

Recall that for any fixed $n\geq m_w$ the spectral sequence $E_{*,*}^k$ ($k\geq 0$)
associated with the {\it level filtration}  $\{S_n\cap \frX_{-d}\}_{d\geq 0}$  of $S_n$
has its first terms
 $$(E^0_{-d,q})_{n}=\calC_{-d+q}(S_n\cap \frX_{-d}, S_n\cap \frX_{-d-1}),\ \ \
 (E^1_{-d,q})_{n}=H_{-d+q}(S_n\cap \frX_{-d}, S_n\cap \frX_{-d-1},\Z).$$
 The point is that $\calC_{-d+q}(S_n\cap \frX_{-d}, S_n\cap \frX_{-d-1})$  is generated by
 cubes of the form $\square =(l, I)$ with $|l|=d$ and $w((l,I))\leq n$. This automatically provides a direct sum decomposition
 as follows.

 For any $l=\sum_il_iE_i\geq 0$ define $\frX_{-l}=\prod_i [l_i,\infty)$  with its cubical decomposition
 $\cup\{(l',I)\,:\, l'\geq l,\ I\subset\cV\}$. Then
 \begin{equation}\label{eq:sum}
 (E^1_{-d,q})_{n}= \bigoplus_{l\in\Z^r_{\geq 0},\, |l|=d}\
 (E^1_{-l,q})_{n}, \ \mbox{where} \ (E^1_{-l,q})_{n}:=
 H_{-d+q}(S_n\cap \frX_{-l}, S_n\cap \frX_{-l}\cap \frX_{-d-1},\Z).\end{equation}
 Accordingly to this direct sum decomposition  we also define
 $${\bPE}_1({\bf T}, Q,h)={\bPE}_1(T_1, \ldots , T_r, Q,h):=\sum_{l\in\Z^r_{\geq 0},\, n,q}\ \rank\big((E^1_{-l,q})_{n}\big)\cdot
 T_1^{l_1}\cdots T_r^{l_r}\, Q^n\, h^{-|l|+q}.$$
 Clearly, ${\bPE}_1(T_1=T,\ldots, T_r=T, Q,h)=PE_1(T,Q,h)$.

\begin{example}\label{ex:22}
Assume that $(C,o)$ is the plane curve singularity $x^2+y^2=0$, cf. \ref{ss:22}. Then using the
shape of the spaces $\{S_n\}_{n\geq 0}$ from \ref{ss:22}, we deduce
$$\bPE_1(T_1,T_2, Q,h)= 1+\frac{T_1T_2(1+Qh)}{(1-T_1Q)(1-T_2Q)}=
\frac{\overline{\bPE_1}(T_1,T_2, Q,h)}{(1-T_1Q)(1-T_2Q)}.$$
Then
$$\overline{\bPE_1}(T_1,T_2, Q,h)|_{T_1\to t_1\sqrt{q},\ T_2\to t_2\sqrt{q},\ Q\to \sqrt{q},\ h\to -\sqrt{q}}=
1-qt_1-qt_2+qt_1t_2 = \overline{P^m}(t_1,t_2,q).$$
\end{example}
Our goal is to extend  this relation  for any $(C,o)$.
More precisely, we will prove that $\bPE_1({\bf T}, Q,h)$ and the multivariable motivic Poincar\'e
series $P^m({\bf t}, q)$ determine each other.

\begin{theorem}\label{th:PP} (a) For any fixed $l$ one has the isomorphism
\begin{equation}\label{eq:PP}
 H_{b}(S_n\cap \frX_{-l}, S_n\cap \frX_{-l}\cap\frX_{-|l|-1})=
 H_{-2n+2\hh(l)-2|l|, b}({\rm gr}_l\calL^-,{\rm gr}_l\partial_U).\end{equation}
In particular, $H_{b}(S_n\cap \frX_{-l}, S_n\cap \frX_{-l}\cap\frX_{-|l|-1})$ has no $\Z$--torsion
(cf. Theorem \ref{zlat}{\it (2)}).

(b) If $H_{b}(S_n\cap \frX_{-l}, S_n\cap \frX_{-l}\cap\frX_{-|l|-1})\not=0$ then necessarily
$n=w(l)+b$.

(c) The next morphism (induced by the inclusion $S_n\hookrightarrow S_{n+1}$) is trivial:
$$ H_{b}(S_n\cap \frX_{-l}, S_n\cap \frX_{-l}\cap\frX_{-|l|-1})
\to H_{b}(S_{n+1}\cap \frX_{-l}, S_{n+1}\cap \frX_{-l}\cap\frX_{-|l|-1}).$$

(d)
$$\bPE_1({\bf T}, Q,h)|_{T_i\to t_i\sqrt{q},\ Q\to \sqrt{q},\ h\to -\sqrt{q}}= P^m({\bf t}, q).$$
In particular,
$$\bPE_1({\bf T}, Q,h)|_{T_i\to t_i,\ Q\to 1,\ h\to -1}= P({\bf t}).$$

(e) Write  $P^m(\bt,q)$ as $\sum_l\pp^m_l(q)\cdot\bt^l$.  Then each $\pp^m_l(q)$
can be written in a unique way in  the form
$$\pp^m_l(q)=\sum_{k\in\Z_{\geq 0} }\pp^m _{l,k}q^{k+\hh(l)}, \ \ (\pp^m_{l,k}\in\Z).$$
In fact,
\begin{equation}\label{eq:HLFUJ}
(-1)^k\pp^m_{l,k}={\rm rank}\,
 H_{-2\hh(l)-k}({\rm gr}_l\calL^-,{\rm gr}_l\partial_U).\end{equation}

(f) Write $P^m(\bt,q)=\sum_l\ \sum_{k\in\Z_{\geq 0} }\pp^m _{l,k}q^{k+\hh(l)}\cdot \bt^l$. Then
$$\bPE_1({\bf T}, Q,h)= \sum_l\ \sum_{k\in\Z_{\geq 0} }\pp^m _{l,k}\ {\bf T}^l Q^{w(l)}\cdot (-Qh)^k.$$

(g) In particular, $\bPE_1$ is a rational function of type
$$\overline{\bPE_1}({\bf T}, Q, h)/\prod_i(1-T_iQ), \ \ \mbox{where} \ \ \overline{\bPE_1}({\bf T}, Q, h)\in\Z[{\bf T}, Q, Q^{-1}, h].$$

(h) The ${\bf T}$-support of ${\bPE_1}$ is exactly $\calS$:  if we write
${\bPE_1}$ as $\sum_l {\mathfrak{ pe}}_l(Q,h){\bf T}^l$, then $\mathfrak{pe}_l\not\equiv 0$ if and only if \ $l\in\calS$.

(i) In the Gorenstein case, after substitution $h=-Q$ one has the symmetry
 $$\overline{\bPE_1}({\bf T}, Q, h=-Q)|_{T_i\to T_i^{-1}}=\prod_i T_i^{-c_i}\cdot \overline{\bPE_1}({\bf T}, Q, h=-Q).$$
\end{theorem}
\begin{proof} The proof of part {\it (a)} has a similar strategy as the proof of Theorem \ref{9STR1}.

Consider the complex  $({\rm gr}_l\calL^-,{\rm gr}_l\partial_U)$  generated over $\Z$ by elements  $U^m\square$, $\square=(l,I)$, $m\geq 0$, cf. \ref{bek:2.22}.
Each generator $U^m\square$\, has a bidegree $(a,b)=(-2m-2\hh(\square), \dim(\square))=
(-2m-2\hh(l+E_I), |I|)$ and ${\rm gr}_l\partial_U$ has a bidegree $(0,-1)$. So, if we consider the subcomplex $({\rm gr}_l\calL^-,{\rm gr}_l\partial_U)_{a,*}$
generated by elements $U^m\square$ with $a$ fixed and graded by $b$,
then $({\rm gr}_l\calL^-,{\rm gr}_l\partial_U)$  decomposes as a direct sum of subcomplexes
$\oplus_a({\rm gr}_l\calL^-,{\rm gr}_l\partial_U)_{a,*}$.

On the other hand,  for any fixed $n$, let $\calC_{n,*}
:=\calC_*(
S_n\cap \frX_{-l}, S_n\cap \frX_{-l}\cap\frX_{-|l|-1})$ be the homological (cubical) complex associated with the corresponding pair. It is generated over $\Z$ by cubes of type $\square=(l,I)\subset S_n$, $I\subset \cV$.
The boundary operator is $\partial ((l,I))=\sum_k (-1)^k (l, I\setminus \{k\})$.

In the next discussion we identify the two complexes using a  correspondence
between  the integers $a$ and $n$.  Recall that in all these discussions $l$ is fixed and $a$ is even.
We claim that under the bijective correspondence
$n:=-a/2+\hh(l)-|l|$ the complexes
$({\rm gr}_l\calL^-,{\rm gr}_l\partial_U)_{a,*}$ and $(\calC_{n,*},\partial)$ are identical.

First note that in the two complexes we use two different weight functions: in $\calC_{n,*}$ (that is, in the definition of $S_n$) we use the weights $w(\square)$,
while in the definition of ${\rm gr}_l\partial_U$ we use $\hh(\square)$. The next identity appeared as identity (7)
in the proof of Theorem 3.1.8 in \cite{{AgostonNemethi}}.  
  It shows that for
restricted  cubes of type $\{(l,I)\}_{I\subset \cV}$
the  `relative weights with respect to $l$' are the same.
\begin{equation}\label{eq:relweights}
w((l,I))-w(l)=\hh((l,I))-\hh(l).
\end{equation}
Next, for any generator
$U^m(l, I)$  of ${\rm gr}_l\calL^-$, the identity $n=-a/2+\hh(l)-|l|$ together with (\ref{eq:relweights})
transforms into $n=m+w((l,I))$. Since $m\geq 0$ we get $w((l,I))\leq n$,
hence $(l, I)\subset \frX_{-l}\cap S_n$. Conversely, any  $(l, I)\subset \frX_{-l}\cap S_n$
defines the  generator $U^m(l,I)$ of ${\rm gr}_l\calL^-$ with $m:= n-w((l,I))\geq 0$.

Finally we verify that this correspondence commute with the boundary operators. Indeed, for any
$U^m(l, I)$ with $m+w((l, I))=n$,
 ${\rm gr}_l\partial_U(U^m(l,I))$ equals $$\sum_k (-1)^k U^{m+\hh((l, I))-\hh((l, I\setminus \{k\}))}(l, I\setminus \{k\})\stackrel{(\ref{eq:relweights})}{=}
\sum_k (-1)^k U^{m+w((l, I))-w((l, I\setminus \{k\}))}(l, I\setminus \{k\})$$
which by the above correspondence coincides with $\partial ((l, I))=\sum_k(-1)^k(l, I\setminus \{k\})$
considered in $S_n$ since $m+w((l, I))-w((l, I\setminus \{k\}))+w((l, I\setminus \{k\}))=n$.

{\it (b)} Use part {\it (a)} and Theorem \ref{zlat}{\it (3)}.

{\it (c)} By part {\it (b)} at least one of the modules vanish.
This part {\it (c)} is compatible (via the proof of part {\it (a)}) with
  Theorem \ref{zlat}{\it (4)}.

{\it (d)}  Using parts {\it (a)} and {\it (b)} we obtain that
  $$\bPE_1({\bf T}, Q,h)=\sum_{l,n,q} {\rm rank}\, H_{-n-q, -|l|+q} ({\rm gr}_l\calL^-, {\rm gr}_l \partial _U)\,{\bf T}^l Q^nh^{-|l|+q}. $$
  Again, by part {\it (b)}, $l, n, q$ are related by the identity
  $n=w(l)-|l|+q=2\hh(l)-2|l|+q$. Therefore,
  $$\bPE_1({\bf T}, Q,h)=\sum_{l,n} {\rm rank}\, H_{-n-|l|} ({\rm gr}_l\calL^-, {\rm gr}_l \partial _U)\, {\bf T}^l Q^nh^{n+|l|-2\hh(l)}. $$
  After the corresponding substitutions, the right hand side transforms into
  $$\sum_{l, n}\bt^lq^{-\hh(l)}\ \sum_n (-q)^{n+|l|} {\rm rank}\, H_{-n-|l|} ({\rm gr}_l\calL^-, {\rm gr}_l \partial _U),$$
  which equals $\sum_l\bt^l\pp^m_l(q)$ by (\ref{hfminus}).

  {\it (e)} By  (\ref{hfminus})  $\pp^m_l(q)=q^{-\hh(l)}\ \sum_n (-q)^{n+|l|} {\rm rank}\, H_{-n-q, -|l|+q} ({\rm gr}_l\calL^-, {\rm gr}_l \partial _U)$ with $-|l|+q\geq 0$
  (cf. part {\it (a)} or Theorem \ref{zlat}{\it (3)}). Then use $n+|l|-\hh(l)=\hh(l)-|l|+q\geq \hh(l)$.

For {\it (f)} use the identities of the proof of {\it (d)} and {\it (e)};
for
{\it (g)} the above correspondence and Proposition \ref{prop:motProp}{\it (b)};
for
{\it (h)}  Proposition \ref{prop:motProp}{\it (c)}, and for
{\it (i)}   Proposition \ref{prop:motProp}{\it (d)} and $2\delta=|c|$.
\end{proof}

\begin{remark}\label{rem:hathat}
For any cube $\square=(l,I)$ set $d(\square)=d(l):=|l|$.
Note also that in Proposition \ref{prop:infty}{\it (c)} $PE_\infty(1,Q, -1)$ can be replaced by $PE_1(1,Q, -1)$,
cf. (\ref{eq:spseq}).
Then we have the following extensions of Proposition \ref{prop:infty}{\it (c)}.
$$PE_1(T, Q,h=-1)=\frac{1}{1-Q}\cdot \sum_{\square_q\subset \frX}\, (-1)^q \, Q^{w(\square_q)}T^{d(\square_q)},$$
$${\bf PE}_1({\bf T}, Q,h=-1)=\frac{1}{1-Q}\cdot \sum_{\square_q\subset \frX}\, (-1)^q \, Q^{w(\square_q)}{\bf T}^{l}.$$
The proof is similar: write
$(\sum_{\square_q}\, (-1)^q \, Q^{w(\square_q)}{\bf T}^{l})/(1-Q)$
as $\sum _{n,l} a_{n,l}Q^n{\bf T}^l$.
Then
$$a_{n,l}=\sum_{\square_q=(l,I):\ w(\square_q)\leq n}\,(-1)^{|I|}=
\chi_{top}(S_n\cap \frX_{-l}, S_n\cap\frX_{-l}\cap  \frX_{-|l|-1},\Z).$$
%
\end{remark}

\subsection{Example. The plane curve singularity $(C,o)=\{x^3+y^3=0\}$.}\label{ss:33}

The embedded link in $S^3$ consists of three Hopf link components, $r=3$. The conductor is $c=(2,2,2)$ and $\delta=3$.
Each irreducible component is smooth, the partial Poincar\'e  series/polynomials are
$1/1-t_i$  for $C_i$, 1 for $C_{i,j}$ and $1-t_1t_2t_3$ for $C$. In particular, via (\ref{hilbert}), the Hilbert series is
\begin{equation}\label{eq:hilb33}
H({\bf t})|_{\geq 0}=
\frac{1}{\prod_{i=1}^3(1-t_i)}\cdot
\Big(  \frac{1}{1-t_1}+ \frac{1}{1-t_2}+ \frac{1}{1-t_3}-t_1t_2-t_2t_3-t_3t_1+t_1t_2t_3(1-t_1t_2t_3)\Big).
\end{equation}
The semigroup $\calS$ can be computed directly, or via Lemma \ref{eq:semi}, it is
$$\calS=\{(0,0,0)\}\cup \{(l,1,1)\}_{l\geq 1}\cup \{(1,l,1)\}_{l\geq 1}\cup \{(1,1,l)\}_{l\geq 1}\cup
\{(l_1,l_2,l_3)\}_{l_1,l_2,l_3\geq 2}.
$$
The lattice homology $\bH_*(C,o)$ can be read from the $w$-weights of the rectangle $R(0,c)$, cf.
Theorem \ref{cor:EUcurves}. Since later we will need the $w$-weights of
$R(0, c+{\bf 1})$ too, here we provide that one:

\begin{picture}(380,85)(-30,-25)

\footnotesize{
\put(65,0){\makebox(0,0){$l_1$}}
\put(-5,50){\makebox(0,0){$l_2$}}

\put(-15,0){\vector(1,0){70}}
\put(-5,-10){\vector(0,1){50}}

\put(5,-5){\makebox(0,0){\small{$0$}}}
\put(20,-5){\makebox(0,0){\small{$1$}}}
\put(35,-5){\makebox(0,0){\small{$2$}}}
\put(50,-5){\makebox(0,0){\small{$3$}}}

\put(25,-20){\makebox(0,0){\small{$l_3=0$}}}

\put(-10,5){\makebox(0,0){\small{$0$}}}
\put(-10,15){\makebox(0,0){\small{$1$}}}
\put(-10,25){\makebox(0,0){\small{$2$}}}
\put(-10,35){\makebox(0,0){\small{$3$}}}

\put(5,5){\makebox(0,0){\small{$0$}}}
\put(5,15){\makebox(0,0){\small{$1$}}}
\put(5,25){\makebox(0,0){\small{$2$}}}
\put(5,35){\makebox(0,0){\small{$3$}}}

\put(20,5){\makebox(0,0){\small{$1$}}}
\put(20,15){\makebox(0,0){\small{$0$}}}
\put(20,25){\makebox(0,0){\small{$1$}}}
\put(20,35){\makebox(0,0){\small{$2$}}}

\put(35,5){\makebox(0,0){\small{$2$}}}
\put(35,15){\makebox(0,0){\small{$1$}}}
\put(35,25){\makebox(0,0){\small{$2$}}}
\put(35,35){\makebox(0,0){\small{$3$}}}

\put(50,5){\makebox(0,0){\small{$3$}}}
\put(50,15){\makebox(0,0){\small{$2$}}}
\put(50,25){\makebox(0,0){\small{$3$}}}
\put(50,35){\makebox(0,0){\small{$4$}}}

\put(85,0){\vector(1,0){70}}
\put(95,-10){\vector(0,1){50}}


\put(125,-20){\makebox(0,0){\small{$l_3=1$}}}


\put(105,5){\makebox(0,0){\small{$1$}}}
\put(105,15){\makebox(0,0){\small{$0$}}}
\put(105,25){\makebox(0,0){\small{$1$}}}
\put(105,35){\makebox(0,0){\small{$2$}}}

\put(120,5){\makebox(0,0){\small{$0$}}}
\put(120,15){\makebox(0,0){\small{$-1$}}}
\put(120,25){\makebox(0,0){\small{$0$}}}
\put(120,35){\makebox(0,0){\small{$1$}}}

\put(135,5){\makebox(0,0){\small{$1$}}}
\put(135,15){\makebox(0,0){\small{$0$}}}
\put(135,25){\makebox(0,0){\small{$1$}}}
\put(135,35){\makebox(0,0){\small{$2$}}}

\put(150,5){\makebox(0,0){\small{$2$}}}
\put(150,15){\makebox(0,0){\small{$1$}}}
\put(150,25){\makebox(0,0){\small{$2$}}}
\put(150,35){\makebox(0,0){\small{$3$}}}


\put(185,0){\vector(1,0){70}}
\put(195,-10){\vector(0,1){50}}


\put(225,-20){\makebox(0,0){\small{$l_3=2$}}}


\put(205,5){\makebox(0,0){\small{$2$}}}
\put(205,15){\makebox(0,0){\small{$1$}}}
\put(205,25){\makebox(0,0){\small{$2$}}}
\put(205,35){\makebox(0,0){\small{$3$}}}

\put(220,5){\makebox(0,0){\small{$1$}}}
\put(220,15){\makebox(0,0){\small{$0$}}}
\put(220,25){\makebox(0,0){\small{$1$}}}
\put(220,35){\makebox(0,0){\small{$2$}}}

\put(235,5){\makebox(0,0){\small{$2$}}}
\put(235,15){\makebox(0,0){\small{$1$}}}
\put(235,25){\makebox(0,0){\small{$0$}}}
\put(235,35){\makebox(0,0){\small{$1$}}}

\put(250,5){\makebox(0,0){\small{$3$}}}
\put(250,15){\makebox(0,0){\small{$2$}}}
\put(250,25){\makebox(0,0){\small{$1$}}}
\put(250,35){\makebox(0,0){\small{$2$}}}


\put(285,0){\vector(1,0){70}}
\put(295,-10){\vector(0,1){50}}


\put(325,-20){\makebox(0,0){\small{$l_3=3$}}}


\put(305,5){\makebox(0,0){\small{$3$}}}
\put(305,15){\makebox(0,0){\small{$2$}}}
\put(305,25){\makebox(0,0){\small{$3$}}}
\put(305,35){\makebox(0,0){\small{$4$}}}

\put(320,5){\makebox(0,0){\small{$2$}}}
\put(320,15){\makebox(0,0){\small{$1$}}}
\put(320,25){\makebox(0,0){\small{$2$}}}
\put(320,35){\makebox(0,0){\small{$3$}}}

\put(335,5){\makebox(0,0){\small{$3$}}}
\put(335,15){\makebox(0,0){\small{$2$}}}
\put(335,25){\makebox(0,0){\small{$1$}}}
\put(335,35){\makebox(0,0){\small{$2$}}}

\put(350,5){\makebox(0,0){\small{$4$}}}
\put(350,15){\makebox(0,0){\small{$3$}}}
\put(350,25){\makebox(0,0){\small{$2$}}}
\put(350,35){\makebox(0,0){\small{$3$}}}

}

\end{picture}

The Gorenstein symmetry of $R(0, c)$ with respect to $c/2=(1,1,1)$ is transparent.

From $(R(0,c), w)$ one reads $\bH_{>0}=0$, and $\bH_0=\calt^-_{2}\oplus \calt_0(1)^2$ is associated with the graded root
 identical with the root of the irreducible plane curve singularity $x^3+y^4=0$, given in
 \ref{ex:34}.

%
%

Using the weighted lattice points (or the spaces $S_n$, see below), we can also see the induced  level grading
$ \{{\rm F}_{-d}\bH_0\}_d$. It turns out that it coincides with the grading $ \{{\rm F}_{-d}\bH_0\}_d$
of the singularity $x^3+y^4=0$ shown in  \ref{ex:34}.
At $d=0$ the lattice point $(0,0,0)$ is `cut out', at $d=3$ the point $(1,1,1)$ with $w=-1$,
at $d=4$ the central cross of $S_0$ is left out, at $d=6$ the lattice point $(2,2,2)$ is cut, and so on.
Hence, in this way we obtain $PE_\infty(T,Q,h)$ 
$$PE_\infty (T,Q,h)=1+T^3Q^{-1}+T^4+ \frac{T^6}{1-TQ}.$$
Regarding of the coincidence $PE_\infty (T,Q,h)(x^3+y^3)=PE_\infty (T,Q,h)(x^3+y^4)$ (for the second one see \ref{ss:irredu})
note that there exists a $\delta$-constant deformation $x^3+ty^3+y^4=0$ connecting the two germs,
and we might expect some kind of stability along such deformation of certain invariants
(see e.g. \cite{AgostonNemethi}). However, the stability of ${\rm Gr}^F_*\bH_0$ is still surprising
(not just because of the jump of the number of irreducible components, but also because
 at the level of
semigroups we do not recognise the trace of an immediate stability).

On the other hand, all the other  invariants of the two germs $x^3+y^3$ and $x^3+y^4$ are very different.

Let us start with the motivic Poincar\'e series of $x^3+y^3$. It can be determined either  via (\ref{eq:pmot}) using  the
$\hh$-function given above in (\ref{eq:hilb33}), or from \cite[5.1]{Gorsky}:
$$P^m(\bt;q)=1+\frac{q(1-q)^2\cdot t_1t_2t_3\,-\,q^3t_1t_2t_3(1-t_1)(1-t_2)(1-t_3)}{(1-t_1q)(1-t_2q)(1-t_3q)}.$$

Then $\bPE_1 ({\bf T}, Q,h)$ can be deduced using Theorem \ref{th:PP}. Indeed, by part {\it (h)} of that theorem
we have to focus on the coefficients of the ${\bf T}^l$ for $l\in \calS$ only. For $l=(0,0,0)$ we get the contribution $1$.
In the case of $l=(1,1,1)$, the coefficient of $\bt^{(1,1,1)}$ in $P^m$ is $q(1-q)^2-q^3=q(1-2q)$.
Since $\hh(1,1,1)=1$ and $w(1,1,1)=-1$ from Theorem \ref{th:PP} we get the term
${\bf T}^{(1,1,1)}Q^{-1}(1+2Qh)$ in $\bPE_1$.

For the semigroup element $s=(l+1,1,1)$ ($l\geq 1$) the coefficient in $P^m$ is
$q(1-q)^2q^l-q^{3+l}+q^{2+l}=q^{l+1}(1-q)$.
Since $\hh(s)=l+1$ and $w(s)=l-1$, we get the contribution
$T_1^{l+1}T_2T_3Q^{l-1}(1+Qh)$  in $\bPE_1$.
The contribution for $(l_1,l_2, l_3)_{l_1,l_2,l_3\geq 2}$ is ${\bf T}^{(l_1,l_2,l_3)}Q^{l_1+l_2+l_3-6}(1+Qh)^2$.

Therefore, $\bPE_1({\bf T}, Q,h)$ is
\begin{equation}\label{eq:TTT}
1+T_1T_2T_3Q^{-1}(1+2Qh)+T_1T_2T_3\cdot (\sum_{i=1}^3\frac{T_i}{1-T_iQ}\,)\cdot (1+Qh)+
\frac{T_1^2T_2^2T_3^2(1+Qh)^2}{\prod_{i=1}^3 (1-T_iQ)}.
\end{equation}
This can be deduced from the filtrations  of the spaces $S_n$ as well. Let us do the first step together with the
study of the corresponding spectral sequences. Clearly $S_{<-1}=\emptyset$.

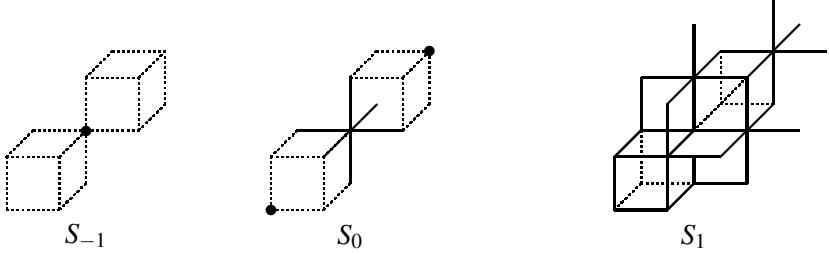
  \begin{picture}(300,100)(-20,-20)

  \dashline[200]{1}(0,0)(20,0)\dashline[200]{1}(0,0)(0,20)\dashline[200]{1}(20,0)(20,20)
   \dashline[200]{1}(0,20)(20,20)\dashline[200]{1}(20,0)(30,10)\dashline[200]{1}(20,20)(30,30)
  \dashline[200]{1}(0,20)(10,30)\dashline[200]{1}(10,30)(30,30)
   \dashline[200]{1}(30,10)(30,30)

    \dashline[200]{1}(30,30)(50,30)\dashline[200]{1}(30,30)(30,50)\dashline[200]{1}(50,30)(50,50)
   \dashline[200]{1}(30,50)(50,50)\dashline[200]{1}(50,30)(60,40)\dashline[200]{1}(50,50)(60,60)
  \dashline[200]{1}(30,50)(40,60)\dashline[200]{1}(40,60)(60,60)
   \dashline[200]{1}(60,40)(60,60)
   \put(30,-10){\makebox(0,0){$S_{-1}$}}
   \put(30,30){\circle*{4}}

    \dashline[200]{1}(100,0)(120,0)\dashline[200]{1}(100,0)(100,20)\dashline[200]{1}(120,0)(120,20)
   \dashline[200]{1}(100,20)(120,20)\dashline[200]{1}(120,0)(130,10)\dashline[200]{1}(120,20)(130,30)
  \dashline[200]{1}(100,20)(110,30)\dashline[200]{1}(110,30)(130,30)
   \dashline[200]{1}(130,10)(130,30)

    \dashline[200]{1}(130,30)(150,30)\dashline[200]{1}(130,30)(130,50)\dashline[200]{1}(150,30)(150,50)
   \dashline[200]{1}(130,50)(150,50)\dashline[200]{1}(150,30)(160,40)\dashline[200]{1}(150,50)(160,60)
  \dashline[200]{1}(130,50)(140,60)\dashline[200]{1}(140,60)(160,60)
   \dashline[200]{1}(160,40)(160,60)
    \put(130,-10){\makebox(0,0){$S_{0}$}}
   \put(100,0){\circle*{4}}\put(160,60){\circle*{4}}
   \thicklines
   \put(110,30){\line(1,0){40}} \put(120,20){\line(1,1){20}} \put(130,10){\line(0,1){40}}

 \put(230,0){\line(1,0){20}}\put(230,0){\line(0,1){20}}
  \put(230,20){\line(1,0){40}}\put(250,0){\line(0,1){40}}
   \put(230,20){\line(1,1){10}} \put(240,30){\line(1,0){10}}
    \put(250,20){\line(1,1){10}}\put(250,0){\line(1,1){10}}\put(260,10){\line(0,1){10}}

     \put(240,30){\line(0,1){20}}
      \put(240,50){\line(1,0){40}}

  \put(250,40){\line(1,1){20}}
  \put(260,30){\line(1,0){40}}
  \put(240,30){\line(1,0){20}}
  \put(260,30){\line(0,1){40}}\put(270,20){\line(1,1){20}}\put(280,10){\line(0,1){40}}\put(280,10){\line(-1,0){20}}
\put(270,60){\line(1,0){40}}\put(280,50){\line(1,1){20}}\put(290,40){\line(0,1){40}}

   \put(230,20){\line(1,1){10}} \put(240,30){\line(1,0){10}}
    \put(250,20){\line(1,1){10}}\put(250,0){\line(1,1){10}}\put(260,10){\line(0,1){10}}
\thinlines

     \dashline[200]{1}(230,0)(240,10)\dashline[200]{1}(240,10)(240,30)\dashline[200]{1}(240,10)(260,10)
        \dashline[200]{1}(260,30)(280,50)\dashline[200]{1}(270,40)(270,60)\dashline[200]{1}(270,40)(290,40)
      \put(260,-10){\makebox(0,0){$S_{1}$}}

    \end{picture}

$S_{-1}$ is the lattice point $(1,1,1)$, hence via (\ref{eq:sum}) the contribution in $\bPE_1$ is ${\bf T}^{(1,1,1)}Q^{-1}$.

The space $S_0$ is $\{(0,0,0)\}\cup \{(2,2,2)\}\cup \{(1,1,[0,2])\}\cup \{(1,[0,2], 1)\}\cup \{([0,2],1,1)\}$. The contribution is
$$Q^0\big(\ 1+2T_1T_2T_3 h+T_1T_2T_3(T_1+T_2+T_3)+T_1^2T_2^2T_3^2\,\big).$$
1 respectively $T_1^2T_2^2T_3^2$ are produced by the isolated lattice points
$(0,0,0)$ and $(2,2,2)$ of $S_0$. If we glue  the three
endpoints with $d=4$  of the central cross of $S_0$ into a single point then we create two 1--loops, this gives
the term $2T_1T_2T_3 h$; while $T_1T_2T_3(T_1+T_2+T_3)$ is given by the three ends considered before.

Next we consider the spectral sequence associated with $S_0$ and the level filtration (i.e. in the above formula
in order to get $PE_1(T,Q,h)$ we substitute $T_1=T_2=T_3=T$. I.e.,
$$PE_1(T,Q,h)=1+T^3Q^{-1}(1+2Qh)+3T^4\frac{1+Qh}{1-TQ}+T^6\frac{(1+Qh)^2}{(1-TQ)^3}.$$
The left diagram $E^1_{*,*}$ is given by the above expression (by the term of $Q^0$), while the right
$E^\infty_{*,*}$ is given by the $Q^0$ term $1+T^4+T^6$ of $PE_\infty$ computed above (or verifying the filtration on $S_0$).
The differential has degree $(-1,0)$, so necessarily  $\Z^3\leftarrow \Z^2$ should be injective  and $E^2_{*,*}=E^\infty_{*,*}$.

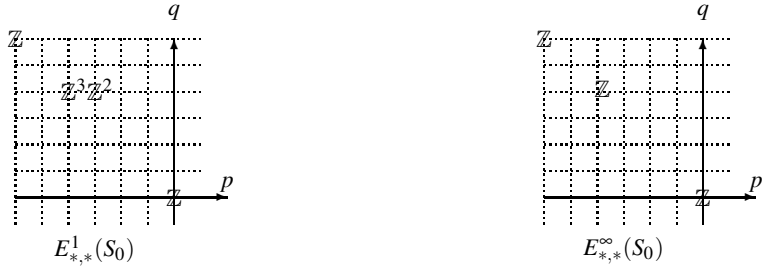
 \begin{picture}(300,110)(-40,-20)

  \put(0,10){\vector(1,0){80}}\put(60,0){\vector(0,1){70}}
  \dashline[200]{1}(50,0)(50,70)\dashline[200]{1}(40,0)(40,70)\dashline[200]{1}(30,0)(30,70)
  \dashline[200]{1}(20,0)(20,70)\dashline[200]{1}(10,0)(10,70)\dashline[200]{1}(0,0)(0,70)
   \dashline[200]{1}(0,20)(70,20)\dashline[200]{1}(0,30)(70,30)\dashline[200]{1}(0,40)(70,40)
  \dashline[200]{1}(0,50)(70,50)\dashline[200]{1}(0,60)(70,60)\dashline[200]{1}(0,70)(70,70)
  \put(60,80){\makebox(0,0){\footnotesize{$q$}}} \put(80,15){\makebox(0,0){\footnotesize$p$}}
   \put(30,-10){\makebox(0,0){\footnotesize{$E^1_{*,*}(S_0)$}}}
    \put(60,10){\makebox(0,0){\footnotesize{$\Z$}}} \put(0,70){\makebox(0,0){\footnotesize{$\Z$}}}
    \put(22,51){\makebox(0,0){\footnotesize{$\Z^3$}}} \put(32,51){\makebox(0,0){\footnotesize{$\Z^2$}}}

    \put(200,10){\vector(1,0){80}}\put(260,0){\vector(0,1){70}}
  \dashline[200]{1}(250,0)(250,70)\dashline[200]{1}(240,0)(240,70)\dashline[200]{1}(230,0)(230,70)
  \dashline[200]{1}(220,0)(220,70)\dashline[200]{1}(210,0)(210,70)\dashline[200]{1}(200,0)(200,70)
   \dashline[200]{1}(200,20)(270,20)\dashline[200]{1}(200,30)(270,30)\dashline[200]{1}(200,40)(270,40)
  \dashline[200]{1}(200,50)(270,50)\dashline[200]{1}(200,60)(270,60)\dashline[200]{1}(200,70)(270,70)
  \put(260,80){\makebox(0,0){\footnotesize{$q$}}} \put(280,15){\makebox(0,0){\footnotesize$p$}}
   \put(230,-10){\makebox(0,0){\footnotesize{$E^\infty_{*,*}(S_0)$}}}
    \put(260,10){\makebox(0,0){\footnotesize{$\Z$}}} \put(200,70){\makebox(0,0){\footnotesize{$\Z$}}}
    \put(222,51){\makebox(0,0){\footnotesize{$\Z$}}} 

  \end{picture}

In the case of $S_1$, $Q^1$ in  $PE_1(T,Q,h)$ appears with $3(T^4h+T^5)+2T^6h+3T^7$ while in $PE_\infty(T,Q,h)$ with $T^7$.
Hence the pages of the spectral sequence are

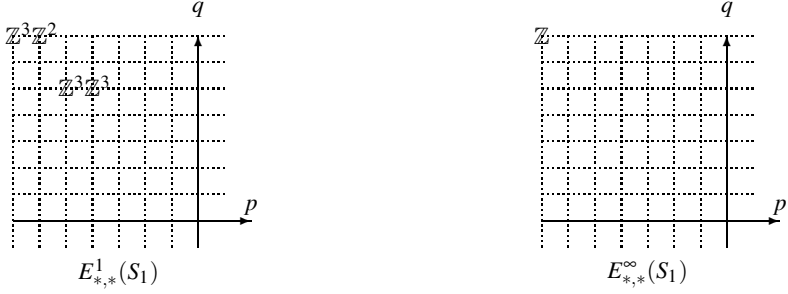
\begin{picture}(300,120)(-40,-20)

  \put(-10,10){\vector(1,0){90}}\put(60,0){\vector(0,1){80}}
  \dashline[200]{1}(50,0)(50,80)\dashline[200]{1}(40,0)(40,80)\dashline[200]{1}(30,0)(30,80)
  \dashline[200]{1}(20,0)(20,80)\dashline[200]{1}(10,0)(10,80)\dashline[200]{1}(0,0)(0,80)\dashline[200]{1}(-10,0)(-10,80)
   \dashline[200]{1}(-10,20)(70,20)\dashline[200]{1}(-10,30)(70,30)\dashline[200]{1}(-10,40)(70,40)
  \dashline[200]{1}(-10,50)(70,50)\dashline[200]{1}(-10,60)(70,60)\dashline[200]{1}(-10,70)(70,70)\dashline[200]{1}(-10,80)(70,80)
  \put(60,90){\makebox(0,0){\footnotesize{$q$}}} \put(80,15){\makebox(0,0){\footnotesize$p$}}
   \put(30,-10){\makebox(0,0){\footnotesize{$E^1_{*,*}(S_1)$}}}
    \put(12,61){\makebox(0,0){\footnotesize{$\Z^3$}}} \put(22,61){\makebox(0,0){\footnotesize{$\Z^3$}}}
\put(-8,81){\makebox(0,0){\footnotesize{$\Z^3$}}} \put(2,81){\makebox(0,0){\footnotesize{$\Z^2$}}}

    \put(190,10){\vector(1,0){90}}\put(260,0){\vector(0,1){80}}
  \dashline[200]{1}(250,0)(250,80)\dashline[200]{1}(240,0)(240,80)\dashline[200]{1}(230,0)(230,80)
  \dashline[200]{1}(220,0)(220,80)\dashline[200]{1}(210,0)(210,80)\dashline[200]{1}(200,0)(200,80)\dashline[200]{1}(190,0)(190,80)
   \dashline[200]{1}(190,20)(270,20)\dashline[200]{1}(190,30)(270,30)\dashline[200]{1}(190,40)(270,40)
  \dashline[200]{1}(190,50)(270,50)\dashline[200]{1}(190,60)(270,60)\dashline[200]{1}(190,70)(270,70)\dashline[200]{1}(190,80)(270,80)
  \put(260,90){\makebox(0,0){\footnotesize{$q$}}} \put(280,15){\makebox(0,0){\footnotesize$p$}}
   \put(230,-10){\makebox(0,0){\footnotesize{$E^\infty_{*,*}(S_1)$}}}
\put(190,80){\makebox(0,0){\footnotesize{$\Z$}}}

  \end{picture}

It turns out that  $E^2_{*,*}=E^\infty_{*,*}$ for all $n$, in particular, $PE_\infty$ is obtained from $PE_1$ by substitution
$h\mapsto -T$, cf. \ref{ss:ss}.

\begin{remark}
The above examples from \ref{ss:22} and \ref{ss:33}  might suggest that in general, for any $n$,
the spectral sequence degenerates at most at $E^2$-level, that is, $E^2_{*,*}=E^\infty_{*,*}$.
However, this is not the case as the next family  shows. In particular, the degeneration
invariant $k_{(C,o)}$ can be strict larger than two.
\end{remark}


\subsection{Decomposable singularities.}\label{ss:decsing}

A curve singularity $(C,o)$ is called `decomposable' (into
$(C',o)$ and $(C'',o)$),
 if it is isomorphic to the one-point union in
$(\bC^n\times \bC^m,o)$ of  $(C',o)\times \{o\}\subset (\bC^n,o)\times \{o\}$ and
$\{o\}\times (C'',o)\subset \{o\}\times (\bC^m,o)$.
We denote this by $(C,o)=(C',o)\vee (C'',o)$.

If $(C',o)\subset (\bC^n,o)$
is given by the ideal $I'$,
then its ideal in $(\bc^n\times \bC^m,o)$ is
$I'+\frm_{(\bc^m,o)}$ (here $\frm$ denotes the maximal ideal).
Using this observation, one can deduce that if $(C,o)=(C',o)\cup(C'',o)$ then
$(C,o)$ is decomposable into  $(C',o)$ and $(C'',o)$
if and only if $(C',C'')_{Hir}=1$  (or, if and only if $\delta(C,o)=\delta(C',o)+\delta(C'',o)+1$, cf. (\ref{eq:Hir})),
see \cite{Steiner83,Stevens85}.

Similarly, a computation shows that
the semigroup of values also behave `additively'. Assume that the number of
irreducible components of  $(C',o)$ and $(C'',o)$ is $r'$ and $r''$, then
\begin{equation}\label{eq:dec1}
\calS(C,o)=\{(0,0)\}\cup\, \big(\ (\calS(C',o)\setminus \{0\}\times \calS(C'',o)\setminus \{0\})\ \big)
\subset
\Z_{\geq 0}^{r'}\times \Z_{\geq 0}^{r''}= \Z_{\geq 0}^{r'+r''}.
\end{equation}
Furthermore, by Lemma \ref{eq:semi} and from  $w(l)=2\hh(l)-|l|$   we also have
\begin{equation}\label{eq:dec2}
 \hh_{(C,o)}(l',l'')=\left\{
 \begin{array}{ll}
        \hh_{(C',o)}(l') & \mbox{if $l''=0$},\\
          \hh_{(C'',o)}(l'') & \mbox{if $l'=0$},\\
            \hh_{(C',o)}(l') +\hh_{(C'',o)}(l'')-1 & \mbox{if $l'>0$ and $l''>0$};
       \end{array}\right.
\end{equation}
\begin{equation}\label{eq:dec3}
 w_{(C,o)}(l',l'')=\left\{
 \begin{array}{ll}
        w_{(C',o)}(l') & \mbox{if $l''=0$},\\
          w_{(C'',o)}(l'') & \mbox{if $l'=0$},\\
            w_{(C',o)}(l') +w_{(C'',o)}(l'')-2 & \mbox{if $l'>0$ and $l''>0$};
       \end{array}\right.
\end{equation}
Furthermore, (\ref{eq:pmot}) gives
\begin{equation}\label{eq:dec4}
P^m_{(C,o)}(\bt',\bt'';q)-1=  q^{-1}(1-q)\big(\,P^m_{(C',o)}(\bt';q)-1\big)\big(P^m_{(C'',o)}(\bt'';q)-1\big).
\end{equation}
In particular, via Theorem \ref{th:PP},
\begin{equation}\label{eq:dec5}
\bPE_{(C,o),1}({\bf T}',{\bf T}'', Q,h)-1=  Q^{-2}(1+Qh)\big(
\bPE_{(C',o),1}({\bf T}', Q,h)-1\big)\big(
\bPE_{(C'',o),1}({\bf T}'', Q,h)-1\big).
\end{equation}
\begin{remark} (a)
The formulae for $\bPE_{(C,o),k}$ for $k>1$, in general is more complicated.
The complete discussion will be given in a forthcoming manuscript. They can be deduced from the structure
of the level spaces $\{S_n\}_n$, which are determined by the weight function according to
(\ref{eq:dec3}). E.g., if we denote by $\bar{S}_n$ the intersection of $S_n$ with $(\R_{\geq 1})^r$, then
at the level of these spaces we have
$$\bar{S}_{(C,o),n}=\cup_{n'+n''=n} \ \bar{S}_{(C',0),n'}\times  \bar{S}_{(C'',0),n''}.$$
By analysing the first lines/columns,
we get that for $n\not=0$ the spaces $S_{(C,o),n}$ and $\bar{S}_{(C,o),n}$
have the same (graded) homotopy type, while  for
$n=0$ the space  $S_{(C,o),0}$ is the disjoint union of $\{0\}$ and a space $S'_{(C,o),0}$
where $S'_{(C,o),0}$ and $\bar{S}_{(C,o),0}$
have the same (graded) homotopy type.

(b) Based on the discussion from {\it (a)} in order to compute $\bH_*(C_1\vee C_2)$ we have to focus on the
lattice homology on the spaces $(\R_{\geq 1})^r$. But by (\ref{eq:dec3}) the weights here behave additively. This creates a
situation similar to the computation of the lattice homology, or of the Heegaard Floer homology associated with the connected
sum of plumbed 3--manifolds, where one compares the homologies associated with graphs $\Gamma_1$, $\Gamma_2$ and the
disjoint union $\Gamma_1\sqcup\Gamma_2$. For such a connected formula in the $HF^-$ theory see
\cite{os22}, section 6. Here a very similar (K\"unneth) formula holds (with similar homological algebra proof).
This is formulated as follows.
\end{remark}
We separate from $\bH_*$ the contribution given by the lattice point $0$: we write $\bH_0$ as $\bar{\bH}_0\oplus  \calt_0(1)$, and
$\bar{\bH}_b=\bH_b$ for $b\geq 1$.
\begin{theorem}\label{th:kunneth}
For any $b\geq 0$ we have the following isomorphism of $\Z[U]$--modules:
$$\bar{\bH}_b(C_1\vee C_2)=\oplus _{i+j=b}\, \bar{\bH}_i(C_1)\otimes_{\Z[U]} \bar{\bH}_j(C_2)[4]\ \oplus \
\oplus _{i+j=b-1}\, {\rm Tor}_{\Z[U]}(\bar{\bH}_i(C_1), \bar{\bH}_j(C_2))[2].$$
\end{theorem}
Note also the following identities: $\calt^-_n\otimes _{\Z[U]} \calt^-_m=\calt^-_{n+m}$,
$\calt^-_n\otimes _{\Z[U]} \calt_m(k)=\calt_{n+m}(k)$, $\calt_n(k')\otimes _{\Z[U]} \calt_m(k)=\calt_{n+m}(\min\{k',k\})$,
${\rm Tor}_{\Z[U]}(\calt^-_n, M)=0$ and ${\rm Tor} _{\Z[U]}(\calt_n(k'),\calt_m(k))=\calt_{n+m}(\min\{k',k\})$.

\vspace{2mm}

 Next we focus on  some concrete key examples.

\begin{example}\label{ex:gen}
Assume that $(C^{(j)},o)$ ($1\leq j\leq r$) is an irreducible curve singularity
with $\calS_{(C^{(j)},o)}=\{0\}\cup \bZ_{\geq c^{(j)}}$, for some $c^{(j)}\geq 1$. This situation can be realized.
In fact, any numerical semigroup $\calS\subset \Z_{\geq 0}$
with $\Z_{\geq 0}\setminus \calS$ finite is the semigroup of some
singular germ. Indeed,
let $\bar{\beta}_0, \ldots, \bar{\beta}_s $ be a minimal set of generators of $\calS$ (for its existence see e.g. \cite{Assi}).
Let $C^\calS$ be the affine curve defined via the parametrization
$t\mapsto (t^{\bar{\beta}_0}, \ldots, t^{\bar{\beta}_s})$. Then the affine coordinate ring
$\C[C^\calS]$ of $C^\calS$ is the image in $\C[t]$ of the morphism
$\varphi:\C[u_0,\ldots, u_s]\to \C[t]$, $\varphi(u_i)=t^{\bar{\beta}_i}$, which correspond to the normalization of
$C^{\calS}$.
The analytic germ $(C^\calS,0)$  is irreducible, and its semigroup is $\calS$.

Having the germs  $(C^{(j)},o)$ ($1\leq j\leq r$), let us consider $(C,o):=\vee_{j=1}^r
 (C^{(j)},o)$. It has $r$ irreducible components, they are isomorphic to $\{(C^{(j)},o)\}_j$.
 Then $\calS_{(C,o)}=\{0\}\cup \prod_j \Z_{\geq c^{(j)}}$, hence the conductor  of
 $\calS_{(C,o)}$ is $c=(c^{(1)}, \dots, c^{(r)})$.  By the above general formula (\ref{eq:dec4})  we have
 $$P^m(\bt;q)=1+\sum_{l\geq c}\bt^l q^{|l|-|c|+1}(1-q)^{r-1}=1+\bt^cq\cdot \frac{(1-q)^{r-1}}
 {\prod_i (1-t_iq)}.$$
 Since for any $l\geq c$ we have $\hh(l)=|l|-|c|+1$, we also have
 $$\bPE_1({\bf T}, Q,h)=1+\sum_{l\geq c}{\bf T}^l Q^{|l|-2|c|+2}(1+Qh)^{r-1}
 =1+{\bf T}^cQ^{2-|c|}\cdot \frac{(1+Qh)^{r-1}} {\prod_i (1-T_iQ)},$$
 $$PE_1(T, Q,h)
 =1+T^{|c|}Q^{2-|c|}\cdot \frac{(1+Qh)^{r-1}} {(1-TQ)^r}.$$
 In this case, from the explicit structure of the $w$-table in $R(0,c)$ we also get that each
 $S_n$ for $n\not=0$ is contractible, and $S_0$ consists of two components, each contractible.
 Therefore, $\bH_{>0}(C,o)$ is trivial, and  $\bH_0(C,o)= \calT^-_{-2(2-|c|)}\oplus \calt_0(1)$.
 (This identity can be deduced via Theorem \ref{th:kunneth} as well.)
 The homological graded root is

 
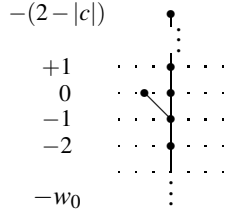
\begin{picture}(300,92)(80,330)

\put(180,380){\makebox(0,0){\footnotesize{$0$}}} \put(177,390){\makebox(0,0){\footnotesize{$+1$}}}
\put(177,410){\makebox(0,0){\footnotesize{$-(2-|c|)$}}}\put(177,370){\makebox(0,0){\footnotesize{$-1$}}}
\put(177,340){\makebox(0,0){\small{$-w_0$}}}\put(177,360){\makebox(0,0){\footnotesize{$-2$}}}
\dashline{1}(200,350)(240,350)
\dashline{1}(200,380)(240,380) \dashline{1}(200,360)(240,360)
\dashline{1}(200,390)(240,390) \dashline{1}(200,370)(240,370)
\put(220,345){\makebox(0,0){$\vdots$}} \put(220,360){\circle*{3}}
\put(220,370){\circle*{3}} \put(210,380){\circle*{3}}
\put(220,380){\circle*{3}}
\put(220,390){\circle*{3}}
\put(220,410){\circle*{3}} \put(220,410){\line(0,-1){5}}
 \put(220,350){\line(0,1){45}}
\put(210,380){\line(1,-1){10}} 
\put(223,403){\makebox(0,0){$\vdots$}}
\end{picture}

\noindent  Thus, one reads directly that
$$PE_\infty(T,Q,h)
  =1+T^{|c|}Q^{2-|c|}\cdot \frac{1} {1-TQ}.$$
In particular, $PE_\infty$ is obtained from $PE_1$ by cancellation of terms of type
$T^aQ^bH^c(T+h)$ (cf. \ref{ss:ss}) (that is, by substitution $h\mapsto -T$),
 hence the spectral sequence degenerates at $E^2$ level.

 Note that in this case $\delta=|c|-1$. This shows that the inequality $\delta\leq |c|-1$
 proved in \ref{bek:AnnFiltr} is sharp, and  this extremal case corresponds exactly to
$\hh(c)=1$, or $\frc=\frm_{(C,o)}$.

The curve is extremal from the point of view of the lattice homology
$\bH_*(C,o)$ as well. In \cite[Example 4.6.1]{AgostonNemethi} is proved that $\bH_{*,red}\not=0$ if and only if $(C,o)$ is non-smooth.
Note that in our case ${\rm rank}_{\Z}\bH_{*,red}(C,o)=1$, the smallest possible among the non-smooth germs.

If $c^{(i)}=1$ for all $i$ then $(C,o)$ is an ordinary $r$-tuple, hence the above discussion
clarifies their invariants as well.

\end{example}

\begin{example}\label{ex:decsing2}
 Assume that $(C',o)=(C'',o)=\{x^3+y^4=0\}$.
 Then $\calS_{(C',o)}=\calS_{(C'',o)}=\langle3,4\rangle$.
 Here we provide some concrete computation for $(C,o)=(C',o)\vee (C'',o)$.

The $w$-table in $R(0,(8,8))$ and the homological graded root is the following:

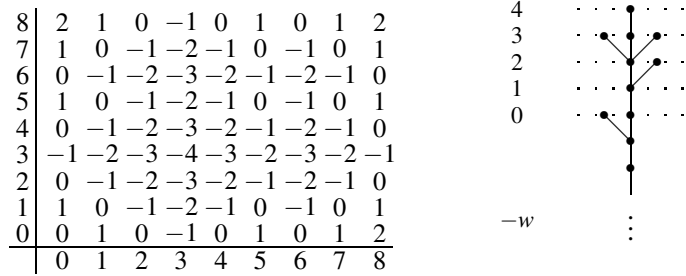
\begin{picture}(320,120)(-100,-20)

\put(177,50){\makebox(0,0){\footnotesize{$0$}}} \put(177,60){\makebox(0,0){\footnotesize{$1$}}}
\put(177,70){\makebox(0,0){\footnotesize{$2$}}}\put(177,80){\makebox(0,0){\footnotesize{$3$}}}
\put(177,90){\makebox(0,0){\footnotesize{$4$}}}\put(177,10){\makebox(0,0){\footnotesize{$-w$}}}
\dashline{1}(200,60)(240,60)
\dashline{1}(200,80)(240,80)
\dashline{1}(200,50)(240,50) \dashline{1}(200,90)(240,90)
\dashline{1}(200,70)(240,70)
\put(220,10){\makebox(0,0){$\vdots$}} \put(220,30){\circle*{3}}
\put(220,40){\circle*{3}} \put(210,50){\circle*{3}}
\put(220,50){\circle*{3}}
\put(220,60){\circle*{3}} \put(220,20){\line(0,1){70}}
\put(210,50){\line(1,-1){10}}

\put(220,80){\circle*{3}}
\put(220,90){\circle*{3}}
\put(210,80){\circle*{3}}
\put(220,70){\circle*{3}}
\put(230,70){\circle*{3}}
\put(230,80){\circle*{3}}
\put(220,70){\line(1,1){10}}
\put(220,70){\line(-1,1){10}}
\put(220,60){\line(1,1){10}}



\put(-15,0){\line(1,0){145}}
\put(-5,-10){\line(0,1){100}}

\put(5,-5){\makebox(0,0){\small{$0$}}}
\put(20,-5){\makebox(0,0){\small{$1$}}}
\put(35,-5){\makebox(0,0){\small{$2$}}}
\put(50,-5){\makebox(0,0){\small{$3$}}}
\put(65,-5){\makebox(0,0){\small{$4$}}}
\put(80,-5){\makebox(0,0){\small{$5$}}}
\put(95,-5){\makebox(0,0){\small{$6$}}}
\put(110,-5){\makebox(0,0){\small{$7$}}}
\put(125,-5){\makebox(0,0){\small{$8$}}}

\put(-10,5){\makebox(0,0){\small{$0$}}}
\put(-10,15){\makebox(0,0){\small{$1$}}}
\put(-10,25){\makebox(0,0){\small{$2$}}}
\put(-10,35){\makebox(0,0){\small{$3$}}}
\put(-10,45){\makebox(0,0){\small{$4$}}}
\put(-10,55){\makebox(0,0){\small{$5$}}}
\put(-10,65){\makebox(0,0){\small{$6$}}}
\put(-10,75){\makebox(0,0){\small{$7$}}}
\put(-10,85){\makebox(0,0){\small{$8$}}}

\put(5,5){\makebox(0,0){\small{$0$}}}
\put(5,15){\makebox(0,0){\small{$1$}}}
\put(5,25){\makebox(0,0){\small{$0$}}}
\put(5,35){\makebox(0,0){\small{$-1$}}}
\put(5,45){\makebox(0,0){\small{$0$}}}
\put(5,55){\makebox(0,0){\small{$1$}}}
\put(5,65){\makebox(0,0){\small{$0$}}}
\put(5,75){\makebox(0,0){\small{$1$}}}
\put(5,85){\makebox(0,0){\small{$2$}}}

\put(20,5){\makebox(0,0){\small{$1$}}}
\put(20,15){\makebox(0,0){\small{$0$}}}
\put(20,25){\makebox(0,0){\small{$-1$}}}
\put(20,35){\makebox(0,0){\small{$-2$}}}
\put(20,45){\makebox(0,0){\small{$-1$}}}
\put(20,55){\makebox(0,0){\small{$0$}}}
\put(20,65){\makebox(0,0){\small{$-1$}}}
\put(20,75){\makebox(0,0){\small{$0$}}}
\put(20,85){\makebox(0,0){\small{$1$}}}

\put(35,5){\makebox(0,0){\small{$0$}}}
\put(35,15){\makebox(0,0){\small{$-1$}}}
\put(35,25){\makebox(0,0){\small{$-2$}}}
\put(35,35){\makebox(0,0){\small{$-3$}}}
\put(35,45){\makebox(0,0){\small{$-2$}}}
\put(35,55){\makebox(0,0){\small{$-1$}}}
\put(35,65){\makebox(0,0){\small{$-2$}}}
\put(35,75){\makebox(0,0){\small{$-1$}}}
\put(35,85){\makebox(0,0){\small{$0$}}}

\put(50,5){\makebox(0,0){\small{$-1$}}}
\put(50,15){\makebox(0,0){\small{$-2$}}}
\put(50,25){\makebox(0,0){\small{$-3$}}}
\put(50,35){\makebox(0,0){\small{$-4$}}}
\put(50,45){\makebox(0,0){\small{$-3$}}}
\put(50,55){\makebox(0,0){\small{$-2$}}}
\put(50,65){\makebox(0,0){\small{$-3$}}}
\put(50,75){\makebox(0,0){\small{$-2$}}}
\put(50,85){\makebox(0,0){\small{$-1$}}}

\put(65,5){\makebox(0,0){\small{$0$}}}
\put(65,15){\makebox(0,0){\small{$-1$}}}
\put(65,25){\makebox(0,0){\small{$-2$}}}
\put(65,35){\makebox(0,0){\small{$-3$}}}
\put(65,45){\makebox(0,0){\small{$-2$}}}
\put(65,55){\makebox(0,0){\small{$-1$}}}
\put(65,65){\makebox(0,0){\small{$-2$}}}
\put(65,75){\makebox(0,0){\small{$-1$}}}
\put(65,85){\makebox(0,0){\small{$0$}}}

\put(80,5){\makebox(0,0){\small{$1$}}}
\put(80,15){\makebox(0,0){\small{$0$}}}
\put(80,25){\makebox(0,0){\small{$-1$}}}
\put(80,35){\makebox(0,0){\small{$-2$}}}
\put(80,45){\makebox(0,0){\small{$-1$}}}
\put(80,55){\makebox(0,0){\small{$0$}}}
\put(80,65){\makebox(0,0){\small{$-1$}}}
\put(80,75){\makebox(0,0){\small{$0$}}}
\put(80,85){\makebox(0,0){\small{$1$}}}

\put(95,5){\makebox(0,0){\small{$0$}}}
\put(95,15){\makebox(0,0){\small{$-1$}}}
\put(95,25){\makebox(0,0){\small{$-2$}}}
\put(95,35){\makebox(0,0){\small{$-3$}}}
\put(95,45){\makebox(0,0){\small{$-2$}}}
\put(95,55){\makebox(0,0){\small{$-1$}}}
\put(95,65){\makebox(0,0){\small{$-2$}}}
\put(95,75){\makebox(0,0){\small{$-1$}}}
\put(95,85){\makebox(0,0){\small{$0$}}}

\put(110,5){\makebox(0,0){\small{$1$}}}
\put(110,15){\makebox(0,0){\small{$0$}}}
\put(110,25){\makebox(0,0){\small{$-1$}}}
\put(110,35){\makebox(0,0){\small{$-2$}}}
\put(110,45){\makebox(0,0){\small{$-1$}}}
\put(110,55){\makebox(0,0){\small{$0$}}}
\put(110,65){\makebox(0,0){\small{$-1$}}}
\put(110,75){\makebox(0,0){\small{$0$}}}
\put(110,85){\makebox(0,0){\small{$1$}}}

\put(125,5){\makebox(0,0){\small{$2$}}}
\put(125,15){\makebox(0,0){\small{$1$}}}
\put(125,25){\makebox(0,0){\small{$0$}}}
\put(125,35){\makebox(0,0){\small{$-1$}}}
\put(125,45){\makebox(0,0){\small{$0$}}}
\put(125,55){\makebox(0,0){\small{$1$}}}
\put(125,65){\makebox(0,0){\small{$0$}}}
\put(125,75){\makebox(0,0){\small{$1$}}}
\put(125,85){\makebox(0,0){\small{$2$}}}

\end{picture}

Then we deduce  that $\bH_{>2}=0$, $\bH_1=\calt_{2}(1)$ is generated by the loop around the lattice point
$(5,5)$ in $S_{-1}$,  and $\bH_0=\calt_{8}^-\oplus \calt_{6}(1)^{\oplus 2}\oplus \calt_{4}(1)\oplus\calt_0(1)$.
(These can be read from Theorem \ref{th:kunneth} as well.)
 The delta invariant is $\delta=7$ (compatibly with $7=3+3+1$).

We wish to emphasize that in general, by Proposition \ref{cor:EUcurves},
the {\it homotopy type} of $S_n$ is given by $S_n\cap R(0,c)$. However, if we wish to determine the graded $S_n$,
or $\bPE_1$, then we need a rectangle which contains $S_n$.  (This motivates that we provided above
$R(0,(8,8))$, though  $c=(6,6)$, since $S_n\subset R(0,(8,8))$ for  $-4\leq n\leq -1$, and we wish to picture these spaces.)
From the rectangle $R(0,c)$ one also sees that $w$ is not symmetric with respect to $l\leftrightarrow c-l$ (cf. \ref{bek:GORdualoty})),
hence $(C,o)$ is not Gorenstein, though both components are plane curve singularities.

The spaces $S_n$  for $n=-4 -3,-2,-1 $  are the following:

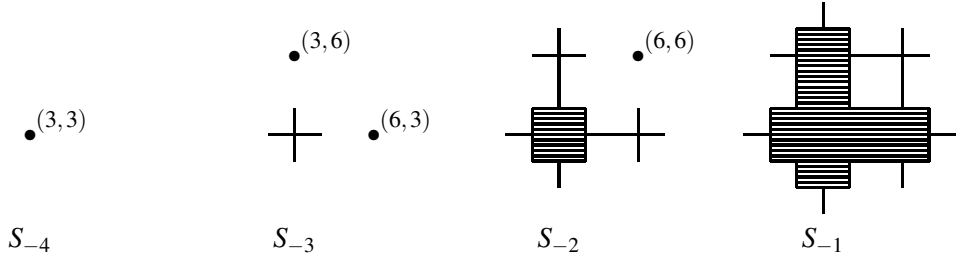
\begin{picture}(350,120)(0,-20)

\thicklines
 \put(30,-10){\makebox(0,0){$S_{-4}$}}\put(42,35){\makebox(0,0){\footnotesize$(3,3)$}}
   \put(30,30){\circle*{4}}

 \put(130,-10){\makebox(0,0){$S_{-3}$}}\put(142,65){\makebox(0,0){\footnotesize$(3,6)$}}
 \put(172,35){\makebox(0,0){\footnotesize$(6,3)$}}
   \put(130,60){\circle*{4}} \put(160,30){\circle*{4}}
\put(120,30){\line(1,0){20}}\put(130,20){\line(0,1){20}}

\put(230,-10){\makebox(0,0){$S_{-2}$}}
 \put(272,65){\makebox(0,0){\footnotesize$(6,6)$}}
   \put(260,60){\circle*{4}} 
\put(210,30){\line(1,0){10}}
\put(240,30){\line(1,0){30}}
\put(220,40){\line(1,0){20}}\put(220,20){\line(1,0){20}}\put(220,60){\line(1,0){20}}

\put(230,10){\line(0,1){10}}
\put(230,40){\line(0,1){30}}
\put(220,20){\line(0,1){20}}\put(240,20){\line(0,1){20}}\put(260,20){\line(0,1){20}}

\put(220,22){\line(1,0){20}}\put(220,24){\line(1,0){20}}\put(220,26){\line(1,0){20}}
\put(220,28){\line(1,0){20}}\put(220,32){\line(1,0){20}}\put(220,34){\line(1,0){20}}
\put(220,36){\line(1,0){20}}\put(220,38){\line(1,0){20}}
\put(220,30){\line(1,0){20}}

\put(330,-10){\makebox(0,0){$S_{-1}$}}
\put(300,30){\line(1,0){80}}
\put(310,20){\line(1,0){60}}\put(310,40){\line(1,0){60}}\put(310,60){\line(1,0){60}}

\put(310,22){\line(1,0){60}}\put(310,24){\line(1,0){60}}\put(310,26){\line(1,0){60}}
\put(310,28){\line(1,0){60}}
\put(310,32){\line(1,0){60}}\put(310,34){\line(1,0){60}}\put(310,36){\line(1,0){60}}
\put(310,38){\line(1,0){60}}

\put(320,18){\line(1,0){20}}\put(320,16){\line(1,0){20}}\put(320,14){\line(1,0){20}}\put(320,12){\line(1,0){20}}

\put(320,42){\line(1,0){20}}\put(320,46){\line(1,0){20}}\put(320,44){\line(1,0){20}}\put(320,48){\line(1,0){20}}\put(320,50){\line(1,0){20}}
\put(320,52){\line(1,0){20}}\put(320,56){\line(1,0){20}}\put(320,54){\line(1,0){20}}\put(320,58){\line(1,0){20}}
\put(320,10){\line(1,0){20}}
\put(320,62){\line(1,0){20}}\put(320,66){\line(1,0){20}}\put(320,64){\line(1,0){20}}\put(320,68){\line(1,0){20}}\put(320,70){\line(1,0){20}}

\put(310,20){\line(0,1){20}}
\put(320,10){\line(0,1){10}}
\put(320,40){\line(0,1){30}}
\put(340,10){\line(0,1){10}}
\put(340,40){\line(0,1){30}}

\put(370,20){\line(0,1){20}}

\put(330,0){\line(0,1){10}}
\put(330,70){\line(0,1){10}}\put(360,10){\line(0,1){10}}\put(360,40){\line(0,1){30}}
\thinlines
\end{picture}

The space $S_{-2}$ has two connected component: one of them is the isolated lattice point $(6,6)$,
let $S_{-2}'$ be the other one.
The non-zero terms in the
 first page of the spectral sequence $E^1_{*,*}$ {\it associated with} $S_{-2}'$ can be read from its  cubical complex:
 $E^1_{-10,10}=\Z^4$, $E^1_{-9,10}=\Z^2$, $E^1_{-8,8}=\Z$, $E^1_{-7,8}=\Z^2$.

 This must converge to $E^\infty_{*,*}$ where the only non-zero term is
  $E^\infty_{-10,10}=\Z$. 
This can happen only if the differential $d^3 $ is non-trivial. I.e.,
$E^1\not=E^2=E^3\not= E^4=E^\infty$, hence $k_{(C,o)}\geq 4$.

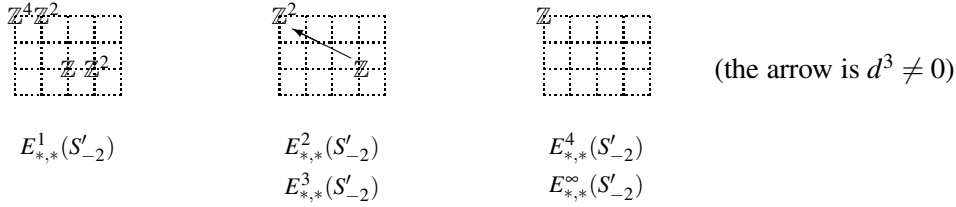
\begin{picture}(300,80)(-40,10)

 \dashline[200]{1}(30,50)(30,80)
  \dashline[200]{1}(20,50)(20,80)\dashline[200]{1}(10,50)(10,80)
  \dashline[200]{1}(0,50)(0,80)
  \dashline[200]{1}(-10,50)(-10,80)

  \dashline[200]{1}(-10,50)(30,50)\dashline[200]{1}(-10,60)(30,60)
  \dashline[200]{1}(-10,70)(30,70)\dashline[200]{1}(-10,80)(30,80)

   \put(10,30){\makebox(0,0){\footnotesize{$E^1_{*,*}(S_{-2}')$}}}
    \put(10,60){\makebox(0,0){\footnotesize{$\Z$}}} \put(21,61){\makebox(0,0){\footnotesize{$\Z^2$}}}
\put(-8,81){\makebox(0,0){\footnotesize{$\Z^4$}}} \put(2,81){\makebox(0,0){\footnotesize{$\Z^2$}}}


  \dashline[200]{1}(130,50)(130,80)
  \dashline[200]{1}(120,50)(120,80)\dashline[200]{1}(110,50)(110,80)
  \dashline[200]{1}(100,50)(100,80)
  \dashline[200]{1}(90,50)(90,80)

  \dashline[200]{1}(90,50)(130,50)\dashline[200]{1}(90,60)(130,60)
  \dashline[200]{1}(90,70)(130,70)\dashline[200]{1}(90,80)(130,80)

   \put(110,30){\makebox(0,0){\footnotesize{$E^2_{*,*}(S_{-2}')$}}}
    \put(110,15){\makebox(0,0){\footnotesize{$E^3_{*,*}(S_{-2}')$}}}
  \put(121,60){\makebox(0,0){\footnotesize{$\Z$}}}
\put(92,81){\makebox(0,0){\footnotesize{$\Z^2$}}}
\put(300,60){\makebox(0,0){\mbox{(the arrow is $d^3\not=0$)}}}

  \put(117,64){\vector(-2,1){22}}


  \dashline[200]{1}(230,50)(230,80)
  \dashline[200]{1}(220,50)(220,80)\dashline[200]{1}(210,50)(210,80)
  \dashline[200]{1}(200,50)(200,80)
  \dashline[200]{1}(190,50)(190,80)

  \dashline[200]{1}(190,50)(230,50)\dashline[200]{1}(190,60)(230,60)
  \dashline[200]{1}(190,70)(230,70)\dashline[200]{1}(190,80)(230,80)

   \put(210,30){\makebox(0,0){\footnotesize{$E^4_{*,*}(S_{-2}')$}}}
     \put(210,15){\makebox(0,0){\footnotesize{$E^\infty_{*,*}(S_{-2}')$}}}
\put(190,80){\makebox(0,0){\footnotesize{$\Z$}}} 
  \end{picture}

(Above in each diagram the upper-left corner is the lattice point $(-10,10)$.)

Looking at  the pages of this spectral sequence (and the examples considered above) we might believe that
each differential $d^k_{*,*}$ has `maximal rank' (that is, its rank is the maximal of the ranks of its
source and target). However, this is not the case in general.
E.g., $S_{-1}$ has the homotopy type of a circle, hence $E^\infty_{p,q}=\Z$ for a pair $(p,q)$  with $p+q=0$ and for a pair
with $p+q=1$.  It turns out that the spectral sequence has the following pages (the upper-left corner is $(-13,13)$):

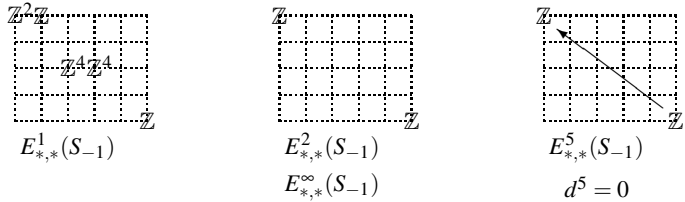
\begin{picture}(300,80)(-40,10)

 \dashline[200]{1}(30,40)(30,80)
  \dashline[200]{1}(20,40)(20,80)\dashline[200]{1}(10,40)(10,80)
  \dashline[200]{1}(0,40)(0,80)
  \dashline[200]{1}(-10,40)(-10,80) \dashline[200]{1}(40,40)(40,80)

  \dashline[200]{1}(-10,50)(40,50)\dashline[200]{1}(-10,60)(40,60)
  \dashline[200]{1}(-10,70)(40,70)\dashline[200]{1}(-10,80)(40,80)\dashline[200]{1}(-10,40)(40,40)

   \put(10,30){\makebox(0,0){\footnotesize{$E^1_{*,*}(S_{-1})$}}}
    \put(12,61){\makebox(0,0){\footnotesize{$\Z^4$}}} \put(22,61){\makebox(0,0){\footnotesize{$\Z^4$}}}
\put(-8,81){\makebox(0,0){\footnotesize{$\Z^2$}}} \put(0,80){\makebox(0,0){\footnotesize{$\Z$}}}
 \put(40,40){\makebox(0,0){\footnotesize{$\Z$}}}

  \dashline[200]{1}(130,40)(130,80)
  \dashline[200]{1}(120,40)(120,80)\dashline[200]{1}(110,40)(110,80)
  \dashline[200]{1}(100,40)(100,80)
  \dashline[200]{1}(90,40)(90,80) \dashline[200]{1}(140,40)(140,80)

  \dashline[200]{1}(90,50)(140,50)\dashline[200]{1}(90,60)(140,60)
  \dashline[200]{1}(90,70)(140,70)\dashline[200]{1}(90,80)(140,80)\dashline[200]{1}(90,40)(140,40)

   \put(110,30){\makebox(0,0){\footnotesize{$E^2_{*,*}(S_{-1})$}}}
   \put(110,15){\makebox(0,0){\footnotesize{$E^\infty_{*,*}(S_{-1})$}}}
\put(90,80){\makebox(0,0){\footnotesize{$\Z$}}}\put(140,40){\makebox(0,0){\footnotesize{$\Z$}}}

  \dashline[200]{1}(230,40)(230,80)
  \dashline[200]{1}(220,40)(220,80)\dashline[200]{1}(210,40)(210,80)
  \dashline[200]{1}(200,40)(200,80)
  \dashline[200]{1}(190,40)(190,80) \dashline[200]{1}(240,40)(240,80)

  \dashline[200]{1}(190,50)(240,50)\dashline[200]{1}(190,60)(240,60)
  \dashline[200]{1}(190,70)(240,70)\dashline[200]{1}(190,80)(240,80)\dashline[200]{1}(190,40)(240,40)

  \put(235,45){\vector(-4,3){40}}
\put(190,80){\makebox(0,0){\footnotesize{$\Z$}}}\put(240,40){\makebox(0,0){\footnotesize{$\Z$}}}
   \put(210,30){\makebox(0,0){\footnotesize{$E^5_{*,*}(S_{-1})$}}}
    \put(210,15){\makebox(0,0){\footnotesize{$d^5=0$}}}
  \end{picture}

I.e., $E^\infty_{-13,13}=E^{\infty}_{-8,9}=\Z$ and $d^5_{-8,9}:E^5_{-8,9}=\Z\to E^5_{-13,13}=\Z$ is the zero morphism.
Hence, the spectral sequence --- starting from $E^1_{*,*}$ --- does not  follow automatically the `maximal rank'
principle.
\end{example}

\begin{example}\label{ex::decsing2}
Consider now $r$ copies of $(C',o)=\{x^3+y^4=0\}$, and define
$(C,o):=\vee_{i=1}^r (C',o)$. Then $l=5$ for $w(C',o)$ is a local maximum point.
Therefore, the lattice point $M:=(5, 5, \ldots, 5)$ will be a local maximum point for the
$w$-values of $(C,o)$ with $w(M)=2-r$.
Hence, $\bH_{r-1}(C,o)\not=0$. This shows that for well-chosen germs  $\bH_b(C,o)$ can be non-zero for arbitrarily large $b$.
\end{example}
\begin{question}
(a) Is there any plane curve singularity with $k_{(C,o)}>2$?

(b) What is the algebraic interpretation of  $k_{(C,o)}(n)$ and $k_{(C,o)}$ ?
\end{question}

   \section{Actions along the spectral sequences. The operators $Y_1, \ldots , Y_r$}\label{s:Uoperators}

    \subsection{} Recall that we have the natural inclusion $S_n\subset S_{n+1}$ for any $n$.
    This induces a $U$--action on $\oplus_n H_*(S_n,\Z)$.
    Besides this,
    in this section we will define certain additional  maps too, and we will analyse the
    actions what they induce on the terms of the spectral sequence $E^*_{*,*}$.
    In this way we endow the terms of the spectral sequence $E^*_{*,*}$ with certain additional structure.

   We will denote by $u=u_n$ the inclusion $S_n\subset S_{n+1}$, and  we will use
   similar notation for inclusions of type $S_n\cap \frX_{-l}\subset S_{n+1}\cap \frX_{-l}$ as well.
    Analogously,  for any $i$, the  inclusions of type $\frX_{-l-E_i}\hookrightarrow \frX_{-l}$
     will be denoted by $x_i$. Furthermore we define another map
   $y_i:\frX\to \frX$ by $y_i(x)=x+E_i$. This map clearly sends $\frX_{-d}$ to  $\frX_{-d-1}$,
   $\frX_{-l}$ to $\frX_{-l-E_i}$ and each cube $\square=(l,I)$ to $(l+E_i, I)$.

  Moreover, since $w(l+E_i)\leq w(l)+1$,  we have
   \begin{lemma}\label{lem:ui}
   (i) \ $w(y_i(l))\leq w(l)+1$, \ (ii) \ $w(y_i(\square))\leq w(\square)+1$, \ (iii) \
   $y_i(S_n)\subset S_{n+1}$.
   \end{lemma}

   The map $y_i:S_n\to S_{n+1}$ is not very interesting at  $H_b(S_n,\Z)\to H_b(S_{n+1},\Z)$ level.
   Indeed, the  homotopy $S_n\times [0,1]\to S_{n+1}$ by $(x,t)\mapsto x+tE_i  $ gives the following.

   \begin{lemma}\label{lem:ui2}
   For each $i$, the map  $y_i:S_n\to S_{n+1}$ is homotopic to the inclusion map $u:S_n\subset S_{n+1}$, hence
   the morphisms $(\bH_b)_{-2n}\to (\bH_b)_{-2n-2}$ induced by $y_i$ and  $u$
   are equal.
   For $b=0$, the action induced by $y_i$ can be read from the edges of graded root as well (as described in
    Example \ref{ex:34}).
   \end{lemma}

   However, along the spectral sequence associated with the filtration $\{S_n\cap \frX_{-d}\}_d$
   of $S_n$  the maps $y_i$ induce more interesting morphisms.
   First, consider the following  diagram, which is commutative up to homotopy (by a similar argument as in Lemma \ref{lem:ui2}):

   \begin{picture}(300,60)(0,10)

   \put(50,50){\makebox(0,0){{$S_n\cap \frX_{-l-E_i}$}}}
    \put(50,10){\makebox(0,0){{$S_{n}\cap \frX_{-l}$}}}
    \put(150,50){\makebox(0,0){{$S_{n+1}\cap \frX_{-l-E_i}$}}}
    \put(150,10){\makebox(0,0){{$S_{n+1}\cap \frX_{-l}$}}}
    \put(85,50){\vector(1,0){30}} \put(85,10){\vector(1,0){30}}
    \put(50,40){\vector(0,-1){20}} \put(150,40){\vector(0,-1){20}}
   \put(80,20){\vector(3,2){30}}
   \put(100,55){\makebox(0,0){{$u$}}}  \put(100,15){\makebox(0,0){{$u$}}}
    \put(43,30){\makebox(0,0){{$x_{i}$}}}  \put(160,30){\makebox(0,0){{$x_{i}$}}}
     \put(87,30){\makebox(0,0){{$y_i$}}}
   \end{picture}

\vspace{2mm}

\noindent    It induced two other commutative diagrams:

    \begin{picture}(300,60)(-50,10)

   \put(50,50){\makebox(0,0){{$H_b(S_n\cap \frX_{-l-E_i},S_n\cap \frX_{-l-E_i}\cap \frX_{-|l|-2}) $}}}
    \put(50,10){\makebox(0,0){{$H_b(S_n\cap \frX_{-l},S_n\cap \frX_{-l}\cap \frX_{-|l|-1} )$}}}
    \put(250,50){\makebox(0,0){{$H_b(S_{n+1}\cap \frX_{-l-E_i},S_{n+1}\cap \frX_{-l-E_i}\cap \frX_{-|l|-2} )$}}}
    \put(250,10){\makebox(0,0){{$H_b(S_{n+1}\cap \frX_{-l},S_{n+1}\cap \frX_{-l}\cap \frX_{-|l|-1} ) $}}}
    \put(140,50){\vector(1,0){10}} \put(130,10){\vector(1,0){30}}
    \put(50,40){\vector(0,-1){20}} \put(250,40){\vector(0,-1){20}}
   \put(130,20){\vector(3,2){30}}
   \put(145,55){\makebox(0,0){{$U$}}}  \put(145,15){\makebox(0,0){{$U$}}}
    \put(40,30){\makebox(0,0){{$X_{i}$}}}  \put(265,30){\makebox(0,0){{$X_{i}$}}}
     \put(130,30){\makebox(0,0){{$Y_i$}}}
   \end{picture}

  \noindent  and

   \begin{picture}(300,60)(0,10)

   \put(50,50){\makebox(0,0){{$F_{-d-1}H_b(S_n)$}}}
    \put(50,10){\makebox(0,0){{$F_{-d}H_b(S_n)$}}}
    \put(150,50){\makebox(0,0){{\ \ $F_{-d-1}H_b(S_{n+1})$}}}
    \put(150,10){\makebox(0,0){{$F_{-d}H_b(S_{n+1})$}}}
    \put(85,50){\vector(1,0){30}} \put(85,10){\vector(1,0){30}}
    \put(50,40){\vector(0,-1){20}} \put(150,40){\vector(0,-1){20}}
   \put(80,20){\vector(3,2){30}}
   \put(100,55){\makebox(0,0){{$U$}}}  \put(100,15){\makebox(0,0){{$U$}}}
    \put(40,30){\makebox(0,0){{$X_{i}$}}}  \put(165,30){\makebox(0,0){{$X_{i}$}}}
     \put(85,30){\makebox(0,0){{$Y_i$}}}
   \end{picture}

\vspace{2mm}

   In the first homological  diagram the morphisms $X_{i}$ are automatically trivial, hence,
   we obtain  that $U$ in that diagram  (i.e. on the page $E^1$) acts trivially too. This is compatible with the statement
   from Theorem \ref{th:PP}{\it (c)}. That statement from
   Theorem \ref{th:PP}{\it (c)} was based     on the structure of the
   Orlik--Solomon algebra exploited in \cite{GorNem2015}, the above proof is based merely on the very  existence of the morphisms $Y_i$ induced by $y_i$.

   The second homological commutative diagram from above  proves that $U:{\rm Gr}_{-d}^F\,H_b(S_n)\to
   {\rm Gr}_{-d}^F\,H_b(S_{n+1})$ is trivial as well. This fact can be deduced by the following argument too. Consider again the map $u_n:S_n\hookrightarrow S_{n+1}$, which is compatible
   with the corresponding filtration $\{\frX_{-d}\}_d$. Therefore, $u_n$ induces morphisms at the level of
   spectral sequences associated with  $S_n$ and  $S_{n+1}$. Since at the $E^1_{*,*}$ page the
   induced map $U$ is trivial, it is trivial at the level of all the pages $E^k_{*,*}$,
   including $E^\infty_{*,*}$.

   For further reference we state:

   \begin{lemma}
   The morphisms $U$ induced by the inclusion $S_n\subset S_{n+1}$ on $(E^*_{*,*})_{n}\to
   (E^*_{*,*})_{n+1}$ (in particular, on ${\rm Gr}_{-d}^F\,H_b(S_n)\to
   {\rm Gr}_{-d}^F\,H_b(S_{n+1})$\,) are trivial.
   \end{lemma}
\subsection{The actions  $Y_i$}
Similarly as above
 the map $y_i:S_n\to S_{n+1}$ is compatible with the filtrations (i.e.
 $y_i:S_n\cap \frX_{-d}\to S_{n+1}\cap \frX_{-d-1}$), hence it induces
  spectral sequence morphisms
   connecting the spectral sequences of $S_n$ and  $S_{n+1}$.
   Thus,    for  any $k\geq 1$ one has  a morphism
    \begin{equation}\label{eq:yi}
    Y_i: (E^k_{-d, b+d})_{n} \to (E^k_{-d-1, b+d+1})_{n+1},\end{equation}
     which commutes with the differentials of the  spectral sequences.
This is compatible with the disjoint decomposition $S_n=\sqcup_v S_n^v$ into connected components. Via the notation of
\ref{ss:grroot}, if $[v,u]$ is an edge of ${\mathfrak{R}}$ then $y_i(S_n^v)\subset S_{n+1}^u$, hence we have a well-defined restriction
  \begin{equation}\label{eq:yi2}
Y_i|_{(E^k_{-d,b+d})^v_n}: (E^k_{-d, b+d})^v_{n} \to (E^k_{-d-1, b+d+1})^u_{n+1},\end{equation}
 which commutes with the differentials of the  spectral sequences. In particular, all the discussions below regarding the actions $\{Y_i\}_i$ from (\ref{eq:yi}) can be extended
to the level of their restrictions (\ref{eq:yi2}) indexed by the edges of $\mathfrak{R}$.
The corresponding details regarding this extensions  are left to the reader.

The morphisms $Y_i$, for $k=1$ and for any fixed $b$,  make
   $$\oplus_{n,d}\, \oplus_{l\,:\, |l|=d}\,  H_b(S_n\cap \frX_{-l}, S_n\cap \frX_{-l}\cap \frX_{-|l|-1}, \Z)
  =\oplus_{n,d}\,  (E^1_{-d, b+d})_{n}$$
   a $\Z[Y_1,\ldots, Y_r]$-- module.
    The operator $Y_i$ increases the $w$--weight (or $n$) by $+1$, increases the lattice
  filtration by $E_i$ (i.e., $l\rightsquigarrow l+E_i$), and preserves the homological degree $b$.
   This new structure completes considerably
   the rank discussions from the previous sections coded in
   $\bPE_1({\bf T}, Q,h)$ and $PE_k(T,Q,h)$.
   One might think that it is a weaker replacement for the missing $U$--action (killed
   in the relative setup), but, in fact,
   as we will see,  the $Y$--action on ${\rm Gr}^F_*\bH_*$ has even a more subtle structure than the
   $U$--action on $\bH_*$.

  Recall also that above, if  the summand  
  $ H_b(S_n\cap \frX_{-l}, S_n\cap \frX_{-l}\cap \frX_{-|l|-1}, \Z)$
  is non-zero,  then necessarily  $n=w(l)+b$.


%

Let us exemplify the operators $Y_i$
in some concrete situations.
\begin{example}\label{ex:irredY1}
Let $(C,o)$ be an {\bf irreducible curve singularity}   with semigroup $\cS$. We assume that
$1\not\in\cS$. We organize the gaps $\Z_{\geq 0} \setminus \cS$  as follows.
$$\cS=\{0=b_0, a_1,a_1+1,\ldots, b_1, a_2, a_2+1, \ldots, b_2, \ldots, a_k, a_k+1, \ldots, b_k, a_{k+1}, a_{k+2}, \ldots\},$$
where the gaps are exactly $\cup_{i=0}^k\{\mbox{integers strictly between $b_i$ and $a_{i+1}$}\}$ and
$a_{k+1}$ is the conductor.

Fix some $s\in\cS$. Note that $w(s+1)=w(s)+1$, hence $s$ determines a generator
$[s]$ in $H_0(S_{w(s)}\cap \frX_{-s}, S_{w(s)}\cap \frX_{-s-1},\Z)=\Z=\Z_{(s)}$.
This element corresponds to $T^sQ^{w(s)}$ in $PE_1(T,Q)$.
It has level degree $d=s$ and weight $n=w(s)$.
Then $Y([s]):=Y_1([s])=[s+1]$ if $s+1\in\cS$, and $=0$ otherwise.
In particular, the  $Y$--module
$\sum _{s\in \cS}H_0(S_{w(s)}\cap \frX_{-s}, S_{w(s)}\cap \frX_{-s-1},\Z)=\oplus _{s\in\cS}\Z_{(s)}$ has the
 following irreducible  summands:
the $Y$--modules of finite $\Z$--rank:
$$\Z_{(0)}, \  \ \Z_{(a_1)}\stackrel{Y}{\longrightarrow} \cdots  \stackrel{Y}{\longrightarrow}
\Z_{(b_1)}, \ \cdots,  \Z_{(a_k)}\stackrel{Y}{\longrightarrow} \cdots  \stackrel{Y}{\longrightarrow}
\Z_{(b_k)}$$
and one $Y$--module of infinite $\Z$--rank
$ \Z_{(a_{k+1})}\stackrel{Y}{\longrightarrow}  \Z_{(a_{k+2})}\stackrel{Y}{\longrightarrow}\cdots$.

Take for example the semogroups $\cS_1=\langle 4,5,7\rangle$ and $\cS_{2}=\langle 3,7,8\rangle$,
cf. Remark \ref{rem:3.5.1}. They have the very same $\Z[U]$--module
$\bH_0=\et^-_{4}\oplus \et_{2}(1)\oplus \et_{0}(1)$, and
$PE_k(T=1,Q)=Q^{-2}+2Q^{-1}+2Q^0+Q^1+\cdots$. The common $(-w)$--graded root is shown by the left graph below.

This, for both cases,  can be compared with $\oplus_d\,{\rm Gr}^F_{-d} \,\bH_0$ with Poincar\'e
series $PE^\infty(T,Q)=\sum_{s\in\cS}T^sQ^{w(s)}$. Of course, this object still keeps all the information about the semigroup.
 From this if we  delete the level degrees (or in $PE^\infty$ we put $T=1$)
we obtain that $\oplus_d\,{\rm Gr}^F_{-d} \,\bH_0$ endowed (only) with the $w$--weights coincides with the $w$--weighted
$\bH_0$ as graded $\Z$--modules (which in both cases of $\cS_1$ and $\cS_2$ are the same). Then we can compare $\bH_0$ and  $\oplus_d\,{\rm Gr}^F_{-d} \,\bH_0$, where
both are endowed with their $w$--weights. The point is that
for the two cases of $\cS_1$ and $\cS_2$, the  $Y_1$--module structures on  $\oplus_d\,{\rm Gr}^F_{-d} \,\bH_0$
(weighted by $n=w(d)$) distinguish the two cases (while the $U$--multiplication on $\bH_0$ not);
see the second and the third root below (where again the edges code the action).

\begin{picture}(320,120)(170,-20)

\put(177,50){\makebox(0,0){\footnotesize{$0$}}} \put(177,60){\makebox(0,0){\footnotesize{$1$}}}
\put(177,70){\makebox(0,0){\footnotesize{$2$}}}
\put(177,10){\makebox(0,0){\footnotesize{$-w$}}}
\dashline{1}(200,60)(240,60)
\dashline{1}(200,50)(240,50)
\dashline{1}(200,70)(240,70)
\put(220,10){\makebox(0,0){$\vdots$}} \put(220,30){\circle*{3}}
\put(220,40){\circle*{3}} \put(210,50){\circle*{3}}
\put(220,50){\circle*{3}}
\put(220,60){\circle*{3}} \put(220,20){\line(0,1){50}}
\put(210,50){\line(1,-1){10}}
\put(220,70){\circle*{3}}
\put(230,60){\circle*{3}}
\put(220,50){\line(1,1){10}}
 \put(220,-10){\makebox(0,0){\footnotesize{$U$--action}}}
 \put(320,-10){\makebox(0,0){\footnotesize{$Y$--action}}}
 \put(420,-10){\makebox(0,0){\footnotesize{$Y$--action}}}

\put(320,70){\circle*{3}}
\put(330,60){\circle*{3}}
\put(320,10){\makebox(0,0){$\vdots$}}
\put(320,30){\circle*{3}}
\put(320,40){\circle*{3}}
\put(310,50){\circle*{3}}
\put(320,50){\circle*{3}}
\put(320,60){\circle*{3}}
\put(320,70){\line(0,-1){8}}
\put(330,60){\line(-1,-1){10}}
\put(320,40){\line(0,-1){8}}
\put(320,30){\line(0,-1){8}}
\put(320,50){\line(0,-1){8}}

\put(420,70){\circle*{3}}
\put(430,60){\circle*{3}}
\put(420,10){\makebox(0,0){$\vdots$}}
\put(420,30){\circle*{3}}
\put(420,40){\circle*{3}}
\put(410,50){\circle*{3}}
\put(420,50){\circle*{3}}
\put(420,60){\circle*{3}}
\put(420,70){\line(0,-1){8}}
\put(420,60){\line(0,-1){8}}
\put(420,40){\line(0,-1){8}}
\put(420,30){\line(0,-1){8}}
\put(420,50){\line(0,-1){8}}

\put(320,85){\makebox(0,0){$\langle 4,5,7\rangle$}}
\put(420,85){\makebox(0,0){$\langle 3,7,8\rangle$}}
\put(500,85){\makebox(0,0){$\langle 4,5,7\rangle$}}
\put(560,85){\makebox(0,0){$\langle 3,7,8\rangle$}}

\put(500,10){\makebox(0,0){$\vdots$}}
\put(500,30){\makebox(0,0){\footnotesize{10}}}
\put(500,40){\makebox(0,0){\footnotesize{9}}}
\put(490,50){\makebox(0,0){\footnotesize{0}}}
\put(500,50){\makebox(0,0){\footnotesize{8}}}
\put(500,60){\makebox(0,0){\footnotesize{5}}}
\put(500,70){\makebox(0,0){\footnotesize{4}}}
\put(510,60){\makebox(0,0){\footnotesize{7}}}

\put(560,10){\makebox(0,0){$\vdots$}}
\put(560,30){\makebox(0,0){\footnotesize{10}}}
\put(560,40){\makebox(0,0){\footnotesize{9}}}
\put(550,50){\makebox(0,0){\footnotesize{0}}}
\put(560,50){\makebox(0,0){\footnotesize{8}}}
\put(560,60){\makebox(0,0){\footnotesize{7}}}
\put(560,70){\makebox(0,0){\footnotesize{6}}}
\put(570,60){\makebox(0,0){\footnotesize{3}}}

\end{picture}

The diagrams on the right show that different level--degrees corresponding to the $w$--weights of the
generators of the root. They are the semigroup elements. Note that $Y$ acts via the following pattern:
 $Y(1_u)=1_v$ if and only if
  $u,v\in\calS$ satisfy $v=u+1$. Hence the $Y$--action still keeps considerably information about $\cS$.

If $(C,o)$ is an irreducible {\it plane}  curve singularity, then
by the formula $\sum_{s\in\cS}t^s=\Delta(t)/(1-t)$ (cf. Theorem
 \ref{Poincare vs Alexander}), the Alexander polynomial transforms into
 $$\Delta(t)=1-t+t^{a_1}-t^{b_1+1}+t^{a_2}-t^{b_2+1}+\cdots t^{a_k}-t^{b_k+1}+t^{a_{k+1}}.$$
Then the  above $Y$ action can be read from this shape
of $\Delta(t)$ too (compatibly to  the staircase diagrams of $(C,o)$
in the language of ${\rm HFL}^-$,  see e.g. \cite{Kr}).
\end{example}

\begin{example}\label{ex:u1}
{\bf Assume that $(C,o)=\{x^2+y^2=0\}$}. Then $r=2$
and $\bPE_1(T_1,T_2, Q,h)= 1+\frac{T_1T_2(1+Qh)}{(1-T_1Q)(1-T_2Q)}$, cf. \ref{ex:22}.
The $Y_1$--action on $\oplus_{n,d}\, \oplus_{l\,:\, |l|=d}\, (E^1_{-d, b+d})_{n}$ is the following.
Corresponding to $(0,0)\in\cS$, or to the monomial $1$ of $\bPE_1$, the actions of
$Y_1$ and $Y_2$ are trivial. Hence, it forms an irreducible $\Z[Y_1,Y_2]$--module with
 $\Z$--rank one. Let us denote it by $\Z_{w=0}$.
There are two other irreducible  $\Z[Y_1,Y_2]$ modules, both of them generated at the semigroup
entry $(1,1)$. One of them corresponds to $b=0$, or to the monomials
$\{T_1^{a_1}T_2^{a_2}Q^{a_1+a_2-2}h^0\}_{(a_1,a_2)\geq (1,1)}$,
with actions $Y_1(T_1^{a_1}T_2^{a_2}Q^{a_1+a_2-2})=T_1^{a_1+1}T_2^{a_2}Q^{a_1+a_2-1}$,
$Y_2(T_1^{a_1}T_2^{a_2}Q^{a_1+a_2-2})=T_1^{a_1}T_2^{a_2+1}Q^{a_1+a_2-1}$.
It is generated over $\Z[Y_1,Y_2]$ by $T_1T_2Q^0h^0$,
let us  denote it by
$\Z[Y_1,Y_2]_{w=0}$.
The other, corresponding to $b=1$ (or to $h^1$) is given by
$Y_1(T_1^{a_1}T_2^{a_2}Q^{a_1+a_2-1}h)=T_1^{a_1+1}T_2^{a_2}Q^{a_1+a_2}h$,
$Y_2(T_1^{a_1}T_2^{a_2}Q^{a_1+a_2-1}h)=T_1^{a_1}T_2^{a_2+1}Q^{a_1+a_2}h$. It is generated by
$T_1T_2QH$, let us denote it by $\Z[Y_1,Y_2]_{w=1}$.
These two modules are isomorphic with $\Z[Y_1,Y_2]$ with the corresponding degree shifts.

 Pictorially, denoted by `big bullets' ($b=0$) and `circles' $(b=1)$,  they are:

 
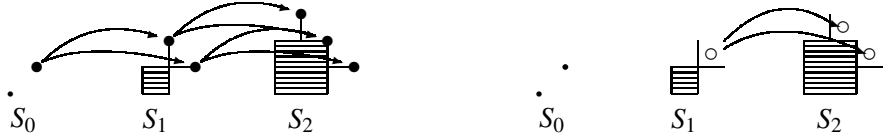
\begin{picture}(300,70)(-50,-10)

\put(10,10){\circle*{2}}\put(20,20){\circle*{4}}
\put(15,0){\makebox(0,0){$S_0$}}

\qbezier(22,22)(40,30)(75,22)
\put(70,23){\vector(4,-1){6}}

\qbezier(22,22)(40,40)(65,32)
\put(60,33){\vector(4,-1){6}}

\qbezier(82,22)(100,30)(135,22)
\put(130,23){\vector(4,-1){6}}
\qbezier(82,22)(100,40)(125,32)
\put(120,33){\vector(4,-1){6}}

\qbezier(72,32)(90,40)(125,32)
\put(120,33){\vector(4,-1){6}}
\qbezier(72,32)(90,50)(115,42)
\put(110,43){\vector(4,-1){6}}

\qbezier(280,27)(300,40)(330,27)
\put(328,27){\vector(3,-1){6}}
\qbezier(280,30)(300,50)(320,37)
\put(318,37){\vector(3,-1){6}}

\put(60,10){\line(1,0){10}}\put(60,10){\line(0,1){10}}
\put(70,10){\line(0,1){20}}\put(60,20){\line(1,0){20}}
\put(60,12){\line(1,0){10}}\put(60,14){\line(1,0){10}}
\put(60,16){\line(1,0){10}}\put(60,18){\line(1,0){10}}
\put(65,0){\makebox(0,0){$S_1$}}
\put(70,30){\circle*{4}}\put(80,20){\circle*{4}}

\put(110,10){\line(1,0){20}}\put(110,10){\line(0,1){20}}
\put(130,10){\line(0,1){20}}\put(110,30){\line(1,0){20}}
\put(120,30){\line(0,1){10}}\put(130,20){\line(1,0){10}}
\put(110,12){\line(1,0){20}}\put(110,14){\line(1,0){20}}\put(110,16){\line(1,0){20}}
\put(110,18){\line(1,0){20}}\put(110,20){\line(1,0){20}}\put(110,22){\line(1,0){20}}\put(110,24){\line(1,0){20}}
\put(110,26){\line(1,0){20}}\put(110,28){\line(1,0){20}}
\put(120,0){\makebox(0,0){$S_2$}}
\put(120,40){\circle*{4}}\put(130,30){\circle*{4}}\put(140,20){\circle*{4}}

\put(210,10){\circle*{2}}\put(220,20){\circle*{2}}
\put(215,0){\makebox(0,0){$S_0$}}

\put(260,10){\line(1,0){10}}\put(260,10){\line(0,1){10}}
\put(270,10){\line(0,1){20}}\put(260,20){\line(1,0){20}}
\put(260,12){\line(1,0){10}}\put(260,14){\line(1,0){10}}\put(260,16){\line(1,0){10}}
\put(260,18){\line(1,0){10}}
\put(265,0){\makebox(0,0){$S_1$}}
\put(220,20){\circle*{2}}
\put(275,25){\circle{4}}

\put(310,10){\line(1,0){20}}\put(310,10){\line(0,1){20}}
\put(330,10){\line(0,1){20}}\put(310,30){\line(1,0){20}}
\put(320,30){\line(0,1){10}}\put(330,20){\line(1,0){10}}
\put(310,12){\line(1,0){20}}\put(310,14){\line(1,0){20}}\put(310,16){\line(1,0){20}}
\put(310,18){\line(1,0){20}}\put(310,20){\line(1,0){20}}\put(310,22){\line(1,0){20}}
\put(310,24){\line(1,0){20}}
\put(310,26){\line(1,0){20}}\put(310,28){\line(1,0){20}}
\put(320,0){\makebox(0,0){$S_2$}}
\put(325,35){\circle{4}}\put(335,25){\circle{4}}

\end{picture}

Next, we can analyse the page $E^2$ too. The differential $d^1$ on $\Z_{w=0}$ is trivial, hence this
term survives in $E^2$ too. The other two modules are connected by a nontrivial graded
differetial $\oplus_n d^1_{*,*}(S_n)$, which is
a  $w$--homogeneous morphism of modules $\Z[Y_1,Y_2]_{w=1}\to \Z[Y_1,Y_2]_{w=0}$, provided by multiplication by $Y_1-Y_2$.
Hence the $E^2 $ terms, as a $\Z[Y_1,Y_2]$--module has two irreducible submodules, one of them is a
$\Z_{w=0}$ of $\Z$-rank one and it has a trivial $Y_1,Y_2$--action, while the other one is
$\Z[Y]$, where both $Y_1$ and $Y_2$ act as multiplication by $Y$, and its generator sit at
the semigroup element $(1,1)$, its weight $w=0$ and $h=0$. Its Poincar\'e series is
$T^2+T^3Q+T^4Q^2+\cdots= PE^2(T,Q)-1$.

The $\Z[Y_1,Y_2]$ module  $ E^2$ (and any $E^k$ for $k\geq 2$), when we keep only the $w$--weight and the $Y_1$ action,
 pictorially is the following

\begin{picture}(320,60)(150,0)

\put(320,5){\makebox(0,0){$\vdots$}}
\put(320,30){\circle*{3}}
\put(320,40){\circle*{3}}
\put(320,20){\circle*{3}}
\put(330,50){\circle*{3}}
\put(310,50){\circle*{3}}
\put(330,50){\line(-1,-1){10}}
\put(320,40){\line(0,-1){8}}
\put(320,30){\line(0,-1){8}}
\put(320,20){\line(0,-1){8}}
\put(400,30){\makebox(0,0){$(\mbox{$Y_1$ and $Y_2$ act identically})$}}
\end{picture}
\end{example}

\begin{example}\label{ex:u1u4}
{\bf Assume that $(C,o)$ is an ordinary $r$-tuple } (cf. \ref{bek:ANcurves}).
For the invariants $\bPE_1({\bf T}, Q,h)$, $PE_k(T,Q,h)$, see \ref{ex:gen} (with $c=(1,1,\ldots, 1)$).
Then based on the discussion from \ref{ex:gen} (or even by construction of the spaces $S_n$)
 one deduce that the $\Z[Y_1,\ldots, Y_r]$ structure on
the first page $E^1$ is
$$\Z_{w=0}\oplus \, \oplus _{b=0}^{r-1} \ \Z[Y_1, \ldots, Y_r]_{(b)}^{\oplus \binom{r-1}{b}},$$
where $\Z_{w=0}$ has $\Z$--rank one, the generator corresponds to the lattice point $l=0$, $w(0)=0$ and $b=0$,
while  $\Z[Y_1, \ldots, Y_r]_{(b)}$ is generated at the lattice point $l=(1,\ldots, 1)$, $n=w(l)+b=2-r+b$ (i.e. it
corresponds to the monomial ${\bf T}^{(1, \ldots ,1 )}Q^{2-r+b }h^b$ and each $Y_i$ is multiplication by $T_iQ$).
\end{example}

The collapse of the $\Z[Y_1,\ldots, Y_r]$--action on $E^{\geq 2}$ in Example \ref{ex:u1}
 is not an accident, it is a genaral fact valid for any $(C,o)$.

\begin{proposition}\label{prop:collapse} Fix an isolated singularity $(C,o)$.
For any $k\geq 2$ the action of $Y_i$ on $\oplus_{n,d} (E^k_{*,*})_n$ is independent of\, $i$.
In particular, $\oplus_{n,d} (E^k_{*,*})_n$ admits a natural $\Z[Y]$--module structure.

${\rm Gr}\,\bH_*:=\oplus_{n,d} (E^\infty_{*,*})_n$ and $\bH_*$ both  considered as a graded $\Z$--modules, graded
by the $w$--weights (equivalently, by the summands associated with each $S_n$) are isomorphic as
graded $\Z$--modules. However, ${\rm Gr}\,\bH_*$ considered as a $Y$--module
 and $\bH_*$  considered as a $U$--module  usually  do not agree.
\end{proposition}
\begin{proof}
Let $\alpha$ be a chain in $S_n\cap \frX_{-d}$. For any $t\in[0,1]$ set $y_{i,t}:\frX\to \frX$  given by $x\mapsto x+tE_i$.
Denote  $\cup_{t\in[0,1]} y_{i,t}(\alpha)$ by $\beta_i$. Then $\beta_i-\beta_j\in S_{n+1}\cap \frX_{-d}$ and
$\partial (\beta_i-\beta_j)=y_i(\alpha)-y_j(\alpha)$. Then using the general construction of spectral sequences,
$y_i(\alpha)-y_j(\alpha)\in (B^2_{*,*})_{n+1}$, hence its class $[y_i(\alpha)-y_j(\alpha)]\in (E^2_{*,*})_{n+1}
=(Z^2_{*,*})_{n+1}/(B^2_{*,*})_{n+1}$ is zero. (Cf. \cite[page 131]{SpecSeq}.)
\end{proof}

\begin{remark}
There is another aspect and difference in the comparison of
${\rm Gr}\,\bH_*$ considered as a $Y$--module
 and $\bH_*$  considered as a $U$--module.
Assume that $(C,o)$ is Gorenstein. Then $\bH_*(C,o)$ and the $U$--action respects the Gorenstein $\Z_2$--symmetry
$l\mapsto c-l$, however, the $Y$--action on ${\rm Gr}\,\bH_*$ does not. (See e.g. the case of an irreducible plane curve singularity, or Example \ref{ex:u1}.)

\end{remark}
\begin{example}\label{ex:u1u3}
{\bf Assume that  $(C,o)$ is the plane curve singularity
$\{x^3+y^3=0\}$}. (This is a continuation of Example \ref{ss:33}.)
For each $b\geq 0$ the $\Z[Y_1,Y_2, Y_3]$--module $\oplus_{n,d}(E^1_{-d, b+d})_n$
decomposes into irreducible summands. They are the following.

\underline{Case $b=0$:} \

$M_{b=0}^1=\Z$ of $\Z$--rank one generated by ${\bf T}^0Q^0h^0$;

$M_{b=0}^2\simeq \Z[Y_1,Y_2,Y_3]/ (Y_1Y_2,Y_2Y_3,Y_3Y_1)$ generated by $T_1T_2T_3Q^{-1}h^0$;

$M_{b=0}^3\simeq \Z[Y_1,Y_2,Y_3]$ generated by $T_1^2T_2^2T_3^2Q^0h^0$.

\underline{Case $b=1$:} \

$M_{b=1}^2$ is supported by the semigroup elements
$(1,1,1)\cup\cup_{k\geq 2}(k,1,1)\cup \cup_{k\geq 2}(1,k,1)\cup \cup_{k\geq 2}(1,1,k)$, similarly as
$M_{b=0}^2$, but in this case the $(1,1,1)$--homogeneous part has $\Z$--rank 2.
All other homogeneous $\Z$--summands have rank one.
So, it has two generators both coded by $T_1T_2T_3Q^0h$.

$M_{b=1}^{3,3'}$, two copies of $\Z[Y_1,Y_2,Y_3]$, both generated  at $T_1^2T_2^2T_3^2Qh$.

\underline{Case $b=2$:} \

$M_{b=2}^{3}\simeq \Z[Y_1,Y_2,Y_3]$ generated  at $T_1^2T_2^2T_3^2Q^2h^2$.

\vspace{2mm}

This shows that we might have irreducible $\Z[Y_1,Y_2,Z_3]$--modules with homogeneous summand associated with certain $l\in\cS$
of $\Z$--rank $\geq 2$. (This never happens if $r=1$.)

For $k=\infty$, the $Z[Y]$--modules  $\oplus_{n,d}(E^1_{-d, b+d})_n$ for $x^3+y^3$ and $x^3+y^4$ agree (for the last one see
Example \ref{ex:irredY1}.)

The above picture coincides  with the description of HFL$^-$ of the torus knot $T_{3,3}$  in \cite{GH}.
\end{example}

\begin{example}\label{ex:u1u2}
The reader is invited to describe the $\Z[Y_1,Y_2]$--modules in the case of Example \ref{ex:decsing2}  (case
$\langle 3,4\rangle \vee \langle 3,4\rangle $).
Here appear modules of type $\Z[Y_1,Y_2]/(Y_1^2,Y_2^2)$ and $\Z[Y_1,Y_2]/(Y_1^2)$ as well.
\end{example}

\begin{remark}
There is a $\Z[Y_1,\ldots, Y_r]$--action on the relative homologies as well. Indeed, the morphism
$$Y_i:H_*(S_n\cap  \frX_{-l}, S_{n-1}\cap \frX_{-l})\to
H_*(S_{n+1}\cap  \frX_{-l-E_i}, S_{n}\cap \frX_{-l-E_i})$$
induced by $x\mapsto x+E_i$ is well defined and usually nontrivial (and distinct for diffenet $i$'s).
\end{remark}

\subsection{Deformations and functors}

Theorem \ref{th:DEF} has the following consequence.
\begin{theorem}\label{th:DEF2}
    Consider  a flat deformation of isolated  curve singularities ${(C_t,o)}_{t\in(\bC,0)}$.

   Assume that  either (a) $(C_{t=0},o)$ irreducible, or (b)  ${(C_t,o)}_{t\in(\bC,0)}$ is a delta-constant deformation
 of   plane curve singularities such that
   the number of irreducible components stays  stable.

 Then the   induced  morphism $ \bH_*(C_{t\not=0},o)\to \bH_*(C_{t=0},o)$  is compatible with
   all the filtrations and degrees (weight, level and homological).
   That is, they induces morphisms at the level of $(E^k_{-d,q})_n$, preserving  all the degrees.
   Moreover, these morphisms are also   compatible with
   the graded graph-map at the level of graded roots ${\mathfrak R}(C_{t\not=0},o)\to{\mathfrak R}(C_{t=0},o)$ and also with the $Y_i$ actions.

\end{theorem}

\section{Plane curves. Relation with Heegaard Floer Link homology}\label{s:HFL}

%

\subsection{Review of Heegaard Floer Link homology}\label{ss:HFL}
Assume that $(C,o)$ is a  plane curve singularity, $(C,o)\subset (\bC^2,0)$.
It defines the link $L(C,o)=L\subset S^3$  with link-components $L_1,\ldots, L_r$.
In this section we relate the above considered filtered lattice homology with
the Heegaard Floer link homology of $L$. For more on Heegaard Floer Link theory  see e.g.
 \cite{MO,os,os2,OSzHol,os4,ras}. In the next paragraphs we recall in short some notations and facts.

To every 3-manifold $M$ with fixed Heegaard splitting (and extra structures) one can associate a {\em Heegaard Floer complex} $CF^{-}(M)$
of free $\Z[U]$-modules. The operator $U$ has homological degree $(-2)$, and the differential $d$ has degree $(-1)$. This complex is not unique, but different choices
(e.g. the splitting)
lead to quasi-isomorphic complexes. Therefore the homology of $CF^{-}(M)$ is an invariant of $M$ called
 {\em Heegaard Floer homology} and denoted by $HF^{-}(M)$. In this note we will have $M=S^3$;  in this case  $HF^{-}(S^3)=\Z[U]$.

To a link $L=L_1\cup\ldots\cup L_{r}\subset S^3$ one can associate
a $\Z^r$--filtered complex of
$\Z[U_1,\ldots,U_r]$-modules, denoted by $CFL^{-}(L)$.
 The  $\Z^r$--filtration is called the Alexander filtration.
 The operators $U_i$
 have homological degree $(-2)$ and shift the filtration level by $E_i$.
 If one ignores the filtration, then the complex is quasi-isomorphic
to the Heegaard Floer complex  $CF^{-}(S^3)$, where all the operators $U_i$ are homotopic to each other,
cf.  \cite{OSzHol}. One can
consider the complex also as a $\Z[U]$-module, where $U=U_1$.
However, the filtration captures nontrivial information about the link.

For $l\in \Z^r$, we will denote the Alexander filtration by $\{A^-(l)\}_l$. Each  $\Cc(l)=(\oplus_\nu A^{-,\nu}(l),d)$
is a subcomplex of $CFL^{-}(L)$ at filtration level $v$ (in \cite{MO}
 these complexes are denoted  by $\mathfrak{A}^{-}(v)$).
  It is spanned by the elements of $CFL^{-}(L)$ with Alexander filtration
 greater than or equal to $l$. (For a more clear match with the algebraic picture, we reverse the sign of $l$, thus reversing the direction of the filtration as well.) The upper index $\nu$ denotes the homological (Maslow) grading.   They 
satisfy
\begin{equation}\label{eq:INCL}\begin{array}{l}
\Cc(l_1)\supset  \Cc(l_2) \mbox{  \ for    $l_2\geq l_1$, \ and }\\
\Cc(l_1)\cap \Cc(l_2)= \Cc(\max\{l_1,l_2\}).
                               \end{array}
\end{equation}
The subcomplexes $\Cc(v)$ are $\Z[U_1,\ldots,U_r]$-submodules,  the induced  operators $U_i$
have homological degree $-2$ and are
homotopic to each other again. Moreover $U_i(A^-(l))\subset A^-(l+E_i)$.
The Heegaard-Floer link homology is defined as the homology of the associated graded pieces of $\Cc(l)$:
$$
\HFL^{-}(L,l):=H_{*}(\, (\gr\Cc)(l)\,), \ \ \mbox{where} \ (\gr\Cc)(l):=
\Cc(l)/\sum_{l'\geq  l}\Cc(l').
$$
\begin{remark}
At present, Heegaard Floer link homology is defined only for $\BF_2$ coefficients, hence, strictly speaking, all results of this section are valid only over $\BF_2$. Nevertheless, we believe that all the statements
are true over $\Z$ as well, but the cautious reader might take everywhere $\BF_2$ instead of $\Z$.
\end{remark}



\subsection{The connection with the local lattice cohomology}\label{ss:HFLloc}

In \cite[Theorem 6.1.3]{GorNem2015}  the following isomorphism  was proved.

\begin{theorem}\label{th:isoGor}
For any $l\in \Z^r$ let ${\rm HL}^-_*$ denote the local lattice cohomology
$H_*({\rm gr}_l\calL^-, {\rm gr}_l \partial _U)$, graded by the homological degree, cf. \ref{bek:2.22}.
Let ${\rm HFL}_*(L,l)$ be the Heegaard Floer link homology, graded by its homological (Maslow) grading.
Then $\HFL^-_*(L,l)$ and $ {\rm HL}^-_*(l)$ are isomorphic as graded\, $\Z$-modules.

In particular, $\HFL^-_*(L,l)$ has no $\Z$--torsion and $\HFL^-_*(L, l)=0$ whenever $l\not\in \calS$.
\end{theorem}

Theorem \ref{th:PP} has the following consequences.

Let us  fix $l\in\calS$ and assume that $\HFL^-_k(L,l)\not=0$. This means  ${\rm HL}^-_k(l)\not=0$. For the convenience of the reader let us write down again  the possible
bidegrees $(k-b,b)$ which might appear. Using parts {\it (a)-(b)} of Theorem \ref{th:PP}, in the isomorphism (\ref{eq:PP}) we have
\begin{equation}
        k  =-2n+2\hh(l)-2|l|+b\ \ \ \ \mbox{and} \ \ \
          w(l) =n-b.
\end{equation}
These identities identify both $n$ and $b$   in terms of $l$ and $k$:
\begin{equation}\label{eq:nb}
        n  =-|l|-k\ \ \ \ \mbox{and} \ \ \
        b  = -|l|-k-w(l)=-k-2\hh(l).
\end{equation}
Hence, with fixed $k$, there is only one bidegree $(k-b,b)$ for which ${\rm HL}^-_{k-b,b}(l)\not=0$ given by the second identity of (\ref{eq:nb}), namely  $b=-k-2\hh(l)\geq 0$. Moreover,  (\ref{eq:PP}) and (\ref{eq:HLFUJ}) transforms into the following isomorphism
\begin{corollary}
\begin{equation}\label{eq:lb1}
\HFL^-_{-2\hh(l)-b}(L,l)\simeq H_b(S_{w(l)+b}\cap\frX_{-l}, S_{w(l)+b}\cap\frX_{-l}\cap \frX_{-|l|-1}).
\end{equation}
In particular, the spaces $\{S_n\}_n$ (as cube--subcomplexes of  $\R_{\geq 0}^r$)
provide by their relative homologies all the Heegaard Floer link homologies. \end{corollary}

E.g., in the case of Example \ref{ss:33}, let us fix the lattice point $l=(2,2,2)$. Then $|l|=6$, $w(l)=0$ and $-2\hh(l)=-6$.
The coefficient of ${\bf T}^l=T_1^2T_2^2T_3^2$ in ${\bf PE}_1({\bf T}, Q,h)$ is  $(1+Qh)^2
=Q^0h^0+2Q^1h^1+Q^2H^2$, cf. (\ref{eq:TTT}).
This means that at the lattice point $l=(2,2,2)$, via the relative homology `at $l$',
from $S_0$ (i.e. from the exponents and coefficient of  $Q^0h^0$) we read that ${\rm rank}\, \HFL^-_{-6}(L,l)=1$. Next, from $S_1$ (i.e. from $2QH$) we read that
${\rm rank}\, \HFL^-_{-6-1}(L,l)=2$, finally from $S_2$ that ${\rm rank}\, \HFL^-_{-6-2}(L,l)=1$. All other homologies
$\HFL^-_{-6-b}(L,l)$, $b\not=0,1,2$,
are zero. (It is instructive  to compare the statements with the pictures of the spaces $\{S_n\}_n$ as well.)

Next, using the Poincar\'e polynomial identity (\ref{hfminus}) together with Theorem \ref{th:PP} we obtain

\begin{corollary}
\begin{equation*} \begin{split}
P^m(\bt;q)&=\sum_l\sum _k (-1)^k \cdot
  {\rm rank} \HFL^-_{-2\hh(l)-k}(L,l)\cdot\bt^l\, q^{\hh(l)+k},\\
{\bf PE}_1({\bf T}, Q, h)&=\sum_l \
\sum_k\  {\rm rank} \HFL^-_{-2\hh(l)-k}(L,l)\cdot {\bf T}^lQ^{w(l)+k}h^{k}\\ &=
\sum_l \
\sum_n\  {\rm rank} \HFL^-_{-n-|l|}(L,l)\cdot {\bf T}^lQ^nh^{n-w(l)},\\
 {\bf PE}_1({\bf T}, Q=1, h=1)&=\sum_l \sum_k (-1)^k\pp^m _{l,k}\cdot {\bf T}^l=\sum_l \
\big(\sum_k\  {\rm rank} \HFL^-_k(L,l)\,\big)\cdot {\bf T}^l,\\
 {\bf PE}_1({\bf T}, Q=1, h=-1)&=\sum_l \sum_k \pp^m _{l,k}\cdot {\bf T}^l=\sum_l \
\chi\big(\HFL^-_*(L,l)\,\big)\cdot {\bf T}^l= P({\bf T}).
\end{split}\end{equation*}
\end{corollary}
The last identity is equivalent via  Theorem \ref{Poincare vs Alexander} by a result of Ozsv\'ath and Szab\'o from
\cite[Proposition 9.2]{OSzHol}, which determines
 the generating function $\sum_{l}\
\chi(\HFL^-(L,l))\cdot \bt^l$
of the  Euler characteristic of the Heegaard Floer link homology
as
$\Delta(\bt)$ if  $r>1$, and
$\Delta(t)/1-t$ if  $r=1$.

Next, we compare $\{\HFL^-_*(L,l)\}_l$ with the spectral sequences.

First of all, we wish to emphasize that the spectral sequence of this note is {\it not}
the spectral sequence constructed in the $HFL^-$--theory. In that theory, the $\infty$--page is the graded
$HF^-(S^3)=\Z[U]$, while in our case the $\infty$--page  is the graded $\bH_*(C,o)$.

Let us write (\ref{eq:lb1}) in the form
\begin{equation}\label{eq:lb2}
\HFL^-_{-n-|l|}(L,l)\simeq H_b(S_n\cap\frX_{-l}, S_n\cap\frX_{-l}\cap \frX_{-|l|-1}),
\end{equation}
where $b=n-w(l)$. For any fixed $d$, summation over $\{l\,:\, |l|=d\}$ gives for any $n,\, b$ and $d$:
\begin{equation}\label{eq:lb3}
\bigoplus_{|l|=d,\ w(l)=n-b}
\HFL^-_{-n-|l|}(L,l)\simeq \bigoplus_{|l|=d}\,
H_b(S_n\cap\frX_{-l}, S_n\cap\frX_{-l}\cap \frX_{-|l|-1})= (E^1_{-d, b+d})_n.
\end{equation}
In particular, for any fixed $n$, the spectral sequence associated with $S_n$
uses and capture only the specially chosen summand $\oplus_{l}{\rm HFL}^-_{-n-|l|}(L,l)$ of
$\oplus _l{\rm HFL^-}_*(L,l)$. This is a partition of $\oplus _l{\rm HFL^-}_*(L,l)$ indexed by $n$.
Each $\oplus_{l}{\rm HFL}^-_{-n-|l|}(L,l)$, interpreted as an $E^1$--term, converges to
$({\rm Gr}^F_*\bH_*)_{-2n}$. (This is not the spectral sequence of the Link Heegaard Floer theory,
which converges to $HF^-(S^3)$, though the entries --- but not
the differentials --- of the first page can be identified.)

This partition of $\{\oplus_{l}{\rm HFL}^-_{-n-|l|}(L,l)\}_{n\geq m_w}$ of
$\oplus _l{\rm HFL^-}_*(L,l)$ can be refined by the vertices of the  graded root.
Indeed, if we replace in (\ref{eq:lb2}) $S_n$ by $\sqcup_v S_n^v$,  then each
$\oplus_{l}{\rm HFL}^-_{-n-|l|}(L,l)$ for fixed $n\geq m_w$ decomposes into a direct sum
\begin{equation}\label{eq:v}
\oplus _{v\in\cV(\mathfrak{R}): w_0(v)=n}\ {\rm HFL}^-_{-n-|l|}(L,l)^v,\end{equation}
providing a direct sum decomposition  of the link Heegaard Floer homology $\oplus _l{\rm HFL^-}_*(L,l)$
indexed by $\cV(\mathfrak{R})$.
(The author does not know whether  this direct sum decomposition can be realized via the HFL theory.)

Let us  make in (\ref{eq:lb3}) a summation over $d$. For fixed $n$ and $b$ we obtain
\begin{equation}\label{eq:lb4}
\bigoplus_{l\,:\, \, w(l)=n-b}
\HFL^-_{-n-|l|}(L,l)\simeq \bigoplus_{d}\,
(E^1_{-d, b+d})_n.
\end{equation}
Since $\rank (E^1_{-d, b+d})_n\geq \rank (E^\infty_{-d, b+d})_n$, we get the following lower bound
for the Heegaard Floer link homology in terms of the lattice homology of $(C,o)$:
\begin{corollary}
(a) For any fixed $d$, $n$ and $b$:
\begin{equation}\label{eq:lb3b}
\sum_{|l|=d,\ w(l)=n-b}
\rank\ \HFL^-_{-n-|l|}(L,l)\geq {\rm rank}\ {\rm Gr}^F_{-d}\,(\bH_b(C,o))_{-2n}.
\end{equation}

(b) For any fixed  $n$ and $b$:
\begin{equation}\label{eq:lb4b}
\sum_{l\, :\, \,  w(l)=n-b}
\rank\ \HFL^-_{-n-|l|}(L,l)=
\sum_{l\, :\, \,  w(l)=n-b}
\rank\ \HFL^-_{-2\hh(l)-b}(L,l)
\geq \rank\,  (\bH_b(C,o))_{-2n}.
\end{equation}
(In (\ref{eq:lb4b}) the left hand side also equals
  the coefficient of $Q^nh^b$ in $PE_1(T=1, Q,h)$.)

\end{corollary}

\begin{example}
(\ref{eq:lb4b}) for $b=0$ reads as
\begin{equation}\label{eq:lb5}
\sum_{l\, :\, \,  w(l)=n}
\rank\ \HFL^-_{-2\hh(l)}(L,l)
\geq \rank\,  (\bH_0(C,o))_{-2n}.
\end{equation}
Note that for  $n\geq 0$ we have $\rank\,
(\bH_0(C,o))_{-2n}\geq 1$ (see e.g.  Proposition \ref{prop:infty}).

Assume that $\bH^{\geq 1}(C,o)=0$. Then (\ref{eq:lb5}) together with  Proposition \ref{prop:infty} {\it (d)} gives
$$\sum_{n<0}\, \sum_{l\,:\, w(l)=n}
\rank\ \HFL^-_{-2\hh(l)}(L,l)\, +\,
\sum_{n\geq 0}\, \sum_{l\,:\, w(l)=n}
\big(\rank\ \HFL^-_{-2\hh(l)}(L,l)-1\big)\, \geq \delta(C,o).$$
\end{example}

\section{Other level filtrations}

In section \ref{s:levfiltr} the filtration of $S_n$ was induced by the filtration $\{\frX_{-d}\}_d$ of $\frX$, where
$\frX_{-d}:=\{\cup(l,I)\,:\, |l|\geq d\}$.
However, there are infinitely many similar level filtrations which might serve equally well and
 can be considered and studied.
 Indeed, fix e.g. an integral nonzero vector $a=(a_1, \ldots, a_r)\in\Z^r$, and set $\frX^{(a)}_{-d}:=
\{\cup(l,I)\,:\, \sum_ia_il_i\geq d\}$. Then  for each fixed $n$, $\{S_n\cap \frX^{(a)}_{-d}\}_d$ is an increasing finite filtration of
$S_n$. In particular (by taking the very same weight function $w$)
 it induces a homological spectral sequence converging to the lattice homology.
That is,  the $\infty$--page is the graded ${\rm Gr}^F_*\bH_*$, where both $F_{*}\bH_*(\frX,w)$ and
 ${\rm Gr}^F_*\bH_*$  depend on the choice of $a$.
 The first pages can also be identified with certain local lattice homologies.

 For example, assume that $a\in (\Z_{>0})^r$.  Then, for each fixed $n$, the first pages
 $ (E^1_{-d,q})_n^{(a)}= H_{-d+q}(S_n\cap \frX^{(a)}_{-d}, S_n\cap \frX^{(a)}_{-d-1},\Z)$
 has a direct sum decomposition
\begin{equation*}
 (E^1_{-d,q})_{n}^{(a)}= \bigoplus_{l\in\Z^r_{\geq 0},\, \sum_ia_il_i=d}\
 (E^1_{-l,q})_{n}^{(a)}, \ \mbox{where} \ (E^1_{-l,q})_{n}^{(a)}:=
 H_{-d+q}(S_n\cap \frX_{-l}, S_n\cap \frX_{-l}\cap \frX^{(a)}_{-d-1},\Z)\end{equation*}
Note that for any fixed $l\in(\Z_{\geq 0})^r$ the modules $(E^1_{-l,q})_{n}^{(a)}$ are $a$--independent
(i.e. they are the terms of the  local lattice homologies considered in the previous sections),
however in  $(E^1_{-d,q})_{n}^{(a)}$ we specially choose the summation set according to the hyperplane equation
$ \sum_ia_il_i=d$.

In particular, for a plane curve singularity, for every $a\in (\Z_{>0})^r$ we get a spectral sequence whose first pages
consists of the `reorganized  packages' of HFL$^-_*$, and it converges to $\bH_*(C,o)$.

The construction can   be compared with the definition of the homologies tHFK (and upsilon invariant) from
\cite{OSZStu}.

\end{document}